\documentclass[12pt]{article}

\usepackage[utf8]{inputenc}
\usepackage[T1]{fontenc}
\usepackage{lmodern}

\usepackage[a4paper, margin=2.7cm]{geometry}

\usepackage{amsmath, amsthm, amsfonts, amssymb}
\usepackage{mathrsfs}           % \mathscr font.

\usepackage{bbm} 

\usepackage{multicol}
\usepackage{enumitem}

\usepackage{cancel}
\usepackage{tikz-cd}
\usepackage{caption}
\usepackage{graphicx}
\usepackage{subfig}
\usepackage[strings]{underscore}
\usepackage{float}
\usepackage[misc]{ifsym}
\usepackage{bm}
\usepackage{verbatim}
\newtheorem{theorem}{Theorem}[section]
\newtheorem{corollary}[theorem]{Corollary}
\newtheorem{lemma}[theorem]{Lemma}
\newtheorem{proposition}[theorem]{Proposition}

\theoremstyle{definition}
\newtheorem{definition}[theorem]{Definition}
\theoremstyle{remark}
\newtheorem{remark}[theorem]{Remark}
\newtheorem{obs}[theorem]{Observation}
\newtheorem{example}[theorem]{Example}
\usepackage{authblk}

\providecommand{\keywords}[1]{\textbf{\textit{Index terms---}} #1}

\font\myfont=cmr12 at 18pt

\begin{document}

\date{}

\title{{\myfont Contrapositionally Complemented Pseudo-Boolean Algebras and Intuitionistic Logic with Minimal Negation}}

\author[1]{Anuj Kumar More\thanks{The work has been done as part of the first author's Ph.D. thesis \cite{anuj2019}. The work was supported by the Council of Scientific and Industrial Research (CSIR) India [09/092(0875)/2013-EMR-I to A.K.M.] and Indian Institute of Technology (IIT) Kanpur, India.}}
\author[2]{Mohua Banerjee}
\affil[1,2]{Department of Mathematics and Statistics \newline
          Indian Institute of Technology Kanpur, Kanpur 208016, India.}
\affil[1]{\textit{more39@gmail.com}}
\affil[2]{\textit{mohua@iitk.ac.in}}
%\affil[ ]{\textit {\{email1,email2,email3,email4,email5\}@xyz.edu}}
%\eid{123@gmail.com}

% =========================================================================

\maketitle

\begin{abstract}
The article is a  study of two  algebraic structures, the `contrapositionally complemented pseudo-Boolean algebra' ({\it ccpBa}) and `contrapositionally $\vee$ complemented pseudo-Boolean algebra' ({\it c$\vee$cpBa}). The algebras have recently been  obtained from a topos-theoretic study of categories of rough sets. The salient feature of these algebras is that there are two negations, one intuitionistic and another minimal in nature, along with a condition connecting the two operators. We study properties of these algebras, give examples, and compare them with relevant existing algebras.  `Intuitionistic Logic with Minimal Negation (ILM)'  corresponding to {\it ccpBa}s and its extension ${\rm ILM}$-${\vee}$ for {\it c$\vee$cpBa}s, are then investigated.  Besides its relations with intuitionistic and minimal logics, ILM is observed to be related to Peirce's logic. With a focus on properties of the two negations, two kinds of relational semantics for  ILM and ${\rm ILM}$-${\vee}$ are obtained, and an inter-translation between the two semantics is provided.
Extracting  features of the two negations in the algebras, a further investigation  is made, %the latter being an important direction in 
following logical studies of negations that define the operators independently of the binary operator of implication. 
Using Dunn's logical framework for the purpose,  two  logics $K_{im}$ and $K_{im-{\vee}}$ are presented, where the language does not include implication. $K_{im}$-algebras are  reducts of {\it ccpBa}s. The negations in the algebras are shown to occupy distinct positions in an enhanced form of Dunn's Kite of negations. Relational semantics for $K_{im}$ and $K_{im-{\vee}}$ are given, based on Dunn's compatibility frames. Finally, relationships are established  between the different algebraic and relational semantics  for the logics defined in the work. 
%is studied for a complete picture.

% sequents
% perp semantics
\end{abstract}

%\tableofcontents

\section{Introduction}
\label{intro}

%New algebras and logics are regularly been obtained based on the area of study. Some are obtained through philosophical discussions, and others through algebraic study of some structures. Our motivation lies in the study of rough sets.
The work presents a study of  two  algebraic structures, `contrapositionally complemented pseudo-Boolean algebra' ({\it ccpBa}) and `contrapositionally $\vee$ complemented pseudo~- Boolean algebra' ({\it c$\vee$cpBa}). The salient feature of these algebras is that there are two negations amongst the operations defining the structures, one negation ($\neg$) being  intuitionistic in nature and another (${\sim}$),  a minimal negation. The two operations are connected by the involutive property for a fixed element ${\sim} 1$, with respect to the negation $\neg$. The algebraic structures came to the fore during a topos-theoretic investigation of categories formed by rough sets \cite{P}. Though the current article does not deal  with category theory or rough sets, in order to give a motivation for our study, the work leading to the definitions of {\it ccpBa} and {\it c$\vee$cpBa} is briefly recounted in the following paragraph.
%Rough sets were  introduced by Pawlak in 1982 to deal with the incompleteness arising due to missing values of attributes of the objects in the universe. The algebraic and logical properties of rough sets \cite{P} have interested researchers since the inception of the theory. 
%Based  on different definitions and properties of rough sets, various known algebras (e.g. MV-algebras, quasi-Boolean, Kleene, Stone and Nelson algebras) have been related to rough sets, and new algebras have been proposed, e.g. topological quasi-Boolean, pre-rough and rough algebras (\cite{BC3,BC4,cattaneo,Kumar2017,polkowski}, to name a few). 

%The study of categories of rough sets was initiated  in 1993 by Banerjee and Chakraborty  \cite{BC1,BC2}.  It was followed up by More and Banerjee  in \cite{AM,AM2}. 
%{\it CcpBa}s and {\it c$\vee$cpBa}s came up in  \cite{AM,AM2} as follows.

%Let us first define the notion of rough sets \cite{P}. 
Rough sets were  introduced by Pawlak in 1982 to deal with situations where only 
partial or inadequate information may be available about objects of a domain of discourse. In such a situation, there may be some objects of the domain that are not distinguishable
from others. Mathematically, the scenario is represented by a pair $(\mathbb{U},R)$ called {\it approximation space}, where $\mathbb{U}$ is a set (domain of discourse) and $R$  an equivalence relation on $\mathbb{U}$. Let $[x]$ denote the equivalence class of $x~(\in \mathbb{U})$, giving all objects of the domain that are {\it indiscernible} from $x$. For any subset $U$ of $\mathbb{U}$,   the {\it upper approximation} and the  {\it lower approximation} of $U$ in  $(\mathbb{U},R)$ are respectively defined as\\
\centerline{$\underline{U} := \{x \in \mathbb{U} ~|~ [x] \subseteq U \}$ \mbox{ and } $\overline{U} := \{x \in \mathbb{U} ~|~ [x] \cap U \neq \emptyset \}.$}
$\underline{U}$ gives the objects of the domain that {\it definitely} belong to $U$, while $\overline{U}$ consists of all objects that are {\it possibly}  in $U$. $\overline{U} {\setminus} \underline{U}$ is called the {\it boundary} of $U$, giving the region of `uncertainty'. 
%elements here could be within or outside $U$). 
So the possible region of $U$  includes its definite and boundary regions. $U$ is called a {\it rough set} in $(\mathbb{U},R)$,  described by its lower and upper approximations. (There are other equivalent definitions of rough sets in literature, cf. \cite{BC4}.)
%The pair $(\underline{X}, \overline{X})$ is called a {\it rough set} \index{rough set} in the approximation space $(U,R)$.
%Here, $\overline{X}$, $\underline{X}$, $\overline{X} {\setminus} \underline{X}$ and $U \setminus \overline{X}$ are called {\it possible}, {\it definite}, {\it boundary} and {\it non-definite} regions of $X$ respectively. Let  $\overline{\mathcal{X}}$ and $\underline{\mathcal{X}}$ denote the collections of equivalence classes in $U$ contained in the upper approximation \index{upper approximation} \index{lower approximation}and lower approximation of set $X~ (\subseteq U)$ respectively. Here $\overline{X}=\bigcup \overline{\mathcal{X}}$, $\underline{X}=\bigcup \underline{\mathcal{X}}$, and $\underline{\mathcal{X}} \subseteq \overline{X}$.
In 1993, the study of categories of rough sets was initiated   by Banerjee and Chakraborty  \cite{BC1,BC2}. A category ROUGH was proposed  with objects as triples of the form  $(\mathbb{U},R,U)$, $(\mathbb{U},R)$ being an approximation space and  $U \subseteq \mathbb{U}$. A morphism between two ROUGH-objects is defined in such a way that it maps the possible and definite regions of the rough set in the domain, respectively into  the possible and definite regions of the rough set in the range. The work was followed up by More and Banerjee  in \cite{AM,AM2}.
%(Other definitions of categories of rough sets can be found in \cite{BC2,BY,Borzooei2017,Diker2013,D,LY,AM,AM2}.)
%The upper  and lower approximations of $X$ are used to define arrows in the category, with the basic requirement that arrows should map upper approximations in one space to those in another, {\it preserving information present in the lower approximations}.
%Another category $RSC$ of rough sets was introduced by Li and Yuan in 2008 \cite{LY}. In this, the objects are not collection of equivalence classes, but sets obtained from some `rough universes'. Formally, an $RSC$-object is $(X_1,X_2)$, where $X_1 \subseteq X_2$. An $RSC$-arrow with domain $(X_1,X_2)$ and codomain $(Y_1,Y_2)$ is a map $f:X_2 \rightarrow Y_2$ such that $f(X_1) \subseteq Y_1$. The approach follows that of Iwi{\'n}ski, who gave an interpretation of rough sets based on a Boolean algebra \cite{I}.
A major objective in  \cite{BC1,BC2,AM,AM2} was to study the topos-theoretic properties of ROUGH and associated categories. In \cite{AM}, it was shown that ROUGH forms a quasitopos. Now any topos or quasitopos has an inherent algebraic structure: the collection of (strong) subobjects of any (fixed) object forms a pseudo-Boolean algebra \cite{elephant}. 
%However, in \cite{AM}, it is shown that ROUGH forms quasitopos. 
%The work in \cite{AM,AM2}  explored the algebra of strong subobjects in ROUGH.
The work in \cite{AM,AM2} explored the notion of negation in the set of strong subobjects of any ROUGH-object $(\mathbb{U},R,U)$. This set 
%It was observed in \cite{AM} that the set of strong subobjects of any ROUGH-object $(\mathbb{U},R,U)$ 
forms a Boolean algebra   $\mathcal{M}(U):=(\mathcal{M},1,0,\vee,\wedge,\rightarrow,\neg)$, where $\neg$ denotes the Boolean negation.
%was then studied \cite{AM,AM2}.
%Let ($\mathbb{U},R$)  be an approximation space and $U \subseteq \mathbb{U}$. The triple $(\mathbb{U},R,U)$ is then a $ROUGH$-object. Then it is observed that the set of strong subobjects of $(\mathbb{U},R,U)$ forms a Boolean algebra \cite{AM}, denoted as $\mathcal{M}(U):=(\mathcal{M},1,0,\vee,\wedge,\rightarrow,\neg)$. 
However, as is well-known, algebraic structures formed by rough sets are typically non-Boolean, cf. \cite{BC4}. Iwi{\'n}ski's {\it rough difference} operator \cite{I} giving relative rough complementation
%, that is negation with respect to the object $(\mathbb{U},R,U)$, 
was subsequently incorporated in $\mathcal{M}(U)$, in the form of a new negation ${\sim}$.
%It was then observed that the Boolean $\neg$ does not actually give 
%negation in the algebra is a 
%{\it relative}  complementation, that is negation  {\it with respect to} the object $(\mathbb{U},R,U)$. 
%However, the possible region of the negation of $(\mathbb{U},R,U)$ does not include the boundary region $\overline{U} {\setminus} \underline{U}$, contrary to the concept of rough sets. 
%It was then noted that a {\it relative}  complementation operation is required in the algebra, as the complementation in the structure is {\it with respect to} the object $(\mathbb{U},R,U)$.
%Iwi{\'n}ski's {\it rough difference} operator \cite{I} was subsequently incorporated in $\mathcal{M}(U)$ in the form of a new negation ${\sim}$ to address this point.
%Here, negation ${\sim}$ captures the property:
Then, for any subobject $(\mathbb{U},R,A)$ of {\rm ROUGH}-object $(\mathbb{U},R,U)$ and for $C (\subseteq \mathbb{U})$ such that ${\sim} (\mathbb{U},R,A)=(\mathbb{U},R,C)$, the possible region of $C$ contains the boundary region of $A$ \cite{anuj2019,AM} -- which is meaningful, as the boundary is the region of uncertainty. But this is not the case if ${\sim}$ is replaced by $\neg$.
%${\sim}$ requires that the possible region of the  subobject obtained on negating a subobject $(\mathbb{U},R,A)$ (say) with respect to $(\mathbb{U},R,U)$, contains the boundary region of the rough set $A$.
An investigation of properties of the two negations $\neg,\sim$ and the enhanced structure $(\mathcal{M},1,0,\vee,\wedge,\rightarrow,\neg,{\sim})$ resulted in the definition of  {\it ccpBa} and {\it c${\vee}$cpBa} \cite{AM2}. 
%A brief study of this new algebra has been done in \cite{AM2}.

%The algebraic structures discussed in this work, namely - `contrapositionally complemented pseudo-Boolean algebra' ({\it ccpBa}) and `contrapositionally $\vee$ complemented pseudo-Boolean algebra' ({\it c$\vee$cpBa}), have emerged from these studies. It is observed in \cite{AM} that a category with rough sets as objects forms a quasitopos. The inherited internal algebraic structure of the class of `strong subobjects' of an object in the quasitopos is then equipped with a new negation, the `relative rough complementation' -- this gives rise 
%. The internal algebras in categories of rough sets , give rise 
%to an instance of the {\it c${\vee}$cpBa}.  {\it c${\vee}$cpBa}s, when generalised, give  {\it ccpBa}s \cite{AM,AM2}.
%In fact, it is shown that entire classes of examples of the algebras can  be obtained by starting from an arbitrary pseudo-Boolean or Boolean algebra.
%Subsequently, algebras of subobjects obtained from these categories were investigated in \cite{AM} and \cite{AM2}. 
%two new . This work explores the two structures further, along with their corresponding logics.

The study of algebraic structures with negation  has a vast literature (cf. e.g. \cite{HR}). Some such structures  that get directly related  to {\it ccpBa}s and {\it c${\vee}$cpBa}s are Boolean algebras, pseudo-Boolean algebras (also called Heyting algebras), contrapositionally complemented lattices and Nelson algebras (cf. \cite{GN,odintsov2008,HR,sikorski1953,vakarelov1977}). 
%The distinguishing feature of  {\it ccpBa}s and {\it c${\vee}$cpBa}s is the fact that they have two negation operators, one of which is intuitionistic (${\neg} a = a\rightarrow 0$) and the other, minimal (${\sim} a = a\rightarrow {\sim} 1$) in nature. 
An important direction of algebraic studies since  the results obtained by Stone (cf. \cite{HR}), has been the investigation of representation theorems connecting algebras and topological spaces   (cf. \cite{nick2006,celani2014,clark1998,davey2002}). A representation result for {\it ccpBa} with respect to  topological spaces is given in \cite{AM2}. Here, we give duality results for both {\it ccpBa}s and {\it c${\vee}$cpBa}s with respect to topological spaces that are certain restrictions of Esakia spaces (cf. \cite{davey2002,celani2014}), the topological spaces corresponding to pseudo-Boolean algebras.

A study  of classes of algebras naturally leads to an investigation of corresponding logics. In \cite{anuj2019,AM2}, the logic ILM - {\it Intuitionistic logic with minimal negation}, and ILM-$\vee$, corresponding to the algebras {\it ccpBa} and  its extension {\it c${\vee}$cpBa} respectively, have been defined. The language of ILM has two negations $\neg$, ${\sim}$, and propositional constants $\top$, $\bot$, apart from other propositional connectives. The nature of the two connectives of negation present in the logics expectedly yields relations of ILM with intuitionistic logic (IL) and minimal logic (ML). (Recall that ML  is the logic corresponding to the class of contrapositionally complemented lattices \cite{Johansson}.) Utilizing the notions of `interpretation' as given in \cite{CD,GFPO,FK,HR}, a further comparison is made between ILM and IL, ML in \cite{AM2}. In continuation of the comparative study, in this work, we relate ILM with Peirce's logic, which is an extension of {\rm ML} obtained by adding Peirce's law $((\alpha \rightarrow \beta) \rightarrow \alpha) \rightarrow \alpha)$ \cite{prior1958,Segerberg1968}. In fact, ILM gets related to a special case of Peirce's logic, ${\rm JP'}$, which is ML along with the axiom $(({\sim}\top \rightarrow \beta) \rightarrow {\sim}\top) \rightarrow {\sim}\top)$. As mentioned by Segerberg \cite{Segerberg1968} and discussed in \cite{odintsov2005}, ${\rm JP'}$ (also called Glivenko's logic) is the weakest logic amongst the extensions of ML in which ${\sim}{\sim} \alpha$ is derivable whenever $\alpha$ is derivable.
%Prawitz \cite{prawitz1965} and
%ILM is observed to be equivalent to an extension of  a special case of Peirce's logic.
 
A logic may be imparted multiple semantics, aside from an algebraic one.
For propositional logics with negation, different relational semantics have been introduced,  for instance in  \cite{Dosen1986,Dosen1999,Dunn2005,YG,Segerberg1968,Vakarelov1989}. As is well-known, Kripke \cite{Kripke1965a,Kripke1965} first studied relational semantics for IL, where frames are partially ordered sets $(W,\leq)$ called {\it normal} frames. On these frames, Segerberg \cite{Segerberg1968} added  a hereditary set $Y_0$ of `queer' worlds at each of which $\bot$ (`falsum') holds. By adding conditions on such frames, natural relational semantics for ML and various extensions of ML are  obtained \cite{odintsov2008,odintsov2007,Segerberg1968}. One such semantics   given by Woodruff \cite{Woodruff1970}, involving {\it sub-normal} frames, is used in our work to obtain relational semantics for {\rm ILM} and {\rm ILM-${\vee}$}.

As Do{\v{s}}en remarked in \cite{Dosen1986}, a drawback of the Segerberg-style semantics is that it cannot be used to characterize logics with negation weaker than minimal negation. This problem was addressed by Do{\v{s}}en and Vakarelov \cite{Vakarelov1989} independently. In their work, taking motivation from Kripke frames in modal logic and in particular, the accessibility relations in the frames, negation is considered as an impossibility (modal) operator. 
%For a unary modal operator $\lozenge$ in basic modal logic, the semantics in the model $(W, R)$ is defined for any formula $\alpha$ and possible world $w$ $(\in W)$ as:\\
%\centerline{$w \vDash \lozenge \alpha \Leftrightarrow \exists v \in W (w R v ~\&~ v \vDash \alpha).$}
%Treating negation as an `{\rm impossibility}' modal operator, represented as $\cancel{\lozenge}$, means: Whenever $w$ satisfies $\cancel{\lozenge} \alpha$, the right side of the above equation should never be true, that is, $\alpha$ should not be true at any world \index{accessible} `accessible' from $w$.\\
%\centerline{$w \vDash \cancel{\lozenge} \alpha \Leftrightarrow \forall v \in W (wRv \Rightarrow v \nvDash \alpha).$}
%Do{\v{s}}en defined semantics for negations ${\sim}$ and ${\neg}$ using $\cancel{\lozenge}$.  We shall see that 
This approach results in another relational semantics for {\rm ILM} and {\rm ILM-${\vee}$}. We give a Do{\v{s}}en-style semantics for the two logics, and show that an inter-translation exists between the two (Segerberg-style and Do{\v{s}}en-style) relational semantics in the lines of that given for ML in \cite{Dosen1986}. 

%Negations in rough sets have a special place in the algebraic and logical studies of the theory. 
In the logical systems given by Do{\v{s}}en, the alphabet of the language has the connectives of implication, disjunction, conjunction and negation. Studying negation independently, particularly in the absence of implication, constitutes an important area of work on logics. Dunn \cite{Dunn1990,dunn1993,Dunn1995,Dunn1996,Dunn2005}, Vakarelov \cite{Vakarelov1989} and others \cite{restall,shramko2005} have studied logics with negations and without implication. A basic feature of these logics is that the logical consequence, in the absence of  implication, is defined through  {\it sequents}: pairs of formulas of the form $(\phi, \psi)$, written as $\phi \vdash \psi$. 
%Negation is taken as a fundamental connective, and 
Properties are introduced in the logics as axioms or rules, defining  {\it pre-minimal}, {\it minimal}, {\it intuitionistic} or other negations. To round up the study of negations defined through ${\rm ILM}$ and ${\rm ILM}$-${\vee}$, we make an investigation adopting the approach of Dunn and Vakarelov. Extracting the features of the two negations,  logics $K_{im}$ and $K_{im-{\vee}}$ are defined. $K_{im}$-algebras are  reducts of {\it ccpBa}s, while $K_{im-{\vee}}$-algebras are  reducts of {\it c${\vee}$cpBa}s. The properties of the algebras show that Dunn's Kite of negations can be enhanced to one accommodating pairs of negations, two distinct nodes in which are occupied by the pair defining the $K_{im}$ and $K_{im-{\vee}}$-algebras.
Dunn's compatibility frames help in specifying the relational semantics for $K_{im}$ and $K_{im-{\vee}}$.

Connections between algebras and relational frames have been studied in literature, together with duality results (cf. e.g. \cite{hartonas1997,Johnstone1986}).
%Along the same lines, relationships are established  between the different algebraic and relational semantics defined for the logics presented in this work.
Kripke \cite{Kripke1965a} demonstrated the connection between normal frames and pseudo-Boolean algebras (cf. \cite{bezhanishvili,chagrov}). 
We observe here that this can easily be extended to obtain connections between sub-normal frames and {\it ccpBa}s. As compatibility frames provide a relational semantics for $K_{im}$, we shall also establish this to and fro connection between relational and algebraic semantics for $K_{im}$ and $K_{im-{\vee}}$.

The paper is organised as follows. In the next section, we give the definitions of {\it ccpBa}s and {\it c${\vee}$cpBa}s, and some of their properties. A few  examples of finite {\it ccpBa}s and {\it c${\vee}$cpBa}s are studied in Section \ref{sec-eg-ccpba}, while representation results are given in Section \ref{representations}. We  then move to the  logics  ILM and ILM -${\vee}$ in Section \ref{seclogic}. Properties and  relationships with IL, ML, and Peirce's logic are discussed. In Sections \ref{sec-KripkeILM} and \ref{subsec-KripkeILM}, we give, respectively, the Do{\v{s}}en-style and Segerberg-style semantics  for the logics. To establish completeness results, an inter-translation between frames in the two semantics is used -- this is given  in Section \ref{inter}. In Section \ref{sec-bdll}, we study  the negations in systems without implication. The logics $K_{im}$ and $K_{im-\vee}$  along with   algebraic and relational semantics are discussed. Connections between algebraic and relational semantics of the logics are observed in Section \ref{sec-compare}. We conclude the work in Section \ref{conc}.
%including a brief mention of other possible relational semantics for the logics.

Hereafter, the symbols $\forall$, $\exists$, $\Rightarrow$, $\Leftrightarrow$, $\&$ ({\it and}), {\it or}, {\it if} $\ldots$ {\it then} and $not$ will be used with the usual meanings in the metalanguage. For basic definitions and results, we refer to \cite{HR} (for algebra and logic), \cite{davey2002} (representation and duality), and \cite{Dosen1986,Dunn2005,Segerberg1968} (relational semantics).

\section{The algebras \textit{ccpBa} and \textit{c\texorpdfstring{{$\vee$}}cpBa}}
\label{sec-ccpba}

%
%In this section we shall first see the definitions and properties of {\it ccpBa} and {\it c$\vee$cpBa}. This shall be followed by examples of both the algebras (Section \ref{sec-eg-ccpba}). A brief comparison of {\it ccpBa} with some lattices with negations is also studied (Section \ref{sec-comparison}). Finally, we give the duality result for the algebras with respect to certain ordered topological spaces (Section \ref{representations}).

%The study of algebraic structures with negation(s) in lattice theory has a vast literature \cite{HR}. Consider a distributive lattice $(A,1,\vee,\wedge,\rightarrow)$ with implication and upper bound $1$, on a partially ordered set $(A,\leq)$. Relatively pseudo-complemented ({\it rpc}) lattices, pseudo-Boolean algebras ({\it pBa}) (also called Heyting algebras) and Boolean algebras are the most common distributive lattices.
%
%Their can be various motivations for new algebraic structures. For example, {\it pBa} were defined through the study of the intuitionistic nature of negation operator in lattices \cite{J1937,sikorski1953}. Our significance comes from the rough set theory. The study of algebra of `subobjects' in a category of rough sets leads us to two new algebraic structures, involving two negations $\neg$ and ${\sim}$.

Let us first recall the definitions  of the two algebraic structures.
% {\it ccpBa} and {\it c$\vee$cpBa}.
\begin{definition} \rm{\cite{AM2}}
\label{ccpBa}
$ $\\
An abstract algebra $\mathcal{A}:=(A,1,0, \vee, \wedge, \rightarrow,\neg,{\sim} )$ is called a {\it contrapositionally complemented pseudo-Boolean algebra} ({\it ccpBa}), if the reduct $(A,1,0, \vee,\wedge, \rightarrow)$ forms a bounded relatively pseudo-complemented (\textit{rpc}) lattice, and for all $a \in A$,\\
\centerline{${\neg} a = a \rightarrow 0$ and ${{\sim}} a = a \rightarrow (\neg \neg {\sim} 1)$.}
If, in addition,  for all $a \in A$, $a \vee {\sim} a =1$, we call $\mathcal{A}$ a {\it contrapositionally $\vee$ complemented pseudo-Boolean algebra} ({\it c${\vee}$cpBa}).
\end{definition}

\begin{obs}
\label{obs-ccpba}
{\rm
For the negation ${\sim}$, the condition ${\sim} a = a \rightarrow (\neg \neg {\sim} 1)$ can be equivalently expressed as
%\vskip 2pt
%\centerline{(1) ${\sim} a = a \rightarrow {\sim} 1$ and (2) ${\sim} 1 = (\neg \neg {\sim} 1)$.}
%\vskip 2pt
\begin{enumerate}[label={{\rm (\arabic*)}}, leftmargin=1.5 \parindent, noitemsep, topsep=0pt] {\rm
\item ${\sim} a = a \rightarrow {\sim} 1$ and
\item ${\sim} 1 = (\neg \neg {\sim} 1)$.
}\end{enumerate}
}
\end{obs}

\noindent  In any {\it ccpBa} $\mathcal{A}:=(A,1,0, \vee, \wedge, \rightarrow,\neg,{\sim})$, the property $\neg a = a \rightarrow 0$ makes the bounded \textit{rpc} lattice 
$(A,1,0, \vee,\wedge,\rightarrow)$ a pseudo-Boolean algebra ({\it pBa})  \cite{HR} with respect to the negation $\neg$. In any  {\it rpc} lattice $(A,1,\vee, \wedge,\rightarrow)$, if negation ${\sim}$ is defined as in Observation \ref{obs-ccpba}(1), the resulting lattice $(A,1,\vee, \wedge,\rightarrow,{\sim})$ forms a contrapositionally complemented ({\it cc}) lattice \cite{HR}.  Thus a {\it ccpBa} can be considered as an amalgamation of a $cc$ lattice and a $pBa$ satisfying Condition (2) in Observation \ref{obs-ccpba} --  the reason for naming the algebraic structure  `contrapositionally-complemented pseudo-Boolean algebra'.

In any {\it cc} lattice, the negation ${\sim}$ is completely determined by $\rightarrow$ and the element ${\sim} 1$. Moreover, the element ${\sim} 1$ need not be the bottom element of the lattice. Note that in a  {\it ccpBa}, there is a bottom element $0$, which defines the negation $\neg$. Observation \ref{obs-ccpba}(2) then gives the involutive property for the element ${\sim} 1$ of a {\it ccpBa}, with respect to the negation $\neg$. We shall observe through examples  in Section \ref{sec-eg-ccpba} that a {\it ccpBa} (a) need not have the involutive property for {\it all} elements, and (b) ${\sim} 1$ need not be the bottom element $0$.

In a {\it c${\vee}$cpBa}, the reduct lattice forms a {\it contrapositionally $\vee$ complemented} ({\it c${\vee}$c}) lattice \cite{MN}, and thus the nomenclature  `contrapositionally $\vee$ complemented pseudo-Boolean algebra'.
Condition (2)  of Observation \ref{obs-ccpba}, as we shall see in Section \ref{sec-eg-ccpba}, is a distinctive property of {\it ccpBa}s that is not true in general for an arbitrary bounded {\it cc} lattice.

%We shall observe this through an example in Section \ref{sec-eg-ccpba}. 
Let us now list some properties of {\it ccpBa}s. For properties of {\it cc}, {\it rpc} lattices and {\it pBa}s, we refer to \cite{Goldblatt1974,HR}.

\begin{proposition}
\label{prop1-ccpBa}
In any {\it ccpBa} $\mathcal{A}:=(A,1,0, \vee, \wedge, \rightarrow,\neg,{\sim} )$, the following hold for all $a,b \in  A$.
\begin{multicols}{2}
\begin{enumerate}[label={{\rm (\arabic*)}}, leftmargin=1.5 \parindent, noitemsep, topsep=0pt]
\item $a \rightarrow {\sim}b = b \rightarrow {\sim} a$
\item $a \leq {\sim}{\sim} a$
%\item $\neg \neg {\sim} 1 = {\sim} 1$
\item ${\sim}{\sim}({\sim} 1 \rightarrow a) = 1$
\item ${\sim} a = \neg (a \wedge \neg {\sim} 1)$
\item $\neg a \leq {\sim} a $
\item $ a \leq {\sim} \neg a$
\item $ \neg {\sim} a \leq {\sim} \neg a$
\item $ {\sim} a = \neg \neg {\sim} a$
\item $ \neg {\sim} \neg a \leq \neg a$
\end{enumerate}
\end{multicols}
\end{proposition}
\begin{proof} $ $ \newline
(1) and (2) follow directly from properties of {\it cc} lattices and Definition \ref{ccpBa}.\\
\noindent (4): In the {\it rpc} lattice $(A,1,\vee, \wedge, \rightarrow)$, for any $a,b,c \in A$,\\
(i) $(a \rightarrow b) \wedge (a \rightarrow c) = (a \rightarrow (b \wedge c))$ and (ii) $a \rightarrow (b \rightarrow c) = (a \wedge b) \rightarrow c$.\\
Using (i), $({\sim} 1 \rightarrow 0) \wedge ({\sim} 1 \rightarrow a) = {\sim} 1 \rightarrow (0 \wedge a)$ = ${\sim} 1 \rightarrow 0 = \neg {\sim} 1$.\\
Now, ${\sim}({\sim} 1 \rightarrow a) = ({\sim} 1 \rightarrow a ) \rightarrow \neg \neg {\sim} 1 = \neg {\sim} 1 \rightarrow \neg ({\sim} 1 \rightarrow a) = \neg {\sim} 1 \rightarrow (({\sim}1 \rightarrow a) \rightarrow 0) = (\neg {\sim} 1 \wedge ({\sim}1 \rightarrow a)) \rightarrow 0) = (({\sim} 1 \rightarrow 0)\wedge ({\sim}1 \rightarrow a)) \rightarrow 0)$. Using (ii), $ {\sim}({\sim} 1 \rightarrow a)= \neg {\sim} 1 \rightarrow 0 = \neg \neg {\sim} 1.$\\
Using {\it rpc} property $a \rightarrow b = 1 \Leftrightarrow a \leq b$, we have ${\sim}({\sim} 1 \rightarrow a)  \leq \neg \neg {\sim} 1 
\Rightarrow {\sim}({\sim} 1 \rightarrow a) \rightarrow \neg \neg {\sim} 1 = 1$. Therefore, ${\sim}{\sim}({\sim} 1 \rightarrow a) = 1$.\\
\noindent (5): ${\sim} a = a \rightarrow (\neg \neg {\sim} 1) = a \rightarrow (\neg {\sim} 1 \rightarrow 0) = (a \wedge \neg {\sim} 1) \rightarrow 0=\neg (a \wedge \neg {\sim} 1)$.\\
(6): Using {\it rpc} property $b \leq c \Rightarrow (a \rightarrow b) \leq (a \rightarrow c)$, we have  $0 \leq {\sim} 1 \Rightarrow a \rightarrow 0 \leq  a \rightarrow {\sim} 1$.  Therefore, $\neg a \leq {\sim} a$.\\
\noindent (7): From (6) and {\it cc} lattice property $a \leq \neg \neg a$, we have $a \leq \neg \neg a \leq {\sim} \neg a$.\\
\noindent (8): From (6), $\neg a \wedge {\sim} a = \neg a$.  Using {\it pBa} property $a \wedge \neg a = 0$, we have $(\neg {\sim} a \wedge {\sim} a)= 0$, i.e. $(\neg {\sim} a \wedge {\sim} a) \wedge \neg a = 0 \leq {\sim} 1$. Recall the {\it rpc} property $a \wedge c \leq b  \Leftrightarrow c \leq a \rightarrow b$. So, $\neg {\sim} a \wedge ({\sim} a \wedge \neg a)  \leq {\sim} 1 \Rightarrow \neg {\sim} a \leq ({\sim} a \wedge \neg a) \rightarrow {\sim} 1 = {\sim} ({\sim} a \wedge \neg a) = {\sim} (\neg a)$.\\
\noindent (9): From (5) and {\it cc} lattice property $\neg a =\neg \neg \neg a$, we have $\neg \neg {\sim} a = \neg \neg \neg (a \wedge \neg {\sim} 1) =\neg (a \wedge \neg {\sim}1) = {\sim} a$.\\
\noindent (10): From (7), $a \leq {\sim} \neg a \leq \neg \neg ({\sim} \neg a)$. So, $a \rightarrow \neg \neg {\sim} \neg a =1 \Rightarrow \neg \neg \neg {\sim} \neg a \rightarrow \neg a = 1 \Rightarrow \neg {\sim} \neg a \rightarrow \neg a = 1 \Rightarrow \neg {\sim} \neg a \leq \neg a$.
\end{proof}

\begin{obs} {\rm Using Observation \ref{obs-ccpba}, the property ${\sim}{\sim}({\sim} 1 \rightarrow a) = 1$ in Proposition \ref{prop1-ccpBa}(3) can also be expressed as:\\
\centerline{$(({\sim} 1 \rightarrow a) \rightarrow {\sim} 1) \rightarrow {\sim} 1 = 1,$}
which is a special case of {\it Peirce's law} \cite{Segerberg1968} $((b \rightarrow a) \rightarrow b) \rightarrow b = 1$, for $b={\sim} 1$.}
\end{obs}

\subsection{Comparison with other algebras}
\label{sec-comparison}

%Let us now compare the algebras defined above with some existing lattices. 
Based on the properties of {\it cc} lattices and Proposition \ref{prop1-ccpBa}, the following is straight-forward.

\begin{proposition}
\label{prop3-ccpBa}
$ $\\
Consider  $\mathcal{A}:=(A,1,0,\vee,\wedge,\rightarrow,\neg, \sim)$ such that the reduct $(A,1,0,\vee,\wedge,\rightarrow,\neg)$ is a {\it pBa}. Then the following are equivalent.
\begin{enumerate}[label={{\rm (\arabic*)}}, leftmargin=1.5 \parindent, noitemsep, topsep=0pt]
\item $\mathcal{A}$ is a $ccpBa$.
\item $(A,1,\vee,\wedge,\rightarrow,\sim)$ is a {\it cc} lattice and $\neg \neg {\sim} 1={\sim} 1$.
\item $(A,1,\vee,\wedge,\rightarrow,\sim)$ is a {\it cc} lattice and  ${\sim}{\sim}({\sim} 1 \rightarrow a) = 1$ for any $a \in A$.
\end{enumerate}
\end{proposition}
%\begin{proof}
%$(1) \Rightarrow (3)$ is directly using Proposition \ref{prop1-ccpBa}.\\
%$(3) \Rightarrow (2)$. ${\sim}{\sim}({\sim} 1 \rightarrow a) = 1 \Rightarrow {\sim}({\sim} 1 \rightarrow a) \rightarrow {\sim} 1 = 1$ because in a {\it cc} lattice, ${\sim} x = x \rightarrow {\sim} 1$. Therefore, we have ${\sim}({\sim} 1 \rightarrow a) \leq {\sim} 1$. In particular  ${\sim}({\sim} 1 \rightarrow 0) \leq {\sim} 1$. Now, $0 \leq {\sim} 1$ implies $\neg\neg{\sim}1 \wedge \neg {\sim} 1 \leq {\sim} 1$ (by {\it pBa} property $\neg x \wedge x = 0$). Now, using {\it rpc} property $a \wedge c \leq b  \Leftrightarrow c \leq a \rightarrow b$, we have\\
%\centerline{$\neg \neg {\sim} 1 \leq \neg {\sim} 1 \rightarrow {\sim} 1 = ({\sim} 1 \rightarrow 0) \rightarrow {\sim} 1  = {\sim} ({\sim} 1 \rightarrow 0) \leq {\sim} 1.$}
%$(2) \Rightarrow (1)$. Now, using {\it cc} lattice property, ${\sim} a = a \rightarrow {\sim} 1 = a \rightarrow (\neg \neg {\sim} 1)$.
%\end{proof}
%In the premise of the above Proposition, the algebraic structure $\mathcal{A}:=(A,1,0,\vee,\wedge,\rightarrow,\neg, \sim)$ has an underlying {\it pBa} $(A,1,0,\vee,\wedge,\rightarrow,\neg)$. 
\noindent Note that (3) in Proposition \ref{prop3-ccpBa} does not have any occurrence of $\neg$. The following example shows that the condition in the statement of the proposition that the reduct $(A,1,0,\vee,\wedge,\rightarrow,\neg)$ is a {\it pBa}, is necessary for $\mathcal{A}:=(A,1,0,\vee,\wedge,\rightarrow,\neg, \sim)$  to be a $ccpBa$.
%There exists an algebra $\mathcal{A} :=(A,1,\vee,\wedge,\rightarrow,{\sim})$ satisfying (3), such that there is no least element $0 \in A$ and no negation $\neg$, with the help of which $\mathcal{A}$ may be extended to a {\it ccpBa} $(A,1,0,\vee,\wedge,\rightarrow,\neg, {\sim})$.

\begin{example}
\label{eg4-ccpBa}
%\label{obs4-ccpBa}
{\rm Consider the linear lattice $(L,0,\vee,\wedge)$, where $L$ consists of all negative integers including 0.
%Note that $L$ has $0$ as its top element, and has no bottom element.
Define an implication operator (cf. \cite{gupta1990}) as
\begin{align*}
a \rightarrow b &:= b \quad \mbox{ if } b < a,\\
&:= 0 \quad \mbox{ otherwise.}
\end{align*}
Here, $(L,0,\vee,\wedge,\rightarrow)$ is an {\it rpc} lattice. Define ${\sim}$ as ${\sim} a := a \rightarrow 0$.
%In other words, for all $a \in L$, ${\sim} a = 0$, the top element of $L$.
Thus, trivially, $(L,0,\vee,\wedge,\rightarrow,{\sim})$
%becomes a {\it cc} lattice satisfying 
satisfies condition (3) of Proposition \ref{prop3-ccpBa}. However, $(L, \leq)$ does not have a lower bound. So, $(L,0,\vee,\wedge,\rightarrow)$ cannot be extended to a {\it pBa}, and hence the {\it cc} lattice $(L,0,\vee,\wedge,\rightarrow,{\sim})$ cannot be extended to a {\it ccpBa}.}
\end{example}

%One of the most 
A familiar lattice with two distinct negation operators  and having the same algebraic type as a {\it ccpBa}, is the \textit{quasi-pseudo Boolean algebra}, also called {\it Nelson algebra} \cite{HR}. 
% in which one of the negations satisfies the involution property. 
We observe the following differences between Nelson algebras and {\it ccpBa}s.\\
(1) The implication $\rightarrow$ is a relative pseudo-complement operator in a ${\it ccpBa}$, which may not be true for $\rightarrow$ in a Nelson algebra.\\
(2) The negation ${\sim}$ satisfies the involution property ${\sim}{\sim}a =a$ for all $a \in A$ in a Nelson algebra $\mathcal{A} := (A,1,0,\vee,\wedge,\rightarrow,\neg,{\sim})$. This may not be true in an arbitrary {\it ccpBa}, neither of the two negations in a {\it ccpBa} may satisfy this property.\\
%\sout{(3). The negation ${\neg}$ satisfies the property ${\sim} a =\neg \neg {\sim} a$ for all $a \in \mathcal{A}$ in a {\it ccpBa} $\mathcal{A}$ (Proposition \ref{prop2-ccpBa}(4)). This property may not hold true for an arbitrary Nelson algebra.\\}
The above points (1) and (2) may easily be verified using the $3$-element Nelson algebra $\mathcal{A}:=(A,1,0,\vee,\wedge,\rightarrow,\neg,{\sim})$ from \cite{HR}, where $A:=\{0,a,1\}$ with the ordering as $0 \leq a \leq 1$. The operators $\rightarrow$, $\neg$ and ${\sim}$ are defined in Table \ref{3nelson}.  In \cite{HR}, it has been observed that every $3$-element Nelson algebra is isomorphic to $\mathcal{A}$.

\begin{table}[ht]
\begin{tabular}{ccc}
    \begin{minipage}{.4\linewidth}
        \centering
        \begin{tabular}{ll}
            \begin{tabular}{ |c|c|c|c| } 
			\hline
			$\rightarrow$ & $0$ & $a$ & $1$ \\ \hline
			$0$ & $1$ & $1$ & $1$ \\ \hline
			$a$ & $1$ & $1$ & $1$ \\ \hline
			$1$ & $0$ & $a$ & $1$ \\ \hline
			\end{tabular}
        \end{tabular}
    \end{minipage} &

    \begin{minipage}{.2\linewidth}
        \centering
        \begin{tabular}{ll}
            \begin{tabular}{ |c|c| } 
			\hline
			$x$ & $\neg x$ \\ \hline
			$0$ & $1$ \\ \hline
			$a$ & $1$ \\ \hline
			$1$ & $0$ \\ \hline
			\end{tabular}
        \end{tabular}
    \end{minipage} &
    
    \begin{minipage}{.2\linewidth}
        \centering
        \begin{tabular}{ll}
            \begin{tabular}{ |c|c| } 
			\hline
			$x$ & ${\sim} x$ \\ \hline
			$0$ & $1$ \\ \hline
			$a$ & $a$ \\ \hline
			$1$ & $0$ \\ \hline
			\end{tabular}
        \end{tabular}
    \end{minipage}
\end{tabular}
\caption{$3$-element Nelson algebra}
\label{3nelson}
\end{table}

\noindent Let us establish the claims made in (1) and (2) above.\\
(1): In $\mathcal{A}$, $a \leq a \rightarrow 0 = 1$ holds, but $a = a \wedge a \nleq 0$. Thus the operator $\rightarrow$ is not a relative pseudo-complement in $\mathcal{A}$, and $\mathcal{A}$ is not a ${\it ccpBa}$. \\
(2): We shall see in the next section that there are only two $3$-element {\it ccpBa}s upto isomorphism. However, none of the negations in these is involutive. So, no $3$-element {\it ccpBa} forms a Nelson algebra.\\
Thus, we can conclude that neither of the two algebraic classes of {\it ccpBa}s and Nelson algebras is a sub-class of the other.

\subsection{Examples of \textit{ccpBa}s and \textit{c\texorpdfstring{${\vee}$}cpBa}s}
\label{sec-eg-ccpba}

%This section covers examples of {\it ccpBa}s and {\it c${\vee}$cpBa}s. 
%In \cite{AM,AM2}, a class of examples of  {\it c${\vee}$cpBa}s is obtained through the study of categories of rough sets. 
%The category of rough sets forms a quasitopos \cite{AM,AM2}; any topos or quasitopos $\mathscr{C}$ has an internal algebraic structure on the class of `strong subobjects' of any $\mathscr{C}$-object. The internal algebras in categories of rough sets along with a new negation operator, give rise to a class of examples of {\it c${\vee}$cpBa}s. In fact, 
It was mentioned in Section \ref{intro} that the set of strong subobjects of any ROUGH-object $(\mathbb{U},R,U)$ gives rise to a {\it ccpBa} (in fact, a {\it c${\vee}$cpBa})
%$(\mathcal{M},1,0,\vee,\wedge,\rightarrow,\neg,{\sim})$
\cite{AM2}. It is  shown in \cite{AM,AM2} that on abstraction of the category-theoretic construction, entire classes of examples of {\it ccpBa}s and {\it c${\vee}$cpBa}s can  be obtained by starting from an arbitrary pseudo-Boolean or Boolean algebra. Let us recall the construction and results from \cite{AM2}. Consider a {\it pBa} $\mathcal{H}:=(H, 1,0,\vee, \wedge, \rightarrow, \neg)$, and  the set $\mathcal{H}^{[2]}:=\{(a,b): \ a \leq b, \ a,b \in H\}$ (cf. \cite{CR}). Fix any element $u:=(u_1,u_2) \in \mathcal{H}^{[2]}$. Define the following set $A_{u}$ and operators on it:
\begin{align*}
A_{u}:=\{(a_1,a_2) \in \mathcal{H}^{[2]} \ &: a_2 \leq u_2 ~{\rm and}~\ a_1 =a_2 \wedge u_1 \}, \\[5pt]
\sqcup:~(a_1,a_2) \sqcup (b_1,b_2) &:= (a_1 \vee b_1, a_2 \vee b_2),\\
\sqcap:~(a_1,a_2) \sqcap (b_1,b_2) &:= (a_1 \wedge b_1, a_2 \wedge b_2),\\
\rightarrow:~(a_1,a_2) \rightarrow (b_1,b_2) &:= ((a_1 \rightarrow b_1) \wedge u_1, (a_2 \rightarrow b_2) \wedge u_2),\\
\sim:~{\sim} (a_1,a_2) &:= (u_1 \wedge \neg a_1, u_2 \wedge \neg a_1), \mbox{ and}\\
\neg:~\neg (a_1,a_2) &:= (a_1,a_2) \rightarrow (0,0).
\end{align*}
%The entire study came from \cite{AM2}. There we showed the following result.

\begin{proposition}\label{propAu}
{\rm \cite{AM2}} $\mathcal{A}_u:=(A_{u},(u_1,u_2),(0,0), \sqcap,\sqcup,\rightarrow,\neg, \sim)$ forms a {\it ccpBa}. Moreover, if $\mathcal{H}$ is a Boolean algebra, then $\mathcal{A}_{u}$ forms a {\it c${\vee}$cpBa}.
\end{proposition}

Consider the 6-element {\it pBa} $\mathcal{H}_6:=(H_6,1,0,\vee,\wedge,\rightarrow,\neg)$, for which the Hasse diagram is given by Figure \ref{fig-6rpcA}.

\begin{figure}[hbt!]
\begin{minipage}[b]{0.50\textwidth}
\centering
\begin{tikzpicture}
		\node (A) at (1, 0) {$0$};
		\node (B) at (0, 1) {$y$};
		\node (C) at (2, 1) {$z$};
		\node (D) at (1, 2) {$w$};
		\node (E) at (3, 2) {$x$};
		\node (F) at (2, 3) {$1$};
		\draw[->]  (A) -- (B);
		\draw[->]  (A) -- (C);
		\draw[->]  (C) -- (E);
		\draw[->]  (B) -- (D);
		\draw[->]  (C) -- (D);
		\draw[->]  (D) -- (F);
		\draw[->]  (E) -- (F);
	\end{tikzpicture}
    \caption{$\mathcal{H}_6$ - a $6$-element {\it pBa}}%
    \label{fig-6rpcA}%
\end{minipage}
\begin{minipage}[b]{0.48\textwidth}
\centering 
\begin{tabular}{ |c|c|c|c|c|c|c| } 
\hline
\hspace{0.3mm} $\rightarrow$ \hspace{0.3mm} & ~$0$~ & ~$y$~ & ~$z$~ & ~$w$~ & ~$x$~ & ~$1$~ \\ \hline
$0$ & $1$ & $1$ & $1$ & $1$ & $1$ & $1$ \\ \hline
$y$ & $x$ & $1$ & $x$ & $1$ & $x$ & $1$ \\ \hline
$z$ & $y$ & $y$ & $1$ & $1$ & $1$ & $1$ \\ \hline
$w$ & $0$ & $y$ & $x$ & $1$ & $x$ & $1$ \\ \hline
$x$ & $y$ & $y$ & $w$ & $w$ & $1$ & $1$ \\ \hline
$1$ & $0$ & $y$ & $z$ & $w$ & $x$ & $1$ \\ \hline
\end{tabular}
\captionof{table}{Implication in $\mathcal{H}_6$}
\label{tab-implA}
\end{minipage}
\end{figure}

\noindent  Define $(u_1,u_2):=(z,w) \in \mathcal{H}_6^{[2]}$. Proposition \ref{propAu} implies that  $\mathcal{A}_u$ forms a ${\it ccpBa}$ (Figure \ref{fig-Au}). Similarly, for other choices of $(u_1,u_2)$, we get different {\it ccpBa}s.

\begin{figure}[hbt!]
  \begin{minipage}[b]{0.48\textwidth}
  \centering
    \begin{tikzpicture}
		\node (A) at (1, 0) {$(0,0)$};
		\node (B) at (0, 1) {$(0,y)$};
		\node (C) at (2, 1) {$(z,z)$};
		\node (D) at (1, 2) {$(z,w)$};
		\draw[->]  (A) -- (B);
		\draw[->]  (A) -- (C);
		\draw[->]  (B) -- (D);
		\draw[->]  (C) -- (D);
	\end{tikzpicture}
    \captionof{figure}{$\mathcal{A}_u$ for $u:=(z,w)$}
    \label{fig-Au}
  \end{minipage}
  \begin{minipage}[b]{0.48\textwidth}
    \centering
  \begin{tabular}{ |c|c| } 
			\hline
			$x$ & ${\sim} x$ \\ \hline
			$(0,0)$ & $(0,y)$ \\ \hline
			$(0,y)$ & $(z,w)$ \\ \hline
			$(z,z)$ & $(0,y)$ \\ \hline
			$(z,w)$ & $(z,w)$ \\ \hline
			\end{tabular}
      \captionof{table}{Negation ${\sim}$ in $\mathcal{A}_u$}
      \label{obs3-table}
    \end{minipage}
  \end{figure}

\noindent Proposition \ref{propAu} also mentions that if $\mathcal{H}$ is a Boolean algebra, then $\mathcal{A}_u$ is a {\it c${\vee}$cpBa}. However, the converse is not always true, i.e. if $\mathcal{A}_u$ is a {\it c${\vee}$cpBa}, then $\mathcal{H}$ need not always be a Boolean algebra. This can be observed from Table \ref{obs3-table} for the above example: $\mathcal{A}_u$ is a {\it c${\vee}$cpBa}, however $\mathcal{H}_6$ is not a Boolean algebra.
\vskip 3pt

Using Definition \ref{ccpBa} and Observation \ref{obs-ccpba}, examples of {\it ccpBa} can also be obtained as follows: in any {\it pBa}  $\mathcal{A}:=(A,1,0, \vee, \wedge, \rightarrow,\neg)$, choose an element `${\sim} 1$' in $\mathcal{A}$ such that ${\sim} 1 = \neg \neg {\sim} 1$, and define ${\sim}a$ as ${\sim} a := a\rightarrow {\sim} 1$. This results in a {\it ccpBa} $\mathcal{A}:=(A,1,0, \vee, \wedge, \rightarrow,\neg,{\sim} )$.
Let us see some examples of finite {\it ccpBa}s and {\it c${\vee}$cpBa}s obtained in this manner.
\vskip 3pt

\noindent (A) Consider the only $3$-element {\it pBa} $(A,1,0,\vee,\wedge,\rightarrow_1,\neg_1)$ upto isomorphism, where $A := \{0,a,1\}$ with ordering $0 \leq a \leq 1$, $\rightarrow_1$ and $\neg_1$ defined as in Tables \ref{3rpc-imp} and \ref{3pba-neg} respectively. There are two choices for ${\sim} 1$ such that ${\sim} 1 = \neg_1 \neg_1 {\sim} 1$, namely the elements $0$ and $1$.

\begin{enumerate}[noitemsep]
\item ${\sim} 1 := 0$. This means ${\sim} x = \neg_1 x$ (Table \ref{3pba-neg}). The resulting algebra $\mathcal{A}':=(A,1,0,\vee,\wedge,\rightarrow_1,\neg_1,\neg_1)$ is a $ccpBa$, but not a $c{\vee}cpBa$. 
\item ${\sim} 1 := 1$, mapping each element of $A$ to 1 (Table \ref{3ccpba-neg}). The resulting algebra  $\mathcal{B}':=(A,1,0,\vee,\wedge,\rightarrow_1,\neg_1,{\sim}_1)$ is a {\it c${\vee}$cpBa}.
\end{enumerate}
Since $(A,1,0,\vee,\wedge,\rightarrow_1,\neg_1)$ is the only $3$-element {\it pBa} upto isomorphism, $\mathcal{A}'$ and $\mathcal{B}'$ are the only $3$-element {\it ccpBa}'s upto isomorphism.

\begin{figure}[hbt!]
  \begin{minipage}[b]{0.35\textwidth}
    \centering
    \begin{tabular}{ |c|c|c|c| } 
	\hline
	$\rightarrow_1$ & $0$ & $a$ & $1$ \\ \hline
	$0$ & $1$ & $1$ & $1$ \\ \hline
	$a$ & $0$ & $1$ & $1$ \\ \hline
	$1$ & $0$ & $a$ & $1$ \\ \hline
	\end{tabular}
    \captionof{table}{Implication $\rightarrow_1$}
	\label{3rpc-imp}
  \end{minipage}
  \begin{minipage}[b]{0.28\textwidth}
  \centering
    \begin{tabular}{ |c|c| } 
	\hline
	$x$ & $\neg_1 x$ \\ \hline
	$0$ & $1$ \\ \hline
	$a$ & $0$ \\ \hline
	$1$ & $0$ \\ \hline
	\end{tabular}
	\captionof{table}{Negation $\neg_1$}
	\label{3pba-neg}
  \end{minipage}
  \begin{minipage}[b]{0.32\textwidth}
  \centering
	\begin{tabular}{ |c|c| } 
	\hline
	$x$ & ${\sim}_1 x$ \\ \hline
	$0$ & $1$ \\ \hline
	$a$ & $1$ \\ \hline
	$1$ & $1$ \\ \hline
	\end{tabular}
	\captionof{table}{Negation ${\sim}_1$}
	\label{3ccpba-neg}
  \end{minipage}
%\caption{$3$-element {\it ccpBa}}
\label{3ccpBa}
\end{figure}

\noindent (B) Consider the 5-element {\it pBa} $\mathcal{H}_5 :=(H_5,1,0,\vee,\wedge,\rightarrow,\neg)$ (Figure \ref{fig-5rpcA}). There are four choices for ${\sim} 1$ such that $\neg \neg {\sim} 1 ={\sim} 1$  -- these are the elements $0,1,a$, and $b$.

\begin{figure}[hbt!]
  \begin{minipage}[b]{0.50\textwidth}
  \centering
    \centering
    \begin{tikzpicture}
		\node (A) at (1, 0) {$0$};
		\node (B) at (0, 1) {$a$};
		\node (C) at (2, 1) {$b$};
		\node (D) at (1, 2) {$e$};
		\node (F) at (1, 3) {$1$};
		\draw[->]  (A) -- (B);
		\draw[->]  (A) -- (C);
		\draw[->]  (B) -- (D);
		\draw[->]  (C) -- (D);
		\draw[->]  (D) -- (F);
	\end{tikzpicture}
    \captionof{figure}{$\mathcal{H}_5$}
    \label{fig-5rpcA}
  \end{minipage}
  \begin{minipage}[b]{0.48\textwidth}
  \centering 
  \begin{tabular}{ |c|c|c|c|c|c| } 
\hline
\hspace{1mm} $\rightarrow$ \hspace{1mm} & \hspace{0.3mm} $0$ \hspace{0.3mm} & \hspace{0.3mm} $a$ \hspace{0.3mm} & \hspace{0.3mm} $b$ \hspace{0.3mm} & \hspace{0.3mm} $e$ \hspace{0.3mm} & \hspace{0.3mm} $1$ \hspace{0.3mm} \\ \hline
$0$ & $1$ & $1$ & $1$ & $1$ & $1$ \\ \hline
$a$ & $b$ & $1$ & $b$ & $1$ & $1$ \\ \hline
$b$ & $a$ & $a$ & $1$ & $1$ & $1$ \\ \hline
$e$ & $0$ & $a$ & $b$ & $1$ & $1$ \\ \hline
$1$ & $0$ & $a$ & $b$ & $e$ & $1$ \\ \hline
  \end{tabular}
      \captionof{table}{Implication in $\mathcal{H}_5$}
      \label{tab-impl2}
    \end{minipage}
\end{figure}
\begin{enumerate}[noitemsep]
\item ${\sim} 1 := 0$, i.e. ${\sim} x = \neg x$. The resulting algebra $(H_5, 1, 0, \vee,\wedge, \rightarrow ,\neg, \neg)$ is a $ccpBa$, but not a $c{\vee}cpBa$ because $a \vee {\sim} a \neq 1$.
\item ${\sim} 1 := a$. The resulting algebra $(H_5, 1, 0, \vee,\wedge, \rightarrow ,\neg, {\sim})$ is a $ccpBa$, but again not a $c{\vee}cpBa$ because $b \vee {\sim} b \neq 1$.
\item ${\sim} 1 := b$. This case is the same as that for ${\sim} 1 := a$, the resulting algebra is a $ccpBa$, but not a $c{\vee}cpBa$.
\item ${\sim} 1 := 1$. The resulting algebra $(H_5, 1, 0, \vee,\wedge, \rightarrow ,\neg, {\sim})$ is  a %$ccpBa$ and a 
$c{\vee}cpBa$.
\end{enumerate}
\vskip 3pt
\noindent (C) For the $6$-element {\it pBa} $\mathcal{H}_6 :=(H_6,1,0,\vee,\wedge,\rightarrow,\neg)$ (Figure \ref{fig-6rpcA}), the available choices for ${\sim} 1$ are the elements $0,1,x$, and $y$.
\begin{enumerate}[noitemsep]
\item ${\sim} 1 := 0$, i.e. ${\sim} a = \neg a$. The resulting algebra $(H_6, 1, 0, \vee,\wedge, \rightarrow ,\neg, \neg)$ is a $ccpBa$, but not a $c{\vee}cpBa$ because $z \vee {\sim} z \neq 1$.
\item ${\sim} 1 := y$. The resulting algebra $(H_6, 1, 0,\vee, \wedge,\rightarrow ,\neg, {\sim})$ is a $ccpBa$, but again not a $c{\vee}cpBa$ because $z \vee {\sim} z \neq 1$.
\item ${\sim} 1 := x$. The resulting algebra $(H_6, 1, 0,\vee,\wedge,\rightarrow ,\neg, \sim)$ is a $c{\vee}cpBa$.
%, and thus a $ccpBa$.
\item ${\sim} 1 := 1$. The resulting algebra $(H_6, 1, 0,\vee, \wedge, \rightarrow ,\neg, {\sim})$ is again a %$ccpBa$ and 
$c{\vee}cpBa$.
\end{enumerate}

\subsection{Representation theorems for the algebras}
\label{representations}

%Representation theorems for the class of {\it rpc} lattices, {\it cc} lattices, and {\it pBa} can also be found in \cite{HR}. For example, for every {\it pBa} $\mathcal{A}$, there exists a monomorphism $h$ from $\mathcal{A}$ into the pseudo-field of all open subsets of a topological space. 

The representation result for {\it pBa}s in terms of pseudo-fields of open subsets of a topological space \cite{HR}, can be directly extended to  that for {\it ccpBa}s \cite{AM2}. `Contrapositionally complemented pseudo-fields' are defined for the purpose. We recall the definition and state the result.

\begin{definition}[Contrapositionally complemented pseudo-fields]$ $\\
Let $\mathscr{G}(X): = (\mathcal{G}(X), X ,\emptyset, \cap, \cup, \rightarrow,\neg)$ be a pseudo-field of open subsets of a topological space $X$. Choose and fix $Y_0 \in \mathcal{G}(X)$. Define 
\begin{align*}
{\sim} X := \neg \neg Y_0, \quad \mbox{and} \quad
{\sim} Z := Z \rightarrow (\neg \neg {\sim} X), \mbox{ for each } Z \in \mathcal{G}(X).
\end{align*}
Then the algebra $(\mathcal{G}(X), X ,\emptyset, \cap, \cup, \rightarrow, \neg,\sim)$ is called  a \textit{contrapositionally complemented pseudo-field (${\it cc}$ pseudo-field)} of open subsets of $X$.
\end{definition}

\begin{proposition}
Any ${\it cc}$ pseudo-field of open subsets of a topological space $X$ forms a $ccpBa$. Moreover, for every {\it ccpBa} $\mathcal{A}:=(A,1,0, \cap, \cup, \rightarrow,\neg, {\sim} )$, there exists a monomorphism $h$ from $\mathcal{A}$ into a {\it cc} pseudo-field of all open subsets of a topological space $X$.
\end{proposition}

In this work, we turn to  representations of $ccpBa$s and $c{\vee}cpBa$s  in terms of Priestley and Esakia spaces. 
Let us recall the basic definitions. 
%first give the concepts of upsets in any poset, filters in a distributive lattice and some basic results about them  \cite{HR}.
\vskip 3pt
%\begin{definition}[Upsets and downsets]
%\label{def-updown} \index{upset} \index{downset}
Consider a poset $(X, \leq)$. $Y( \subseteq X)$ is called an {\it upset} ({\it downset}) if for all $x \in Y$ and $y \in X$, $x \leq y$ ($y \leq x$) implies $y \in Y$. Let $Up(X)$ denote the set of all upsets of $X$.
For  $Z (\subseteq X)$,\\
\centerline{$\uparrow Z :=\{x \in X~|~ \mbox{ there exists } y \in Z \mbox{ satisfying } y \leq x \}$, \mbox{ and }}
\centerline{$\downarrow Z :=\{x \in X~|~ \mbox{ there exists } y \in Z \mbox{ satisfying } x \leq y \}$}
are the {\it upset and downset generated by $Z$} respectively. 
%\end{definition}

%\noindent For a non-empty subset $B \subseteq A$, the filter generated by $B$  is denoted  $\langle B \rangle$.
%
%\begin{lemma}{\rm \cite{HR}}
%\label{lem-filter}
%The following hold for any  filter $F$ in a distributive lattice $\mathcal{A}:=(A, \vee,\wedge)$.\\
%{\rm (1).} $F$ is an upset in $A$.\\
%{\rm (2).} If $F$ is maximal then $F$ is prime.\\
%{\rm (3).} For all $a \in A$, $b \in \langle F \cup\{a\} \rangle$ if and only if there exists $b' \in F$ such that $b' \wedge a \leq b$.\\
%{\rm (4).} The property in {\rm (3)} can be extended: for all filters $F_1,F_2$ in $\mathcal{A}$, $x \in \langle P \cup Q \rangle$ if and only if there are $y \in P$, $z \in Q$ such that $y \wedge z \leq x$.\\
%{\rm (5).} For all $a, b \in A$ such that $a \nleq b$, there exists a prime filter $Q$ in $\mathcal{A}$ such that $a \in Q$ and $b \notin Q$.\\
%{\rm (6).} The property in {\rm (5)} can be extended: let $F$ be proper. Suppose  $\Gamma \subseteq A$ is such that $\Gamma \cap F = \emptyset$, and $\Gamma$ is $\vee$-closed, i.e. for all $a, b \in \Gamma$, $a \vee b \in \Gamma$. Then there exists a prime filter $P$ in $\mathcal{A}$ such that $F \subseteq P$ and $\Gamma \cap P = \emptyset$.
%\end{lemma}

\begin{definition}[Priestley and Esakia spaces] {\rm (cf. \cite{davey2002,celani2014})}
\label{priestley} \index{Priestley space} \index{Esakia space} A {\it Priestley space} is a tuple $(X,\tau,\leq)$, where $(X,\leq)$ is a poset and $\tau ~(\neq \emptyset)$ is a compact topological space on $X$ satisfying the following property: for every $x,y \in X$, if $x \nleq y$, then there exists a {\rm clopen}  (closed, as well as open) upset $Y$ of $X$ such that $x \in Y$ and $y \notin Y$.\\
Additionally, if a Priestley space $(X,\tau,\leq)$ satisfies the property that for any $U \subseteq X$, $U$ is clopen implies $\downarrow U$ is clopen, then it is called an {\it Esakia space}. 
\end{definition}

Let $CpUp(X)$ be the set of clopen upsets of $\tau$ in a Priestley space $(X,\tau,\leq)$. Then $\mathcal{D}(X):=(CpUp(X),X,\emptyset,\cup, \cap)$ forms a bounded distributive lattice. Define the operator $\rightarrow$ on $CpUp(X)$, for any $U,V \in CpUp(X)$:
\begin{align}
\label{eq-ra}
U \rightarrow V:= X {\setminus}  \downarrow (U {\setminus} V).
\end{align}
The operator $\rightarrow$ is closed in $CpUp(X)$ if and only if $(X,\tau,\leq)$ is an Esakia space (cf. \cite{celani2014}). In this case, $\mathcal{D}(X):=(CpUp(X),X,\emptyset,\cup,\cap,\rightarrow,\neg)$ forms a {\it pBa}, where $\neg U:=U \rightarrow \emptyset$.\\
Now consider a {\it pBa} $\mathcal{A}:=(A,1,0,\vee,\wedge,\rightarrow,\neg)$, and let $X_{A}$ denote the set of prime filters in $\mathcal{A}$. Define the topology $\tau_{\mathcal{A}}$ on $X_{A}$ generated by the subbasis\\
\centerline{$\{\sigma(a)~|~ \sigma(a) \subseteq X_{A}~ \&~ a \in A\} \cup \{X_A {\setminus} \sigma(a)~|~ \sigma(a) \subseteq X_{A}~ \&~  a \in A\},$}
where $\sigma(a)$ is the set of prime filters containing $a \in A$. Then $(X_{A},\tau_{\mathcal{A}},\subseteq)$ forms an Esakia space. For $a \in A$, $\sigma(a)$ and $X_{A} {\setminus} \sigma(a)$ are the only clopen upsets in $\tau_{\mathcal{A}}$. These definitions lead to the following representation result for {\it pBa}s (cf. \cite{nick2006}).

\begin{theorem}[Duality for \textit{pBa}s] {\rm \cite{nick2006}}
\label{dual-pba}
\begin{enumerate}[label={{\rm (\arabic*)}}, leftmargin=1.4 \parindent, noitemsep, topsep=0pt]
\item Given any {\it pBa} $\mathcal{A}:=(A,1,0,\vee,\wedge,\rightarrow,\neg)$, there exist an Esakia space $(X_{A},\tau_{\mathcal{A}},\subseteq)$ and a {\it pBa} $\mathcal{D}(X_A):=(CpUp(X_A),X_A,\emptyset,\cup,\cap,\rightarrow,\neg)$ such that $\mathcal{A}$ is isomorphic to $\mathcal{D}(X_{A})$, through the map $\Phi:A \rightarrow CpUp(X_{A})$ defined as $\Phi (a) := \sigma (a)$, for any $a \in A$.

\item Given any Esakia space $(X,\tau,\leq)$, there exist a {\it pBa} $\mathcal{D}(X):=$\\
$(CpUp(X),X,\emptyset,\cup,\cap,\rightarrow,\neg)$ and an Esakia space $(X_{CpUp(X)},\tau_{\mathcal{D}(X)},\subseteq)$ such that $\tau$ is homeomorphic to $\tau_{\mathcal{D}(X)}$ and the poset $(X,\leq)$ is order-isomorphic to the poset $(X_{CpUp(X)},\subseteq)$.
\end{enumerate}
\end{theorem}

The above duality result can be extended to $ccpBa$s and $c{\vee}cpBa$s. Consider a $ccpBa$ $\mathcal{A}:=(A,1,0, \vee,\wedge, \rightarrow,\neg,{\sim})$. Observation \ref{obs-ccpba} implies the reduct $(A,1,0,\vee,\wedge,\rightarrow,\neg)$ is a {\it pBa}.  Using Theorem \ref{dual-pba}, we have an Esakia space $(X_{A},\tau_{\mathcal{A}},\subseteq)$ such that the {\it pBa} $(A,1,0,\vee,\wedge,\rightarrow,\neg)$ is isomorphic to  the {\it pBa} $\mathcal{D}(X_A):=(CpUp(X_A),X_A,\emptyset,\cap,\cup,\rightarrow,\neg)$, through the map $\Phi$.
%Here, the isomorphism is given by the map $\Phi:A \rightarrow CpUp(X_{A})$ defined as $\Phi (a) := \sigma (a)$, for any $a \in A$ (cf. \cite{celani2014}).
Define $Y_0:=\sigma({\sim} 1)$. Since $\Phi$ is a homomorphism,\\
\centerline{$\sigma ({\sim} 1) = \Phi ({\sim} 1) = \Phi (\neg \neg {\sim} 1) = \neg \neg \Phi ({\sim} 1) = \neg \neg \sigma ({\sim} 1)$}
giving $Y_0 = X_A {\setminus} \downarrow (X_A {\setminus} \downarrow Y_0)$. In particular, we have $X_A {\setminus} \downarrow (X_A {\setminus} \downarrow Y_0) \subseteq Y_0$. Expanding this,
\begin{align*}
&\forall x \in X_A (x \notin \ \downarrow (X_A {\setminus} \downarrow Y_0)) \Rightarrow x \in Y_0) \\
&\forall x \in X_A (\forall y \in X_A (x \subseteq y \Rightarrow y \in \ \downarrow Y_0) \Rightarrow x \in Y_0) \\
&\forall x \in X_A (\forall y \in X_A (x \subseteq y \Rightarrow \exists z \in X_A (y \subseteq z ~ \& ~ z \in Y_0)) \Rightarrow x \in Y_0).
\end{align*}
Thus, starting from a {\it ccpBa} $\mathcal{A}$, we have obtained an Esakia space $(X_{A},\tau_{\mathcal{A}},\subseteq)$ satisfying
\begin{align}
\label{eq-E1}
\forall x \in X_A (\forall y \in X_A (x \subseteq y \Rightarrow \exists z \in X_A (y \subseteq z ~ \& ~ z \in Y_0)) \Rightarrow x \in Y_0).
\end{align}
\vskip 3pt

\noindent Conversely, consider an Esakia space $(X,\tau,\leq)$ and $Y_0 \in X$ such that $Y_0$ is a clopen set in $\tau$ satisfying
\begin{align}
\label{eq-E1-1}
\forall x \in X (\forall y \in X (x \leq y \Rightarrow \exists z \in X (y \leq z ~ \& ~ z \in Y_0)) \Rightarrow x \in Y_0).
\end{align}
Then, by Theorem \ref{dual-pba},  we have the {\it pBa} $\mathcal{D}(X):=(CpUp(X),X,\emptyset,\cup,\cap,\rightarrow,\neg)$.\\
Now, define ${\sim} X := Y_0$ and ${\sim} U := U \rightarrow {\sim} X$ for all $U \in CpUp(X)$, thus making  $(CpUp(X),X,\emptyset,\cup,\cap,\rightarrow,{\sim})$ a {\it cc} lattice. Any {\it pBa} has the property $a \leq \neg \neg a$, for all $a \in A$. Thus, ${\sim} X \subseteq \neg \neg {\sim} X$. Moreover, it can be observed that (\ref{eq-ra}) and (\ref{eq-E1-1}) imply $\neg \neg {\sim} X \subseteq {\sim} X$. Therefore,  ${\sim} X = \neg \neg {\sim} X$. By Proposition \ref{prop3-ccpBa}(2), $(CpUp(X),X,\emptyset,\cup,\cap,\rightarrow,\neg,{\sim})$ is a  {\it ccpBa}. \\
For simplicity, we use the same notation $\mathcal{D}(X)$ for the {\it pBa} and the {\it ccpBa}.
\vskip 3pt

For {\it c${\vee}$cpBa}s, note that the property $a \vee {\sim} a  = 1$ corresponds to the following condition
\begin{align}
\label{eq-E2-1}
\forall x,y  \in X(x,y \notin Y_0 \Rightarrow (x\leq y \Rightarrow y \leq x))
\end{align}
in the context of logic and relational semantics \cite{Segerberg1968}. This correspondence can be replicated here to get the dual topological spaces for {\it c${\vee}$cpBa}s. If $\mathcal{A}$ is a $c{\vee}cpBa$, we have $\sigma(a) \cup (\sigma(a) \rightarrow Y_0) =X_A$ for all $a \in A$, i.e. 
\begin{align}
\label{eq-E3}
\sigma(a) \cup X_A {\setminus} \downarrow (\sigma (a) {\setminus} Y_0) =X_A \mbox{ for all } a\in A.
\end{align}
In this case, the Esakia space $(X_A,\tau_A,\subseteq)$ satisfies
\begin{align}
\label{eq-E2}
\forall x,y  \in X_A(x,y \notin Y_0 \Rightarrow (x\subseteq y \Rightarrow y \subseteq x)).
\end{align}
For this, we have to show that for any two prime filters $F,G \in X_A$, if $F,G \notin Y_0$ and $F\subseteq G$ then $G \subseteq F$. Suppose not, i.e.  there exist two prime filters $F,G$ such that $F,G \notin Y_0$, $F\subseteq G$ and there exists $x \in G$ such that $x \notin F$. Thus $F \notin \sigma (x)$ and $G \in \sigma (x)$. In Condition (\ref{eq-E3}), since $F \in X_A$, we must have $F \in \sigma(x) \cup X_A {\setminus} \downarrow (\sigma (x) {\setminus} Y_0)$. We already have $F \notin \sigma (x)$. Therefore $F \in X_A {\setminus} \downarrow (\sigma (x) {\setminus} Y_0)$, i.e. $F \notin {\downarrow (\sigma (x) {\setminus} Y_0)}$. Since $F \subseteq G$, using the definition of $\downarrow (\sigma (x) {\setminus} Y_0)$, we have $G \notin \sigma(x) {\setminus} Y_0$, i.e. $G \notin \sigma(x)$ (because $G \notin Y_0$), a contradiction.
\vskip 3pt

\noindent Conversely, consider an Esakia space $(X,\tau,\leq)$ and $Y_0 \in X$ such that $Y_0$ is a clopen set satisfying Conditions (\ref{eq-E1-1}) and (\ref{eq-E2}). 
%We have already shown that  ${\sim} X = \neg \neg {\sim} X$.
%$Y_0$ satisfies Condition \ref{eq-E2}. 
We have to only show that for any $V \in CpUp(X)$, $V \cup (V \rightarrow {\sim} X) = X$, i.e. $V \cup (X {\setminus} \downarrow (V {\setminus} Y_0 )) = X$. Suppose not, i.e. there exists $x \in X$ such that $x \notin V$ and $x \in {\downarrow (V {\setminus} Y_0)}$. Then there exists $y \in X$ such that $x \leq y$ and $y \in V {\setminus} Y_0$. We have $x \leq y$, $y \in V$ and $y \notin Y_0$, i.e. $x \notin Y_0$ (because $Y_0$ is an upset). Using $x \leq y$ and Condition (\ref{eq-E2}), we have $y \leq x$. Since $y \in V$ and $V$ is an upset, we have $x \in V$, a contradiction.

We thus obtain 
\begin{theorem}[Duality for \textit{ccpBa}s and \textit{c${\vee}$cpBa}s]
\label{dual-ccpba}
$ $
%The class of $c{\vee}cpBa$s is dual to  ordered topological spaces of the above kind that additionally satisfy the following:
\begin{enumerate}[label={{\rm (\arabic*)}}, leftmargin=1.3 \parindent, topsep=0pt]
\item Given any {\it ccpBa} ({\it c${\vee}$cpBa}) $\mathcal{A}:=(A,1,0, \vee,\wedge, \rightarrow,\neg,{\sim})$, we have the following.
\begin{enumerate}[label={{\rm (\alph*)}}, leftmargin=1.3 \parindent, noitemsep, topsep=0pt]
\item The tuple $(X_A,\tau_{\mathcal{A}},\subseteq)$ forms an Esakia space; $Y_0:=\sigma({\sim} 1)$ is a clopen set  in $\tau_{\mathcal{A}}$ satisfying Condition \ref{eq-E1} (Conditions \ref{eq-E1} and \ref{eq-E2}).
\item The tuple $\mathcal{D}(X_A):=(CpUp(X_A),X_A,\emptyset,\cup,\cap,\rightarrow,\neg,{\sim})$ is a {\it ccpBa} ({\it c${\vee}$cpBa}).
\item $\mathcal{A}$ is isomorphic to $\mathcal{D}(X_A)$, through the map $\Phi:A \rightarrow CpUp(X_{A})$ defined as $\Phi (a) := \sigma (a)$, for any $a \in A$.
\end{enumerate}
\item Given any ordered topological space $(X,\tau,\leq,Y_0)$ such that $(X,\tau,\leq)$ is an Esakia space and $Y_0$ is a clopen set satisfying Condition {\rm (\ref{eq-E1-1})} (Conditions {\rm (\ref{eq-E1-1})} and {\rm (\ref{eq-E2-1})}), we have the following.
\begin{enumerate}[label={{\rm (\alph*)}}, leftmargin=1.3 \parindent, noitemsep, topsep=0pt]
\item The tuple $\mathcal{D}(X):=(CpUp(X),X,\emptyset,\cup,\cap,\rightarrow,\neg,{\sim})$ forms a {\it ccpBa} ({\it c${\vee}$cpBa}).
\item $(X_{CpUp(X)},\tau_{\mathcal{D}(X)},\subseteq)$ is an Esakia space.\\
Moreover, $\widetilde{Y_0}$ defined as the set of all prime filters containing ${\sim}X$ in $\mathcal{D}(X)$, satisfies the following.\\
$\forall x \in X_{CpUp(X)} (\forall y \in X_{CpUp(X)} (x \subseteq y \Rightarrow$\\
\null \hfill $\exists z \in X_{CpUp(X)} (y \subseteq z ~ \& ~ z \in \widetilde{Y_0})) \Rightarrow x \in \widetilde{Y_0}).$\\
If $Y_0$ satisfies Condition {\rm (\ref{eq-E2-1})} then $\widetilde{Y_0}$ satisfies\\
\null \hspace{3cm}$\forall x,y  \in X_{CpUp(X)}(x,y \notin \widetilde{Y_0} \Rightarrow (x\subseteq y \Rightarrow y \subseteq x)).$
%\begin{align*}
%\forall x \in X_{CpUp(X)} (\forall y \in X_{CpUp(X)} (x \subseteq y \Rightarrow & \\
%& \hspace{-15mm} \exists z \in X_{CpUp(X)} (y \subseteq z ~ \& ~ z \in \widetilde{Y_0})) \Rightarrow x \in \widetilde{Y_0}).
%\label{eq-E2-3}
%\end{align*}
%If $Y_0$ satisfies Condition {\rm (\ref{eq-E2-1})} then $\widetilde{Y_0}$ satisfies
%\begin{align*}
%\forall x,y  \in X_{CpUp(X)}(x,y \notin \widetilde{Y_0} \Rightarrow (x\subseteq y \Rightarrow y \subseteq x)).
%\end{align*}
\item $\tau$ is homeomorphic to $\tau_{\mathcal{D}(X)}$ and $(X,\leq)$ is order-isomorphic to $(X_{CpUp(X)},\subseteq)$.
\end{enumerate}
\end{enumerate}
\end{theorem}
\begin{proof}$ $\\
(1) From Theorem \ref{dual-pba}(1), we already have that $\Phi$ is an isomorphism between the underlying {\it pBa}s - $(A,1,0, \vee, \wedge, \rightarrow,\neg)$ and $(CpUp(X_A),X_A,\emptyset,\cup,\cap,\rightarrow,\neg)$. We have to only show $\Phi({\sim} x) = {\sim}\Phi(x)$ for all $x \in A$. Indeed, $\Phi ({\sim} x) = \Phi(x \rightarrow {\sim} 1) = \Phi(x) \rightarrow \Phi ({\sim} 1) = \sigma(x) \rightarrow Y_0 = \sigma(x) \rightarrow {\sim} X_A = {\sim} \sigma (x)  = {\sim} \Phi(x)$.\\
(2) This is direct from Theorem  \ref{dual-pba}(2) and the discussion following it.
\end{proof}

\section{Intuitionistic logic with minimal negation}
\label{seclogic}

%We saw two new algebraic structures, namely {\it ccpBa} and {\it c$\vee$cpBa}, in the previous section. 

In this section, we present the logic corresponding to the class of {\it ccpBa}s, called `Intuitionistic logic with minimal negation' ({\rm ILM}) \cite{AM2}.
%An extension of ILM gives {\rm ILM}-{$\vee$}, the logic for the class of {\it c${\vee}$cpBa}s.
%In this section, we shall study this logic.
As is well-known,  positive logic (PL), minimal logic (ML), intuitionistic  logic (IL) and classical logic (CL)  
%Note that {\rm PL}, {\rm ML}, {\rm IL} and {\rm CL} 
are sound and complete with respect to the class of {\it rpc} lattices, {\it cc} lattices, {\it pBa}s and Boolean algebras respectively \cite{HR,odintsov2008}. Since the ${\it ccpBa}$ is an amalgamation of {\it cc} lattice and {\it pBa}, it is then expected that {\rm ILM} will be defined using IL and ML. We shall use the terminology and axiomatization of {\rm ML} and {\rm IL} as given in \cite{HR}.

\subsection{ILM and ILM-\texorpdfstring{${\vee}$}~}
\label{secILM}

\begin{definition}[Intuitionistic logic with minimal negation (\textrm{ILM})]
\label{ILM}
\index{intuitionistic logic with minimal negation (ILM)}
\index{ILM-{$\vee$}}
{\rm \cite{AM2}} {\rm The alphabet of the language $\mathcal{L}$ of {\rm ILM} is that of {\rm IL}, consisting of propositional constants $\top$ and $\bot$, a set PV of propositional variables, and logical connectives $\rightarrow$ (implication), $\vee$ (disjunction), $\wedge$ (conjunction), ${\neg}$ (negation). Additionally, there is a unary connective ${\sim}$. The formulas are given by the scheme:
\[ \top \mid \bot \mid p \mid \alpha \wedge \beta \mid \alpha \vee \beta \mid \alpha \rightarrow \beta \mid \neg \alpha \mid {\sim} \alpha \]
where $p \in$ PV. The set   of all formulas is denoted by $F$.
\vspace{3mm}

\noindent {\it Axioms}:

\begin{enumerate}[label={\rm{(A}$\arabic*)$}, leftmargin=3 \parindent, noitemsep]
\item $\alpha \rightarrow (\beta \rightarrow \alpha)$
\item $(\alpha \rightarrow (\beta \rightarrow \gamma)) \rightarrow ((\alpha \rightarrow \beta) \rightarrow (\alpha \rightarrow \gamma))$
\item (i) $\alpha \rightarrow (\alpha \vee \beta)$, (ii) $\beta \rightarrow (\alpha \vee \beta)$
\item $(\alpha \rightarrow \gamma) \rightarrow ((\beta \rightarrow \gamma) \rightarrow ((\alpha \vee \beta) \rightarrow \gamma))$
\item (i) $(\alpha \wedge \beta) \rightarrow \alpha$, (ii) $(\alpha \wedge \beta) \rightarrow \beta$
\item $(\alpha \rightarrow \beta) \rightarrow ((\alpha \rightarrow \gamma) \rightarrow (\alpha \rightarrow (\beta \wedge \gamma)))$
\item $\alpha \rightarrow \top$
\item $\bot \rightarrow \alpha$ 
\item $(\alpha \rightarrow \beta) \rightarrow ((\alpha \rightarrow \neg \beta) \rightarrow \neg \alpha)$
\item $\neg \alpha \rightarrow (\alpha \rightarrow \beta)$
\item ${\sim} \alpha \leftrightarrow (\alpha \rightarrow \neg \neg {\sim} \top)$
\end{enumerate}
Modus ponens {\rm (MP)} is the only rule of inference in {\rm ILM}. The deduction procedure for ILM is specified in the usual manner, to give the relation of syntactic consequence $\vdash_{{\rm ILM}}$  and define $\Gamma \vdash_{{\rm ILM}} \alpha$, for any $\Gamma \cup \{\alpha\} \subseteq F$.
\vskip 3pt
\noindent Addition of the following axiom gives the logic {\rm ILM}-{$\vee$}.
\begin{enumerate}[label={\rm{(A}$\arabic*)$}, leftmargin=3 \parindent, noitemsep]
\setcounter{enumi}{11}
\item $\alpha \vee {\sim} \alpha$
\end{enumerate}}
\end{definition}

\noindent Observe that axiom (A11) connects  the negations $\neg$ and ${\sim}$.

\subsubsection{Properties of {\rm ILM} and comparison with {\rm IL} and {\rm ML}}
\label{sec-embed}
It can be shown that the deduction theorem (DT) holds for {\rm ILM}.
Let us now give the algebraic semantics for the logics. 

Consider any {\it ccpBa} $\mathcal{A}
:=(A,1,0, \vee, \wedge, \rightarrow,\neg,{\sim} )$. 
A {\it valuation} is a map $v$ from  PV to $A$, and   can  be extended to  $F$  in the standard way \cite{HR}. For a formula $\alpha \in F$, if for all valuations $v$ on $\mathcal{A}$, $v(\alpha)=1$, then we say $\alpha$ is valid in $\mathcal{A}$ (denoted as $\vDash_{\mathcal{A}} \alpha$). In the classical manner, we obtain

\begin{theorem}[{\bf Algebraic semantics}]
\label{soundcomplete}
\index{intuitionistic logic with minimal negation (ILM)!sound}
\index{intuitionistic logic with minimal negation (ILM)!complete}
For any $\alpha \in F$, $\vdash_{\mathrm{ILM}} \alpha$ ($\vdash_{\mathrm{ILM}-{\vee}} \alpha$) if and only if $\vDash_{\mathcal{A}} \alpha$ for every {\it ccpBa} ({\it c${\vee}$cpBa}) $\mathcal{A}$.
\end{theorem}

\noindent Thus, using the soundness part of the above theorem, Proposition \ref{prop1-ccpBa} gets a version in terms of ILM-formulas.

\begin{proposition}
\label{prop-ilm}~
\begin{multicols}{2}
\begin{enumerate}[label={{\rm (\alph*)}}, leftmargin=1.5 \parindent, noitemsep, topsep=0pt]
\item $\vdash_{{\rm ILM}} {\neg} \alpha \leftrightarrow (\alpha \rightarrow \bot)$
\item $\vdash_{{\rm ILM}} {\neg} {\neg} {\sim} {\top} \leftrightarrow {\sim} \top$
\item $\vdash_{{\rm ILM}} {\sim} \alpha \leftrightarrow (\alpha \rightarrow {\sim} \top)$
\item $\vdash_{{\rm ILM}}{\sim}{\sim}({\sim} \top \rightarrow \alpha)$
\item $\vdash_{{\rm ILM}}{\sim} \alpha \leftrightarrow \neg (\alpha \wedge \neg {\sim} \top)$
\item $\vdash_{{\rm ILM}} \neg \alpha \rightarrow {\sim} \alpha$
\item $\vdash_{{\rm ILM}} \alpha \rightarrow {\sim} \neg \alpha$
\item $\vdash_{{\rm ILM}} \neg {\sim} \alpha \rightarrow {\sim} \neg \alpha$
\item $\vdash_{{\rm ILM}}  {\sim} \alpha \leftrightarrow \neg \neg {\sim} \alpha$
\item $\vdash_{{\rm ILM}} \neg {\sim} \neg \alpha \rightarrow \neg \alpha$
%\item $\vdash_{{\rm ILM}} (({\sim} {\top}\rightarrow \bot) \rightarrow \bot) \leftrightarrow {\sim} \top$
\end{enumerate}
\end{multicols}
\end{proposition}
\begin{proof}
As $\neg a = a\rightarrow 0$ in any ${\it ccpBa}$, Theorem \ref{soundcomplete} gives (a). Observation \ref{obs-ccpba}(2) and Theorem \ref{soundcomplete} give (b). (c) is obtained using (A11) and (b).  Proposition \ref{prop1-ccpBa} (3)-(9) and Theorem \ref{soundcomplete} give (d)-(j) respectively.
%Using Proposition \ref{prop1-ccpBa} (4)-(10) and Theorem \ref{soundcomplete} again, (d)-(h) are direct. 
\end{proof}

\begin{remark}
\label{negations}
\cite{HR} The formulas in (a) and (c) in the above Proposition are the defining conditions for the negations ${\neg}$ and ${\sim}$ to be intuitionistic and {\rm minimal} respectively. Intuitionistic negation can also be equivalently defined using formulas (A9) and (A10).
\end{remark}

Recall that for any two logics ${\rm L}$ and ${\rm L'}$ with sets of formulas $F$ and $F'$ respectively, 
%let $F$ and $F'$ be the sets of formulas in the languages of ${\rm L}$ and ${\rm L'}$ respectively. 
${\rm L}$ is {\it embedded} in ${\rm L}'$ 
%({\rm notation} - ${\rm L} \prec {\rm L}'$) 
if there exists a map $r:F \rightarrow F'$ such that for any $\alpha \in F$, $\vdash_{{\rm L}} \alpha$ if and only if  $\vdash_{{\rm L'}} r(\alpha)$ \cite{HR}. In case $r$ is the inclusion map, ${\rm L'}$ is called an {\it extension} of ${\rm L}$. ${\rm L}$ is {\it equivalent} to ${\rm L'}$, denoted as ${\rm L} \cong {\rm L}'$, if the languages of ${\rm L}$ and ${\rm L'}$ are the same, that is $F=F'$, and $r$ is the identity map (cf. \cite{odintsov2008}). Another way to show equivalence between two logics over the same language is by comparing their corresponding classes of algebras: ${\rm L}$ and ${\rm L}'$ are equivalent if every ${\rm L}$-algebra is isomorphic to some ${\rm L}'$-algebra and conversely \cite{HR}. It is shown in \cite{HR} that both the definitions of equivalence coincide. It is clear that

\begin{proposition} {\rm \cite{AM2}}
{\rm ML} and {\rm IL} are both embedded in {\rm ILM}.
\end{proposition}

\noindent Comparison of two logics is also done by giving an `interpretation' between them \cite{PM}. We use a more general definition of interpretation than that given in \cite{PM}:
%namely that,
%and not require an interpretation to be `schematic'. 
a map $r:F \rightarrow F'$ is  an {\it interpretation of ${\rm L}$  in ${\rm L}'$ with respect to derivability}, if for any set $\Gamma \cup \{\alpha\} \subseteq F$, we have $\Gamma \vdash_{{\rm L}} \alpha \textrm{ if and only if } r(\Gamma) \cup \Delta_{\alpha} \vdash_{{\rm L}'} r(\alpha)$,
where  $\Delta_{\alpha} \subseteq F'$ is a finite set corresponding to $\alpha$. This yields
%If $\Gamma$ is empty, $r$ is just an {\rm interpretation of ${\rm L}$  in ${\rm L}'$}.

\begin{proposition} {\rm \cite{AM2}}
There exist  interpretations
%$t: F \rightarrow \bar{F}$ 
of $\mathrm{ILM}$ in $\mathrm{IL}$ and in $\mathrm{ML}$.
\end{proposition}
 
As observed in case of {\it ccpBa}s  (Observation \ref{obs-ccpba}), we get an equivalent axiomatization of {\rm ILM} as follows.
%\newpage
\begin{theorem}[The logic $\mathbf{ILM_1}$] 
\label{ILM-ILM1}
Consider a logic ${\rm ILM_1}$ over the language $\mathcal{L}$, with axioms {\rm (A1)-(A10)}, ${\sim} \alpha \leftrightarrow (\alpha \rightarrow {\sim} \top)$, and $\neg \neg {\sim} \top \leftrightarrow {\sim} \top$. {\rm MP} is the only rule of inference. Then ${\rm ILM} \cong {\rm ILM_1}$.
%\centerline{${\rm ILM} \cong {\rm ILM_1}$.}
\end{theorem}
\begin{proof}
Axioms ${\sim} \alpha \leftrightarrow (\alpha \rightarrow {\sim} \top)$ and $\neg \neg {\sim} \top \leftrightarrow {\sim} \top$  of ${\rm ILM_1}$ imply that (A11) is an ${\rm ILM_1}$-theorem. On the other hand, we have $\vdash_{{\rm ILM}} \neg \neg {\sim} \top \leftrightarrow {\sim} \top$ and $\vdash_{{\rm ILM}} {\sim} \alpha \leftrightarrow (\alpha \rightarrow {\sim} \top)$ (Proposition \ref{prop-ilm}(b) and (c)).
\end{proof}
\noindent We shall use this equivalent version ${\rm ILM_1}$ of ILM while discussing the two relational semantics for ILM in Sections \ref{sec-KripkeILM} and \ref{subsec-KripkeILM}.

\subsubsection{Comparison of \textrm{ILM} with Peirce's logic \texorpdfstring{$\textrm{JP}'$}~}
\label{sec-JP'}

Consider the language $\mathcal{L}'$ with formulas given by the scheme:
\[  \top \mid p \mid \alpha \vee \beta \mid \alpha \wedge \beta \mid \alpha \rightarrow \beta \mid {\sim} \alpha \]
The axioms of {\rm ML} are given as (A1)-(A7) of {\rm ILM} and
\begin{enumerate}[label={{\rm (A}$\arabic*)$}, leftmargin=2 \parindent, noitemsep]
\setcounter{enumi}{12}
\item $(\alpha \rightarrow \beta) \rightarrow ((\alpha \rightarrow {\sim} \beta) \rightarrow {\sim} \alpha)$, \hfill (${\sim}$ reductio ad absurdum)
\end{enumerate}
and {\rm MP} as the rule of inference. In \cite{Segerberg1968}, Segerberg defined various extensions of {\rm ML}. One of these is the system ${\rm JP}$, which is obtained by adding Peirce's law $({\rm P})$ as an axiom to {\rm ML}:
\[ {\rm (P)} \quad ((\alpha \rightarrow \beta) \rightarrow \alpha) \rightarrow \alpha \]
Let us consider the logic ${\rm JP'}$ defined as below.

\begin{definition}[$\bm{{\rm JP'}}$ \cite{Segerberg1968}]
\label{JP'}
The language of ${\rm JP'}$ is $\mathcal{L}'$. The axioms are $({\rm A}1)-({\rm A}7)$, $({\rm A}13)$ and
\[ {\rm (P')} \quad  {\sim} {\sim} ({\sim} \top \rightarrow \beta). \]
{\rm MP} is the only rule of inference.
\end{definition}

\noindent One may remark here that in \cite{odintsov2008,Segerberg1968}, while defining ${\rm (P')}$, a propositional constant `$\bot$' that is logically equivalent to ${\sim}\top$ is used in place of ${\sim}\top$.  We have used ${\sim}\top$ instead, in order not to confuse with the propositional  constant $\bot$ that is already present in ILM, and which is {\it not} logically equivalent to ${\sim}\top$ -- as is established by the proof of Theorem \ref{prop-JP'ILM} below. 

Using the ML-theorem $\vdash_{\rm ML} {\sim} \alpha \leftrightarrow (\alpha \rightarrow {\sim} \top)$, one observes that ${\rm (P')}$ is logically equivalent to the following formula (which we also refer to as ${\rm (P')}$)
%\begin{align}
%\label{eq-ILM2}
\[(({\sim} \top \rightarrow \beta) \rightarrow {\sim} \top) \rightarrow {\sim} \top.\]
%\end{align}
Thus, $({\rm P}')$ may be regarded as a special case of $({\rm P})$ where $\alpha$ is ${\sim} \top$. Let us now compare the logics ${\rm JP}'$ and ${\rm ILM}$. 
\begin{theorem}
\label{thm-JP}
{\rm ILM} is an extension of ${\rm JP'}$.
%${\rm JP'}$ is embedded in ${\rm ILM}$.
\end{theorem}
\begin{proof} ${\rm (P')}$ is a theorem in {\rm ILM} (Proposition \ref{prop-ilm}(d)). Axioms (A1)-(A7) are common to ${\rm JP}'$ and ${\rm ILM}$. Let us see the proof of (A13) in ${\rm ILM}$. Given (A11), (A13) can be equivalently expressed as $(\alpha \rightarrow \beta) \rightarrow ((\alpha \rightarrow (\beta\rightarrow \neg \neg {\sim} \top)) \rightarrow (\alpha\rightarrow \neg \neg {\sim} \top)).$
Using MP, one can derive $\{\alpha \rightarrow \beta, \alpha \rightarrow (\beta\rightarrow \neg \neg {\sim} \top), \alpha\} \vdash_{{\rm ILM}} \neg \neg {\sim} \top.$
Applying DT, we obtain the above equivalent expression for (A13)  as a theorem in {\rm ILM}.
\end{proof}

\noindent We shall show in Theorem \ref{prop-JP'ILM} below that ${\rm JP}'$ is however, {\it not equivalent to} ${\rm ILM}$.  This is established  by comparing the classes of algebras corresponding to the respective logics. 
%Recall that logics ${\rm L}$ and ${\rm L}'$ over the same language are {\it equivalent} if every ${\rm L}$-algebra is isomorphic to some ${\rm L}'$-algebra and conversely \cite{HR}. It is shown in \cite{HR} that both the definitions of equivalence coincide. It is clear that
%Let ${\rm L}$-algebras and ${\rm L}'$-algebras be the corresponding classes of algebras for ${\rm L}$ and ${\rm L}'$ respectively. To show that ${\rm L}$ is {\it not equivalent} to ${\rm L}'$, one needs to find an ${\rm L}$-algebra (or ${\rm L}'$-algebra) $\mathcal{A}$ such that $\mathcal{A}$ is not isomorphic to any ${\rm L}'$-algebra (or ${\rm L}$-algebra). 
Theorem \ref{soundcomplete} implies that any ${\rm ILM}$-algebra is just  a {\it ccpBa}. For ${\rm JP'}$, the corresponding class of algebras is given by  
%the class of Let us formally define a ${\rm JP'}$-algebra.

\begin{definition}[$\mathbf{JP'}$-algebra]
A ${\it JP'}$-algebra $(A,1,\vee,\wedge,\rightarrow,{\sim})$ is a ${\it cc}$ lattice satisfying ${\sim}{\sim}({\sim} 1 \rightarrow a) =1$ for all $a \in A$.
\end{definition}
\noindent Note that ${\rm JP'}$-algebra is just the structure mentioned in (3) of  Proposition \ref{prop3-ccpBa}.
%\noindent The logic ${\rm JP'}$ is sound and complete with respect to the class of  ${\rm JP'}$-algebras. 
%We have 

\begin{theorem}
\label{prop-JP'ILM}
${\rm JP}'$ is not equivalent to ${\rm ILM}$. 
\end{theorem}
\begin{proof}
Suppose we had the equivalence. Then, the languages of ${\rm JP}'$ and ${\rm ILM}$ would be the same, the axioms of ${\rm JP}'$ would be  theorems in {\rm ILM}, and the axioms of {\rm ILM}  theorems in ${\rm JP}'$. This would imply that we can define the connective $\neg$ and propositional constant $\bot$ in $\mathcal{L}'$ such that (A8), (A9) and (A10)  are theorems in ${\rm JP}'$. In particular, $\vdash_{\rm JP'} \bot \rightarrow \alpha$.\\
Now consider any ${\rm JP'}$-algebra $(A,1,\vee,\wedge,\rightarrow,{\sim})$. By soundness, any such algebra will have a bottom element $0$ and an operator $\neg$ such that $\neg a = a \rightarrow \neg 1$ for all $a \in A$. However, in Example \ref{eg4-ccpBa}, we have encountered a ${\rm JP}'$-algebra $(L,0,\vee,\wedge,\rightarrow,{\sim})$ that does not have a bottom element and thus can never be extended to a {\it ccpBa}.
\end{proof}

%Note that the above example establishes that $\bot$ is not equivalent to ${\sim} \top$ in ${\rm JP'}$.

It is then expected that if we add a new propositional constant $\bot$ to the alphabet of $\mathcal{L}'$ and consider ${\rm JP}'$ enhanced with Axiom {\rm (A8)}, defining  $\neg$  as $\neg \alpha := \alpha \rightarrow \bot$, the resulting system will be equivalent to {\rm ILM}.

\begin{definition}[The logic $\mathbf{ILM}_2$]
\label{ILM2}
\index{ILM$_2$} The formulas of ${\rm ILM}_2$ are given by the scheme:
\[  \top \mid \bot \mid p \mid \alpha \wedge \beta \mid \alpha \vee \beta \mid \alpha \rightarrow \beta \mid {\sim} \alpha \]
%Define a logic ${\rm ILM}_1$   by 
Axioms are {\rm (A1)}-{\rm (A8)}, {\rm (A13)} and ${\rm (P')}$. {\rm MP} is the only rule of inference.
\end{definition}

\noindent We now show that ${\rm ILM}_2$ is equivalent to ILM. Note that standard results like DT can be proved in ${\rm ILM}_2$.

\begin{theorem}
\label{ILM-ILM2}
${\rm ILM} \cong{\rm ILM}_2$.
\end{theorem}
\begin{proof}
Define $\neg \alpha := \alpha \rightarrow \bot$ in ${\rm ILM}_2$. As shown for Theorem \ref{thm-JP}, (A13) and (P$'$) are ${\rm ILM}$-theorems.\\
One can show (A9), (A10) and (A11) are ${\rm ILM}_2$-theorems, using following results.\\
(a) $\vdash_{{\rm ILM}_2} \beta \leftrightarrow (\top \rightarrow \beta)$.\hspace{30mm} (b) $\vdash_{{\rm ILM}_2} {\sim} \alpha \leftrightarrow (\alpha \rightarrow {\sim} \top)$.\\
(c) $\{\alpha, \neg \alpha \} \vdash_{{\rm ILM}_2} \bot$. \hspace{35mm} (d) $\vdash_{{\rm ILM}_2} (({\sim} \top \rightarrow \bot) \rightarrow {\sim} \top) \rightarrow {\sim} \top$.
\end{proof}

%\begin{remark}
%As described by Prior \cite{prior1958},  in Peirce's axiomatic system  (a) the axiom $\bot \rightarrow \alpha$ is present, and (b) negation of $\alpha$ is defined as $\alpha \rightarrow \bot$. 
%**drop**It is further observed  that Peirce's formulation of negation uses both the fact - (A) what is false implies anything at all ($\bot \rightarrow \alpha$), and (B) negation as implication of what is false ($p \rightarrow \bot$). 
%**Since, his system has only one negation, both (**?**) formulations of negations  are treated as the same.\\
%In the case of ${\rm ILM}$ (equivalent to ${\rm ILM_2}$), (a) results in defining the negation $\neg$ (as $\vdash_{\rm ILM} \neg \alpha \leftrightarrow (\alpha \rightarrow \bot)$), and (b) results in defining the negation ${\sim}$ (as $\vdash_{\rm ILM} {\sim} \alpha \leftrightarrow (\alpha \rightarrow {\sim} \top)$)**typo.
%\end{remark}

\section{Do{\v{s}en} semantics for ILM}
\label{sec-KripkeILM}

In this section, we refer to the relational semantics given by Do{\v{s}}en for minimal logic and its extensions \cite{Dosen1986,Dosen1999}. The class of `strictly condensed $J$-frames' and that of `strictly condensed $H$-frames' characterize {\rm ML} and {\rm IL} respectively {\rm \cite{Dosen1986}}, and these help us derive relational semantics for the logics ILM and ${\rm ILM}$-${\vee}$. Let us recall the definitions of these frames. Such a frame is essentially a structure based on a poset,  equipped  with a binary relation on the set that defines the semantics for negation.
%A detailed discussion on Do{\v{s}}en's semantics and negation as modal operator can be found in \cite{Dosen1999}. There are notions of classes of `$N$-frames', `$J$-frames' and `$H$-frames' which characterize various logics with negation. 

%An {\it $N$-frame} is a triple $\mathcal{F}:=(X,R_I,R_N)$, where $X$ is a non-empty set and $R_I$, $R_N$ are binary relations satisfying certain properties. Here, we shall always take $R_I$ as a partial order, however, in a general {\rm $N$-frame} it is not a necessary condition. A distinguishing feature of Do{\v{s}}en's relational semantics for the language with negation is the binary relation $R_N$ which defines the semantics for the negation. Then sub-classes of $N$-frames is obtained to characterize various extensions of logics with negation by putting conditions on $R_N$ \cite{Dosen1986}. 

\begin{definition}[Strictly Condensed \textit{J}-frame and \textit{H}-frame]
\label{J-frame} {\rm \cite{Dosen1986}}
$ $\\
A triple $\mathcal{F}:=(X,\leq,R_N)$ is a {\it strictly condensed $J$-frame}, if it satisfies the following.
\begin{enumerate}[label={{\rm (\arabic*)}}, leftmargin=1.5 \parindent, noitemsep, topsep=0pt]
\item $(X,\leq)$ is a poset.
\item $(\leq R_{N} \leq^{-1}) \subseteq R_{N}$.
\item $R_N$ is symmetric.
\item $\forall x,y \in X(x R_N y \Rightarrow \exists z \in X (x \leq z ~\&~ y \leq z ~\&~ x R_N z))$.
\end{enumerate}
\vskip 3pt 
%A strictly condensed $J$-frame $\mathcal{F}:=(X,\leq,R_N)$ is  a {\it strictly condensed  $H$-frame}, if the 
$\mathcal{F}$  is a {\it strictly condensed  $H$-frame}  when, in addition to the above conditions, the relation $R_N$ is reflexive.
\end{definition}

\begin{remark}\label{Nframe}$ $\\
(1) In a strictly condensed  $H$-frame $\mathcal{F}$, reflexivity of $R_N$ implies its symmetry.\\
(2) $\mathcal{F}$ satisfying (1) and (2) is a {\it strictly condensed $N$-frame} \cite{Dosen1986}.
\end{remark}

%Let us now give Do{\v{s}}en semantics for {\rm ILM}.
\noindent `$\hat{N}$-frames', giving the semantics for ILM,  may now be defined.

\subsection{{\texorpdfstring{$\hat{N}$}~}-frames}
\label{dosen-ILM}

There are two negations  $\neg$ and ${\sim}$ in ${\rm ILM}$. Therefore, to define the semantics for each negation, the frame should have two binary relations, one (say, $R_{N_1}$) corresponding to the the negation $\neg$, and another (say, $R_{N_2}$) corresponding to the negation ${\sim}$. Recall Theorem \ref{ILM-ILM1} and Remark \ref{negations} (Section \ref{sec-embed}).

\begin{obs}$ $
\label{obs-Nhatframes}
{\rm 
\begin{enumerate}[label={($\arabic*$)}, leftmargin=1.4 \parindent, noitemsep, topsep=0pt]
\item As $\neg$ is an intuitionistic negation, $R_{N_1}$ must satisfy the properties corresponding to the relation $R_N$ in a strictly condensed  $H$-frame $\mathcal{F}:=$  $(X,\leq,R_N)$.
\item On the other hand, ${\sim}$ is a minimal negation. So $R_{N_2}$ must satisfy the properties corresponding to the relation $R_N$ in a strictly condensed $J$-frame $\mathcal{F}:=(X,\leq,R_N)$.
\item $\neg \neg {\sim} \top \leftrightarrow {\sim} \top$ is the sole axiom in ${\rm ILM}_1$ involving both the negations $\neg,\sim$. So there should be a property corresponding to the axiom that connects $R_{N_1}$ and $R_{N_2}$.
\item The logic ${\rm ILM}$-${\vee}$ has axiom (A12) for the negation ${\sim}$. Thus, there must be a condition on $R_{N_2}$ corresponding to this axiom in an ${\rm ILM}$-${\vee}$ frame.
\end{enumerate}
}
\end{obs}

\noindent Let us now incorporate these points to define a new class of frames. 
%Hereafter, we shall always take $X$ to be a non-empty set and $\leq$ as a partial order. 

%\noindent Some of the conditions on strictly condensed $J$-frames and $H$-frames follow from the others. Let us observe the following for strictly condensed $N$-frames.
%
%\begin{enumerate}
%\item A triple $(X,\leq, R_N)$, where $X$ is non-empty and $\leq$ is a partial order, is a strictly condensed $N$-frame if and only if $(\leq R_N \leq^{-1}) \subseteq R_N$ \cite{Dosen1986}.
%\item In a strictly condensed $N$-frame $(X,\leq,R_N)$, $(R_N \leq^{-1}) = R_N$ always holds. Thus, the conditions `$R_N \leq^{-1}$ is symmetric' and `$R_N \leq^{-1}$ is reflexive' reduces just to `$R_N$ is symmetric' and `$R_N$ is reflexive' respectively.
%\item In a strictly condensed $N$-frame $(X,\leq,R_N)$, `$R_N$ is reflexive' and `$\forall x,y \in X(x R_{N} y \Rightarrow \exists z \in X (x \leq z ~\&~ y \leq z ~\&~ x R_{N} z))$' together imply that `$R_N$ is symmetric'.
%%\sout{This is because for all  $x,y \in X$, $x R_N y$ implies there exists $z \in X$ such that $x \leq z$, $y \leq z$ and $x R_N z$. Now, $y \leq z$, $z R_N z$ ($R_N$ is reflexive), and $x \leq z$ imply $y (\leq R_N \leq^{-1}) x$. Using (1), $y R_N x$.}
%\end{enumerate}
%
%
%Therefore, we can simplify the definition of $\hat{N}$-frames.

\begin{definition}
\label{def-Nhatframes}\index{Nhatframe@$\hat{N}$-frame}
An {\it $\hat{N}$-frame} is a quadruple $\mathcal{F}:= (X, \leq ,R_{N_1},R_{N_2})$, satisfying the following conditions.
\begin{enumerate}[label={($\arabic*$)}, noitemsep, topsep=0pt]
\item $(X, \leq ,R_{N_1})$ is a strictly condensed  $H$-frame.
%\begin{enumerate}
%\item $(\leq R_{N_1} \leq^{-1}) \subseteq R_{N_1}$.
%\item $R_{N_1}$ is reflexive.
%\item $\forall x,y \in X(x R_{N_1} y \Rightarrow \exists z \in X (x \leq z ~\&~ y \leq z ~\&~ x R_{N_1} z))$.
%\end{enumerate}
\item $(X, \leq ,R_{N_2})$ is a strictly condensed  $J$-frame.
%\begin{enumerate}
%\item $(\leq R_{N_2} \leq^{-1}) \subseteq R_{N_2}$.
%\item $R_{N_2}$ is symmetric.
%\item $\forall x,y \in X(x R_{N_2} y \Rightarrow \exists z \in X (x \leq z ~\&~ y \leq z ~\&~ x R_{N_2} z))$.
%\end{enumerate}
\item $\forall x \in X ~ \bigl( \forall y \in X  \bigl(x R_{N_1} y \Rightarrow \exists z \in X  (y R_{N_1} z ~\&~ \forall z'\in X ~ (z \cancel{R}_{N_2} z')) \bigr) \\
~ \hfill \Rightarrow \forall z''\in X ~ (x \cancel{R}_{N_2} z'')\bigr)$.
\end{enumerate}
%$\mathcal{F}$ is called an {\rm $\hat{N}$-frame}. 
The class of all $\hat{N}$-frames is denoted by $\mathfrak{F}_1$.
\vskip 3pt
\noindent An {\it $\hat{N}'$-frame} is an $\hat{N}$-frame $\mathcal{F}$ satisfying
\begin{enumerate}[label={($\arabic*$)}, noitemsep, topsep=0pt]
\setcounter{enumi}{3}
\item $R_{N_2} \subseteq (\leq^{-1})$.
\end{enumerate}
\vskip 2pt
\noindent The class of all $\hat{N}'$-frames is denoted by $\mathfrak{F}_{1}'$.
%An $\hat{N}$-frame $\mathcal{F}$ satisfying $R_{N_2} \subseteq (\leq^{-1})$ is called an {\it $\hat{N}'$-frame}, and the class of all $\hat{N}'$-frames is denoted by $\mathfrak{F}_{1}'$.
\end{definition}

\noindent We shall see in the sequel that Condition (3) of Definition \ref{def-Nhatframes} corresponds to the ${\rm ILM}_1$-axiom  $\neg \neg {\sim} \top \leftrightarrow {\sim} \top$ (cf. Observation  \ref{obs-Nhatframes}(3)).
%Observation \ref{obs-Nhatframes}(3) is the only condition giving the relationship between $R_{N_1}$ and $R_{N_2}$. We shall re-visit this in the sequel to explain its significance. 
Moreover, the condition defining an $\hat{N}'$-frame, namely $R_{N_2} \subseteq (\leq^{-1})$ corresponds to the formula $\alpha \vee {\sim} \alpha$ \cite{Dosen1986}, and thus, is the required condition as mentioned in Observation \ref{obs-Nhatframes}(4).

Let us now give the  notions of valuation, truth and validity, which are introduced in the usual manner as defined in \cite{Dosen1986}.

For an $\hat{N}$-frame $\mathcal{F}:=(X,\leq,R_{N_1},R_{N_2})$, a map $v: PV \rightarrow \mathcal{P}(X)$ is called a {\it valuation} of $\mathcal{L}$ on $\mathcal{F}$ if $v(p)$ is an upset for each $p \in PV$. The pair $\mathcal{M} :=(\mathcal{F}, v)$ is called an $\hat{N}$-{\it model} on the $\hat{N}$-frame $\mathcal{F}$. %:=(X,\leq,R_{N_1},R_{N_2})$.

\begin{definition}[Truth of a formula]
\label{ILM-truth}
\index{Nhatframe@$\hat{N}$-frame!truth}
The {\it truth of a formula} $\alpha \in F$ {\it at a world} $x \in X$ {\it in an $\hat{N}$-model} $\mathcal{M}$ 
%:=(\mathcal{F}, v) = ((X,\leq,R_{N_1},R_{N_2}),v)$ 
(notation: $\mathcal{M},x \vDash \alpha$) is defined by extending the valuation $v: PV \rightarrow  \mathcal{P}(X)$ to the set $F$ of formulas in the standard way \cite{Dosen1986}. We only give the semantic clauses for the connectives $\rightarrow$, $\neg$ and ${\sim}$.
\begin{enumerate}[label={$\arabic*$.}, leftmargin=1 \parindent, noitemsep]
%\item $\mathcal{M}, x \vDash p \Leftrightarrow x \in v(p) \mbox{ for all $p$ in $PV$}$.
%\item $\mathcal{M}, x \vDash  \alpha \wedge \beta \Leftrightarrow \mathcal{M}, x \vDash \alpha \mbox{ and } \mathcal{M}, x \vDash \beta$.
%\item $\mathcal{M}, x \vDash  \alpha \vee \beta \Leftrightarrow \mathcal{M}, x \vDash \alpha \mbox{ or } \mathcal{M}, x \vDash \beta$.
\item $\mathcal{M}, x \vDash  \alpha \rightarrow \beta \Leftrightarrow \mbox{ for all } y \in X, \mbox{ if } x \leq y \mbox{ and } \mathcal{M}, y \vDash \alpha  \mbox{ then } \mathcal{M}, y \vDash \beta$.
%\item $\mathcal{M}, x \vDash  \top$.
%\item $\mathcal{M}, x \nvDash  \bot$.
\item $\mathcal{M}, x \vDash  \neg \alpha \Leftrightarrow \mbox{ for all } y \in X (x R_{N_1} y \Rightarrow \mathcal{M}, y \nvDash \alpha).$
\item $\mathcal{M}, x \vDash  {\sim} \alpha \Leftrightarrow \mbox{ for all } y \in X (x R_{N_2} y \Rightarrow \mathcal{M}, y \nvDash \alpha).$
\end{enumerate}
\end{definition}

\noindent %We shall simply write $x \vDash \alpha$ in place of $\mathcal{M},x \vDash \alpha$, when the model $\mathcal{M}$ is clear from the context.
A formula $\alpha$ is {\it true in a model} $\mathcal{M}$ (notation: $\mathcal{M} \vDash \alpha$) if $\mathcal{M} ,x \vDash \alpha$ for all $x \in X$. A formula $\alpha \in F$ is {\it valid in the} $\hat{N}$-frame $\mathcal{F}$ (notation: $\mathcal{F} \vDash \alpha$) if $\mathcal{M} \vDash \alpha$ for every model $\mathcal{M}$ on the $\hat{N}$-frame $\mathcal{F}$.
A formula $\alpha \in F$ is {\it valid in a class} $\mathcal{C}$ of $\hat{N}$-frames  (notation: $\mathcal{C} \vDash \alpha$) if for every $\hat{N}$-frame $\mathcal{F} \in \mathcal{C}$, $\mathcal{F} \vDash \alpha$.
\vskip 3pt

\noindent The following can  easily be observed.
\begin{proposition} For any formula $\alpha \in F$ and $x \in X$, 
\begin{enumerate}[label={$\arabic*$.}, leftmargin=1 \parindent, noitemsep]
\item $\mathcal{M}, x \vDash  \neg \top  \Leftrightarrow \forall y \in X ~ (x \cancel{R}_{N_1} y)$,
\item $\mathcal{M}, x \vDash  {\sim} \top  \Leftrightarrow \forall y \in X ~ (x \cancel{R}_{N_2} y)$,
\item $\forall y \in X((\mathcal{M}, x \vDash \alpha ~\&~ x \leq y) \Rightarrow \mathcal{M}, y \vDash \alpha$),
\item $\mathcal{M}, x \vDash \neg \alpha \Leftrightarrow \forall y \in X(\exists z \in X(x \leq z ~\&~ y \leq z) \Rightarrow y \nvDash \alpha)$,
\item $\mathcal{M}, x \vDash {\sim} \alpha \Leftrightarrow \forall y \in X \bigl(x R_{N_1} y \Rightarrow (y \vDash \alpha \Rightarrow \exists z \in X(y R_{N_1} z ~\&~ \forall z' \in X ~ (z \cancel{R}_{N_2} z')))\bigr)$.
\end{enumerate}
\end{proposition}

Let us return to Condition (3) in Definition \ref{def-Nhatframes}. Recall the ${\rm ILM}_1$-axiom $\neg \neg {\sim} \top \leftrightarrow {\sim} \top$. For any $\hat{N}$-model $\mathcal{M}$, expansion of $\mathcal{M} \vDash \neg \neg {\sim} \top \leftrightarrow {\sim} \top$  using Definition \ref{ILM-truth} gives
\begin{align*}
\forall x \in X \bigl(\forall y \in X \bigl(x R_{N_1} y \Rightarrow \exists z\in X (y R_{N_1} z ~\&~ \forall z'\in X(z \cancel{R}_{N_2} z'))\bigr) \Leftrightarrow  \forall z''\in X(x \cancel{R}_{N_2} z'')\bigr)
\end{align*}
%$\forall y \in X \bigl(x R_{N_1} y \Rightarrow \exists z\in X (y R_{N_1} z ~\&~ \forall z'\in X(z \cancel{R}_{N_2} z'))\bigr) \\
%~ \hfill \hspace{-4mm} \Leftrightarrow  \forall z''\in X(x \cancel{R}_{N_2} z'')$.\\
The reverse implication in the above is always true. Arguing for its contrapose, let $y \in X$ such that $x R_{N_1}y$ and $\forall z \in X (yR_{N_1}z \Rightarrow \exists z' \in X~ z R_{N_2}z')$. $xR_{N_1}y$ implies $y R_{N_1}x$. Therefore, there exists $z' \in X$ such that $x R_{N_2}z'$. Take $z'' = z'$.\\
Thus, one direction of the above bi-implication is always true. The other direction is exactly Condition (3) in Definition \ref{def-Nhatframes}. 
%So, we may expect that ${\rm ILM}$ will be complete with respect to the class $\mathfrak{F}_1$ of $\hat{N}$-frames.
Therefore this condition expresses the required property as mentioned in Observation \ref{obs-Nhatframes}(3).

A logic is said to be {\it determined} by a class of frames if it is complete with respect to the class of frames. Our aim now is to show that the logic ${\rm ILM}$ (${\rm ILM}$-${\vee}$) is determined by the class $\mathfrak{F}_1$ ($\mathfrak{F}_1'$) of $\hat{N}$-frames ($\hat{N}'$-frames).

\subsection{Characterization results for \textrm{ILM} and \textrm{ILM}-\texorpdfstring{${\vee}$}~}
\label{subsec-char}
Soundness, i.e. for any formula $\alpha \in F$, $\vdash_{\rm ILM} \alpha \Rightarrow \mathfrak{F}_1 \vDash \alpha$ and $\vdash_{\rm ILM-{\vee}} \alpha \Rightarrow \mathfrak{F}_1' \vDash \alpha$,
can be obtained in the standard manner, using induction on the number of connectives of $\alpha$.
%The proof of completeness follows the usual route taken  for logics with negation \cite{Dosen1999}.
%, by first obtaining an equivalent version of Truth Lemma, and then using it to show completeness. 
Let us sketch the proof of completeness for ILM and {\rm ILM}-${\vee}$. The structure of the proof is similar to the cases of {\rm ML} and {\rm IL} given in \cite{Dosen1986}. Therefore, we shall only show the steps where there is a change or an extension to the proofs in \cite{Dosen1986}.
%\begin{theorem}[Completeness]
%\label{comp-ILM-dosen}
%For any formula $\alpha \in F$,\\
%{\rm (i).} $\mathfrak{F}_1 \vDash \alpha \Rightarrow ~\vdash_{\rm ILM} \alpha$.\\
%{\rm (ii).} $\mathfrak{F}_1' \vDash \alpha \Rightarrow ~\vdash_{\rm ILM-{\vee}} \alpha$.
%\end{theorem}
We first require the concept of a theory \cite{Segerberg1968,Dosen1986}, that we extend to the context of {\rm ILM}. 

\begin{definition}[Theory]
\label{theory}\index{theory}
A {\it theory} $T \subseteq F$ with respect to an extension ${\rm S}$ of {\rm ILM}, is a non-empty set of formulas in $\mathcal{L}$ such that, for formulas $\alpha,\beta \in F$,
\begin{enumerate}[label={$\arabic*$.}, leftmargin=1 \parindent, noitemsep, topsep=0pt]
\item if $\alpha \in T$ and $\alpha \rightarrow \beta \in T$, then $\beta \in T$ (closed under deduction),
\item $\alpha \in T$, where $\vdash_{\rm S} \alpha$, and
\item if $\alpha, \beta \in T$ then $\alpha \wedge \beta \in T$ (closed under $\wedge$).
\end{enumerate}
\index{theory!consistent}A theory is {\it consistent} if $\bot \notin T$, otherwise {\it inconsistent}. A {\it prime} theory \index{theory!prime} is a consistent theory such that for any two formulas $\alpha,\beta \in F$, if $\alpha \vee \beta \in T$ then either $\alpha \in T$ or $\beta \in T$.
\end{definition}

\noindent Using Axiom (A8) namely $ \bot \rightarrow \alpha$, 
%$\vdash_{\rm S} \bot \rightarrow \alpha$, 
a theory $T$ is consistent if and only if there exists a formula $\alpha$ such that $\alpha \notin T$. Now, for an arbitrary set $\Delta$ of formulas, the intersection of all theories containing $\Delta$ is also a theory; it is called the {\it theory generated by $\Delta$} and denoted by $Th(\Delta)$.
%The following lemmas and notions will be useful in proving completeness. The proofs of the results are in the same lines as in \cite{Dunn1995,Dunn2005}.
%
%\begin{lemma}
%\label{gentheory}
%\[ Th(\Delta) = \{ \alpha \in F ~|~ \exists \phi_1,\phi_2 \ldots,\phi_n \in \Delta \mbox{ such that } \vdash_{\rm S} (\phi_1 \wedge \phi_2 \wedge \cdots \wedge \phi_n) \rightarrow \alpha \}.\]
%As a consequence, $Th(\{\top\}) = \{\alpha \in F ~|~ \vdash_S \alpha\}$.
%\end{lemma}
%
%\begin{lemma}
%\label{lem1-theory}$ $\\
%(1). For theories $P$ and $Q$, $\alpha \in Th(P \cup Q)$ if and only if there exist $\beta \in P$ and $\phi \in Q$ such that $\vdash_{\rm S} \beta \rightarrow (\phi \rightarrow \alpha)$.\\
%(2). For a theory $T$, and formulas $\phi,\alpha \in F$, $\alpha \in Th(T\cup\{\phi\})$ if and only if there exists $\beta \in T$ such that $\vdash_{\rm S} \beta \rightarrow (\phi \rightarrow \alpha)$. 
%\end{lemma}
A set $F' \subseteq F$  is {\it closed under} $\vee$, if for any $\alpha,\beta \in F'$, $\alpha \vee \beta \in F'$; $F'$ is then called {\it disjunctive closed} (or $\vee${\it -closed}). In fact, any arbitrary $\Delta \subseteq F$ can be extended to a $\vee$-closed set, called {\it disjunctive closure} of $\Delta$ (denoted ${\rm dc}(\Delta)$), as follows.
\[ {\rm dc}(\Delta):= \bigcap\{ \Delta' \subseteq F ~|~ \Delta \subseteq \Delta' \mbox { and } \forall \alpha,\beta\in F (\alpha,\beta \in \Delta' \Rightarrow \alpha \vee \beta \in \Delta') \}. \]
%For $\Delta \subseteq F$, 
${\rm dc}(\Delta)$ is $\vee$-closed. Moreover, for $\alpha \in F$, if $\beta \in {\rm dc}(\{\alpha\})$ then $\vdash_{\rm S} \alpha \leftrightarrow \beta$.

\begin{lemma}[Extension lemma]
\label{lem-pel}
Let $\Delta$ be a consistent theory and $\Gamma \subseteq F$ be a $\vee$-closed set. If $\Delta \cap \Gamma = \emptyset$ then there exists a prime theory $P$ such that $\Delta \subseteq P$ and $P \cap \Gamma = \emptyset$.
\end{lemma}

\noindent An immediate corollary is obtained when  $\Gamma := {\rm dc}\{\alpha\}$ for any $\alpha \in F$.

\begin{corollary}
\label{cor-pel}
Let $\Delta$ be a consistent theory and $\alpha \in F$ be such that $\alpha \notin \Delta$. Then there is a prime theory $P$ such that $\Delta \subseteq P$ and $\alpha \notin P$.
\end{corollary}

Let us now define the `canonical' frame $\mathcal{F}^c := (X^c, \subseteq, R_{N_1}^c, R_{N_2}^c)$ in the standard way. 

\begin{definition}[Canonical frame]
\label{can-frame}
The {\rm canonical frame} for any extension ${\rm S}$ of ${\rm ILM}$ is the quadruple $\mathcal{F}^c := (X^c, \subseteq, R_{N_1}^c, R_{N_2}^c)$, where\\
$X^c := \{P \subseteq F ~|~ P \mbox{ is a prime theory}\}$,\\
$P R_{N_1}^c Q$ if and only if (for all $\alpha \in F$, $\neg \alpha \in P \Rightarrow \alpha \notin Q$), and\\
$P R_{N_2}^c Q$ if and only if (for all $\alpha \in F$, ${\sim} \alpha \in P \Rightarrow \alpha \notin Q$).
\end{definition}

\noindent We shall now show that $\mathcal{F}^c$ is indeed an $\hat{N}$-frame. For that, we shall require the following standard result.

\begin{lemma}
\label{lem-can1}
For any $\alpha \in F$ and any $P \in X^c$,
\begin{enumerate}[label={$\arabic*$.}, leftmargin=1 \parindent, noitemsep, topsep=0pt]
\item ${\neg} \alpha \in P$ if and only if for all $Q \in X^c$, $P R_{N_1}^c Q \Rightarrow \alpha \notin Q$, and
\item ${\sim} \alpha \in P$ if and only if for all $Q \in X^c$, $P R_{N_2}^c Q \Rightarrow \alpha \notin Q$.
\end{enumerate}
\end{lemma}
\noindent Recall Definition \ref{def-Nhatframes} of an $\hat{N}$-frame, more specifically its distinctive feature 
%of a $\hat{N}$-frame is 
Condition (3). We shall only show the following: for all $P \in X^c$,\\
$\bigl( \forall Q \in X^c  \bigl(P R_{N_1}^c Q \Rightarrow \exists R \in X^c  (Q R_{N_1}^c R ~\&~ \forall R'\in X^c ~ (R \cancel{R}_{N_2}^c R'))\bigr)$\\
\hspace*{\fill} $\Rightarrow \forall R'' \in X^c ~ (P \cancel{R}_{N_2}^c R'')\bigr)$.\\
Suppose $\forall Q\in X^c\bigl(P R_{N_1}^c Q \Rightarrow \exists R \in X^c  (Q R_{N_1}^c R ~\&~ \forall R'\in X^c ~ (R \cancel{R}_{N_2}^c R'))\bigr)$. Using Lemma \ref{lem-can1}, $\neg \neg {\sim} \top \in P$. Since  $\vdash_{\rm ILM} \neg \neg {\sim} \top \leftrightarrow {\sim} \top$, we have ${\sim} \top \in P$. Finally, Lemma \ref{lem-can1} implies $\forall R'' \in X^c ~ (P \cancel{R}_{N_2}^c R'')$. Thus we have

\begin{proposition}
$\mathcal{F}^c := (X^c, \subseteq, R_{N_1}^c, R_{N_2}^c)$ is an $\hat{N}$-frame.
\end{proposition}
\noindent The canonical valuation $v^c$ on the canonical frame $\mathcal{F}^c := (X^c, \subseteq, R_{N_1}^c, R_{N_2}^c)$, defined as $v^c(p):= \{P \in X^c~|~ p \in P\}$ for all $p \in PV$, gives  the  canonical $\hat{N}$-model  $\mathcal{M}^c := (\mathcal{F}^c,v^c)$. The truth lemma follows.

\begin{lemma}[Truth Lemma]
\label{lem-truth}\index{truth lemma}
For any $\alpha \in F$ and $P \in X^c$, $\mathcal{M}^c, P \vDash \alpha$ if and only if $\alpha \in P$.
\end{lemma}

Finally, using the truth lemma, the completeness result for the logics ${\rm ILM}$ and ${\rm ILM}$-$\vee$ is obtained.
\begin{theorem}[Completeness]
\label{comp-ILM-dosen}
For any formula $\alpha \in F$,\\
{\rm (i)} $\mathfrak{F}_1 \vDash \alpha \Rightarrow ~\vdash_{\rm ILM} \alpha$.\\
{\rm (ii)} $\mathfrak{F}_1' \vDash \alpha \Rightarrow ~\vdash_{\rm ILM-{\vee}} \alpha$.
\end{theorem}
\begin{proof}$ $\\
(i) Let $\nvdash_{\rm ILM} \alpha$. If $\alpha =\bot$, then by definition of valuation, $\mathcal{F} \nvDash \bot$ for any $\hat{N}$-frame $\mathcal{F}$. Therefore, suppose $\alpha$ is not $\bot$. Then define $\Delta := Th(\top)$ and $\Gamma := {\rm dc}(\{ \alpha \})$. $\Delta$ is a consistent theory, because $\bot \notin \Delta$ using soundness. $\nvdash_{\rm ILM} \alpha \Rightarrow \alpha\notin \Delta$, which implies, by Corollary \ref{cor-pel}, there is $P \in X^c$ such that $\Delta \subseteq P$ and $\alpha \notin P$. Using Lemma \ref{lem-truth}, $ \mathcal{M}^c , P \nvDash \alpha$.

\noindent (ii) For ${\rm ILM}$-$\vee$, we have to show that the canonical frame for ${\rm ILM}$-$\vee$ belongs to $\mathfrak{F}_1'$,  and for all frames $\mathcal{F} \in \mathfrak{F}_1'$ and $\alpha \in F$, $\mathcal{F} \vDash \alpha \vee {\sim} \alpha$. Consider the canonical frame $\mathcal{F}^c :=(X^c,\subseteq,R_{N_1}^c,R_{N_2}^c)$. We have already shown that it is an $\hat{N}$-frame. We have to show that for all $P,Q, \in X^c$, if $P R_{N_2}^c Q$ then $Q \subseteq P$. Let $P$ and $Q$ be such that  $P R_{N_2}^c Q$ and $Q \nsubseteq P$. Then there exists $\alpha \in Q$ such that $\alpha \notin P$. We have $\alpha \vee {\sim} \alpha \in P$, as $\alpha \vee {\sim} \alpha$ is an axiom in ${\rm ILM}$-${\vee}$. Therefore, ${\sim} \alpha \in P$.  Using the definition of $R_{N_2}^c$, we have $\alpha \notin Q$, a contradiction.
\end{proof}

%Next, we shall see Segerberg's relational semantics for {\rm ILM}.

\section{Segerberg semantics for ILM}
\label{subsec-KripkeILM}

In this section, we shall see another relational semantics for {\rm ILM} based on Segerberg's semantics for {\rm ML} and its extensions. 
%To show the soundness and completeness, one way is to process as done in the previous Section. However, we shall show this in Section \ref{inter} using inter-translatability between the $\hat{N}$-frames and frames in Segerberg's semantics. This approach has been adopted from Do{\v{s}}en \cite{Dosen1986}.
%Let us  recall Segerberg's semantics for {\rm ML} and its extensions. 
The latter is defined through  {\it $j$-frames}, which are triples of the form  $(W, \leq, Y_0)$, where $(W,\leq)$ is a poset with $W \neq \emptyset$ and $Y_0$ is an upset of $W$. $Y_0$ is used to define the semantics for negation. 
When $Y_0 = \emptyset$, the $j$-frame is called a {\it normal} frame, denoted simply by the pair $(W,\leq)$. Segerberg showed that {\rm ML} and {\rm IL} are determined by the class of all $j$-frames and normal frames respectively \cite{Segerberg1968}.

\subsection{Sub-normal frames}
\label{segerberg-ILM}

We shall now study the class of $j$-frames that can characterize {\rm ILM}. In  Section \ref{dosen-ILM}, for the   case of Do{\v{s}}en semantics for {\rm ILM},  two binary relations $R_{N_1}$ and $R_{N_2}$ were considered to define the semantics for the negations $\neg$ and ${\sim}$ respectively. 
%In $j$-frames, an upset $Y_0$ defines the semantics for negation. 
In this case, we 
%\begin{enumerate} 
consider a poset $(W,\leq)$, and  two upsets $Y_{\neg}$ and $Y_{{\sim}}$ (say), for the  negations $\neg$ and ${\sim}$ respectively. 
Moreover, since  $\neg$ is intuitionistic, $Y_{\neg}$ must be the empty set, as for normal frames corresponding to {\rm IL.}
%\end{enumerate}

\noindent Recall next the following class of frames, first given by Woodruff \cite{Woodruff1970}, to characterize the  extension ${\rm JP'}$ (see Definition \ref{JP'}) of {\rm ML}.

\begin{definition}[Sub-normal frame]
\label{subnormal}
$ $\\
A {\it sub-normal frame}  {\rm \cite{Woodruff1970}} is a $j$-frame $\mathcal{F}:=(W,\leq,Y_0)$ satisfying the following condition:
\[ \forall x \in W  (\forall y \in W (x \leq y \Rightarrow \exists z \in W (y \leq z ~\&~ z \in Y_0 )) \Rightarrow x \in Y_0) \tag{D}. \]
%$\mathcal{F}$ is called a {\rm sub-normal frame}. 
The class of all sub-normal frames is denoted by $\mathfrak{F}_2$.\\
A {\it sub-normal identity frame} is a sub-normal frame $\mathcal{F}:=(W,\leq,Y_0)$ where $\leq$ is an identity relation on $W {\setminus Y_0}$, i.e.
\[\forall x,y \in W {\setminus} Y_0 ~ (x\leq y \Rightarrow y \leq x)  \tag{E}. \]
The class of all sub-normal identity frames is denoted by $\mathfrak{F}_2'$.
\end{definition}

\begin{theorem}
\label{char-JP'}
{\rm \cite{Woodruff1970}} ${\rm JP'}$ is determined by the class of all sub-normal frames.
\end{theorem}

%\begin{obs}
%{\rm ?
%The following relationship between sub-normal frames and normal frames can be observed. Consider a sub-normal frame $(W,\leq,Y_0)$. Trivially $(W,\leq)$ is a normal frame. However, we have more. For every $w \in W {\setminus} Y_0$, one can generate a sub-normal frame $(W_{v_w},\leq)$: by the contraposition of $(D)$, we have $v_w \in W$ such that $w \leq v_w$ and $\forall v' \in W (v_w \leq v' \Rightarrow v' \notin Y_0)$. Consider the subframe `generated by $v_w$' - $(W_{v_w}, \leq, Y_{0_{v_w}})$, where $W_{v_w} := \{ z \in W ~ \mid ~ v_w \leq z \}$ and $Y_{0_{v_w}} := W_{v_w} \cap Y_0$. Then $Y_{0_{v_w}}$ empty. Therefore $(W_{v_w}, \leq)$ is a normal frame.
%}
%\end{obs}

\noindent As observed in Theorem \ref{thm-JP}, {\rm ILM} is an extension of ${\rm JP'}$. So it is expected that models of {\rm ILM} and {\rm ILM}-${\vee}$, in this semantics, would be based on sub-normal frames.
Let us give the basic definitions. 
A {\it valuation} of $\mathcal{L}$ on a sub-normal frame $\mathcal{F}:=(W,\leq,Y_0)$ is a mapping $v:PV \rightarrow \mathcal{P}(W)$ such that $v(p)$ is an upset for each $p \in PV$. 
%For a sub-normal frame $\mathcal{F}:=(W,\leq,Y_0)$, and a valuation $v_0$ on $\mathcal{F}$, 
A pair $\mathcal{M} :=(\mathcal{F}, v)$ is called a {\it sub-normal model} on the sub-normal frame $\mathcal{F}$. The {\rm truth} of a formula $\alpha \in F$ at a world $w \in W$ in the model $\mathcal{M} $ ({\rm notation}: $\mathcal{M},w \vDash \alpha$) for propositional variables, $\vee$, $\wedge$, $\bot$ and $\top$ is given in the standard way; for $\rightarrow$ as in Definition \ref{ILM-truth}(1); and for $\neg $ and ${\sim}$, 
\begin{enumerate}[label={{\rm \arabic*.}}, leftmargin=1.5 \parindent, noitemsep] 
\item $\mathcal{M}, w \vDash  {\sim} \alpha \Leftrightarrow \mbox{ for all } w' \in W, \mbox{ if } w \leq w' \mbox{ and } \mathcal{M}, w' \vDash \alpha \mbox{ then } w' \in Y_0.$
\item $\mathcal{M}, w \vDash  \neg \alpha \Leftrightarrow \mbox{ for all } w' \in X (w \leq w' \Rightarrow \mathcal{M}, w' \nvDash \alpha).$
%\item  $\mathcal{M}, w \vDash  \top$.
%\item $\mathcal{M}, w \nvDash  \bot$.
\end{enumerate}

%\noindent used? When the model $\mathcal{M}$ is clear from the context, we simply write $w \vDash \alpha$. 
Standard definitions and notations give the notions of truth of a formula in a sub-normal model, its validity  in a sub-normal frame and validity in a class of sub-normal frames. Some properties of a sub-normal model $\mathcal{M}$ are as follows -- these can be derived in a straightforward manner.
\begin{proposition} \noindent
\begin{enumerate}[label={{\rm \arabic*.}}, leftmargin=1.5 \parindent, noitemsep, topsep=0pt] 
\item For $w \in W$, $\mathcal{M},w \vDash {\sim}\top \Leftrightarrow w \in Y_0$.
\item For all $w' \in W$, if $\mathcal{M}, w \vDash \alpha$ and $w \leq w'$, then $\mathcal{M}, w' \vDash \alpha$.
\item For all $w' \in W$, $\mathcal{M}, w \vDash \neg \alpha$ if and only if $\forall w' \in W(\exists w'' \in W(w \leq w'' ~\&~ w' \leq w'') \Rightarrow w' \nvDash \alpha)$.
\item For all $w' \in W$, $\mathcal{M}, w \vDash {\sim} \alpha$ if and only if $\forall w' \in W\bigl(w \leq w' \Rightarrow \bigl(w' \vDash \alpha \Rightarrow \exists w'' \in W(w' \leq w'' ~\&~ w'' \in Y_0)\bigr)\bigr)$.
\end{enumerate}
\end{proposition}
Now, as done for $\hat{N}$-frames in the previous section, let us consider the ${\rm ILM}_1$-axiom  $\neg \neg {\sim} \top \leftrightarrow {\sim} \top$. Let $\mathcal{M}$ be any sub-normal model. Expanding $\mathcal{M} \vDash \neg \neg {\sim} \top \leftrightarrow {\sim} \top$, 
%where $\mathcal{M}:=(\mathcal{F}, v)$ and $\mathcal{F}:=(W,\leq,Y_0)$ is a sub-normal frame. 
one obtains the following.
%\begin{align*}
%\forall w \in W & (w \vDash \neg \neg {\sim} \top \leftrightarrow {\sim} \top) \\
%\Leftrightarrow ~\forall w \in W & (w \vDash \neg \neg {\sim} \top \Leftrightarrow w \vDash {\sim} \top)\\
%\Leftrightarrow ~\forall w \in W & (\forall w' \in W (w \leq w' \Rightarrow w' \nvDash \neg {\sim} \top ) \Leftrightarrow w \vDash {\sim} \top) \\
%\Leftrightarrow ~\forall w \in W & \bigl(\forall w' \in W \bigl(w \leq w' \Rightarrow \exists w'' \in W (w' \leq w'' ~\&~ w'' \in Y_0 ) \bigr) \Leftrightarrow w \in Y_0\bigr).
%\end{align*}
\begin{align*}
\forall w \in W & \bigl(\forall v \in W \bigl(w \leq v \Rightarrow \exists v' \in W (v \leq v' ~\&~ v' \in Y_0 ) \bigr) \Leftrightarrow w \in Y_0\bigr).
\end{align*}
The reverse direction of the implication in the above holds anyway. Indeed, consider the contraposition, and let $v \in W$ be such that $w \leq v$ and $\forall v' \in W (v \leq v' \Rightarrow v' \notin Y_0)) $. If $w \in Y_0$ then $v \in Y_0$ (as $Y_0$ is an upset). However, $v\leq v$ implies $v \notin Y_0$, a contradiction.\\
The forward direction of the implication  is exactly Condition (D). 
Therefore, we expect that ${\rm ILM}$ will be complete with respect to the class $\mathfrak{F}_2$ of sub-normal frames. Moreover, Condition (E) on sub-normal identity frames corresponds to the formula $\alpha \vee {\sim} \alpha$ \cite{Segerberg1968}. 
 Thus, we also expect that ${\rm ILM}$-${\vee}$ will be complete with respect to the class $\mathfrak{F}'_2$ of sub-normal identity frames.
In the next section, we prove these completeness results by obtaining relationships between the classes of sub-normal (sub-normal identity) frames and $\hat{N}$-frames ($\hat{N}'$-frames).

%In the next section, we prove these completeness results by obtaining relationships between the classes of sub-normal (identity) frames and N-hat (N-hat') -frames.

\subsection{\hspace{-2.5pt}Inter-translation between {\texorpdfstring{$\hat{N}$}~}-frames and sub-normal frames}
\label{inter}
%In Section \ref{sec-KripkeILM}, we defined a relational semantics for ILM based on Do{\v{s}}en's $N$-frames \cite{Dosen1986}. 
Do{\v{s}}en had shown that there is an inter-translation between strictly condensed $J$-frames and $j$-frames for {\rm ML}, preserving the truth of any formula $\alpha \in F$ at any world $w$ of the respective frame. We observe that this inter-translation can be extended to the case of {\rm ILM}, i.e. starting from a sub-normal frame, we can obtain an $\hat{N}$-frame; and from an $\hat{N}$-frame, we can obtain a sub-normal frame -- preserving truth. We  only give the highlights of the proofs.

\begin{theorem}
\label{trans1}
Let $\mathcal{F}:=(W,\leq,Y_0)$ be a sub-normal frame, i.e $\mathcal{F} \in \mathfrak{F}_2$. Define the relations $R_{N_1}$ and $R_{N_2}$ over $W$ for all $x,y \in W$ as:
\begin{enumerate}[label=(\Alph*), leftmargin=2 \parindent, noitemsep]
\item $x R_{N_1} y$ if and only if $\exists z \in W(x \leq z ~\&~ y \leq z)$.
\item $x R_{N_2} y$ if and only if $\exists z \in W(x \leq z ~\&~ y \leq z ~\&~ z \notin Y_0)$.
\end{enumerate}
Then we have the following.
\begin{enumerate}[label={{\rm \arabic*.}}, leftmargin=1.5 \parindent, noitemsep, topsep=0pt] 
\item $Y_0 = \{z \in W ~|~ \forall x \in W (z \cancel{R}_{N_2} x) \}$.
\item $\Phi(\mathcal{F}):= (W,\leq,R_{N_1},R_{N_2})$ is an $\hat{N}$-frame.  If $\mathcal{F}$ is a sub-normal identity frame then $\Phi(\mathcal{F})$ is an $\hat{N}'$-frame.
%Denote this frame as $\Phi(\mathcal{F})$.
\item If $v$ is a valuation on $\mathcal{F}$, then $v$ is a valuation on $\Phi(\mathcal{F})$ such that for all $\alpha \in F$ and $x \in W$,
\begin{enumerate}[label={{\rm (\alph*)}}, leftmargin=1.5 \parindent, noitemsep, topsep=0pt] 
\item $(\mathcal{F}, v) , x \vDash \alpha \Leftrightarrow (\Phi(\mathcal{F}),v), x \vDash \alpha$, 
\item $(\mathcal{F}, v) \vDash \alpha \Leftrightarrow (\Phi(\mathcal{F}),v) \vDash \alpha$, and 
\item $\mathcal{F} \vDash \alpha \Leftrightarrow \Phi(\mathcal{F}) \vDash \alpha$.
\end{enumerate}
\end{enumerate}
\end{theorem}
\begin{proof}$ $\\
2. We shall only prove that   Condition (3) of Definition \ref{def-Nhatframes} is satisfied, i.e. $\forall x \in W ~ \bigl( \forall y \in W  \bigl(x R_{N_1} y \Rightarrow \exists z \in W  (y R_{N_1} z ~\&~ \forall z'\in W ~ (z \cancel{R}_{N_2} z'))\bigr) \Rightarrow \forall z''\in W ~ (x \cancel{R}_{N_2} z'')\bigr)$.\\
Using the expression for $Y_0$ obtained in (1), this condition is equivalent to the following: $\forall x \in W ~ \bigl( \forall y \in W  \bigl(x R_{N_1} y \Rightarrow \exists z \in W  (y R_{N_1} z ~\&~ z \in Y_0)\bigr) \Rightarrow x \in Y_0 \bigr)$. Let $x \in W$ and let\\
$\forall y \in W  (x R_{N_1} y \Rightarrow \exists z \in W  (y R_{N_1} z ~\&~ z \in Y_0))$. \hfill{(*)}\\
To show $x \in Y_0$, we shall use Condition (D) of the sub-normal frame $(W,\leq,Y_0)$, i.e. $\forall x \in W  \bigl(\forall y \in W \bigl(x \leq y \Rightarrow \exists z \in W (y \leq z ~\&~ z \in Y_0 )\bigr) \Rightarrow x \in Y_0 \bigr)$.\\
{\it Claim}: $\forall y \in W (x \leq y \Rightarrow \exists z \in W (y \leq z \mbox{ and } z \in Y_0 ))$.\\ Indeed, let $x \leq y$, for $y \in W$. Using (A) and $y \leq y$, we have $x R_{N_1} y$. Then using (*), we have $z' \in W$ such that $y R_{N_1} z'$ and $z' \in Y_0$. Using (A) on $y R_{N_1} z'$, there exists $z \in W$ such that $y \leq z$ and $z' \leq z$. Since $Y_0$ is an upset and $z' \in Y_0$, $z \in Y_0$. Thus we have the claim.\\
%Finally, $y \leq z$ and $z \in Y_0$ gives us the claim.\\ 
Therefore, using Condition (D), we have $x \in Y_0$.
\vspace{1mm}

\noindent Now let $\mathcal{F} \in \mathfrak{F}_2'$. We have to show that $R_{N_2} \subseteq (\leq^{-1})$. Let $x,y \in W$ such that $x R_{N_2} y$. Condition (B) implies that there exists $z \in W$ such that $x \leq z$, $y \leq z$, and $z \notin Y_0$. Since $Y_0$ is an upset, $z\notin Y_0$ implies $x,y \notin Y_0$. $\mathcal{F} \in \mathfrak{F}_2'$ and $y,z \notin Y_0$, $y\leq z$ imply that $z \leq y$. Then, $x \leq z$ and $z \leq y$ imply $x \leq y$. Again using $\mathcal{F} \in \mathfrak{F}_2'$, $x,y \notin Y_0$ and $x\leq y$, we get  $y \leq x$.
\end{proof}

\begin{theorem}
\label{trans2}
Let $\mathcal{G}:=(W,\leq,R_{N_1},R_{N_2})$ be an $\hat{N}$-frame. Define 
%the subset $Y_0 \subseteq W$ as 
\begin{enumerate}[label=(\Alph*), leftmargin=2 \parindent, noitemsep]
%[label=(\Alph*)]
\item[(C)] $Y_0 := \{z \in W ~ \mid ~ \forall x \in W (z \cancel{R}_{N_2} x) \}.$
\end{enumerate} 
Then we have the following.
\begin{enumerate}[label={{\rm \arabic*.}}, leftmargin=1.5 \parindent, noitemsep, topsep=0pt] 
\item 
\begin{enumerate}[label={{\rm (\alph*)}}, leftmargin=1.5 \parindent, noitemsep, topsep=0pt] 
	\item $x R_{N_2} y$ if and only if $\exists z \in W (x \leq z ~\&~ y \leq z ~\&~ z \notin Y_0)$, and
	\item $x R_{N_1} y$ if and only if $\exists z \in W (x \leq z ~\&~  y \leq z)$.
\end{enumerate}
\item $\Psi(\mathcal{G}):=(W,\leq,Y_0)$ is a sub-normal frame.  If $\mathcal{G}$ is an $\hat{N}'$-frame then $\Psi(\mathcal{G})$ is a sub-normal identity frame.
%Denote this frame by $\Psi(\mathcal{G})$.
\item If $v$ is a valuation on $\mathcal{G}$, then $v$ is a valuation on $\Psi(\mathcal{G})$ such that for all $\alpha \in F$ and $x \in W$,
\begin{enumerate}[label={{\rm (\alph*)}}, leftmargin=1.5 \parindent, noitemsep, topsep=0pt] 
\item $(\mathcal{G}, v) , x \vDash \alpha \Leftrightarrow (\Psi(\mathcal{G}),v), x \vDash \alpha$, 
\item $(\mathcal{G}, v) \vDash \alpha \Leftrightarrow (\Psi(\mathcal{G}),v) \vDash \alpha$, and 
\item $\mathcal{G} \vDash \alpha \Leftrightarrow \Psi(\mathcal{G}) \vDash \alpha$.
\end{enumerate}
\end{enumerate}
\end{theorem}
\begin{proof}$ $\\
2. We only show that the sub-normal frame satisfies Condition (D) of Definition \ref{subnormal}, i.e.$(\forall y \in W (x \leq y \Rightarrow \exists z \in W (y \leq z \mbox{ and } z \in Y_0 )) \Rightarrow x \in Y_0)$.
Let $x\in W$ be such that \\
$\forall y \in W (x \leq y \Rightarrow \exists z \in W (y \leq z \mbox{ and } z \in Y_0))$.  \hfill{(**)}\\
%We have to show that $x \in Y_0$. Using Condition (3) of $\hat{N}$-frames, we show:\\
To get $x \in Y_0$, we utilize Condition (3) of $\hat{N}$-frames and show:\\
\centerline {$\forall y \in W  (x R_{N_1} y \Rightarrow \exists z \in W  (y R_{N_1} z \mbox{ and } \forall z'\in W ~ (z \cancel{R}_{N_2} z')))$,}\\
\centerline{i.e. $\forall y \in W  (x R_{N_1} y \Rightarrow \exists z \in W  (y R_{N_1} z \mbox{ and } z \in Y_0))$.}\\
Let $xR_{N_1}y$. Using 1(b), there exists $z \in W$ such that $x \leq z$ and $y \leq z$. Using $x \leq z$ and (**), we have $z' \in W$ such that $z \leq z'$ and $z' \in Y_0$. Since $R_{N_1}$ is reflexive, we have $z'R_{N_1}z'$. Using $y \leq z \leq z'$ and $(\leq R_{N_1} \leq^{-1}) \subseteq R_{N_1}$, we have $y R_{N_1} z'$. \\
Thus, Condition (3) of $\hat{N}$-frames gives $\forall z''\in W ~ (x \cancel{R}_{N_2} z'')$, i.e. $x \in Y_0$.
\vspace{1mm}

\noindent Now let $\mathcal{G} \in \mathfrak{F}_1'$. We have to show that $\forall x,y \in W {\setminus} Y_0 ~ (x\leq y \Rightarrow y \leq x)$. Let $x,y \in W {\setminus} Y_0$ and $x \leq y$. Using 1(a), $x \leq y$, $y \leq y$ and $y \notin Y_0$ imply $x R_{N_2} y$. As $\mathcal{G} \in \mathfrak{F}_1'$, $R_{N_2} \subseteq (\leq^{-1})$. Thus, $y \leq x$.
\end{proof}

\begin{obs}
{\rm In \cite{Dosen1986}, for any strictly condensed $J$-frame $(W,\leq,R_N)$, $Y_0$ is defined as follows. 
\[ z \in Y_0 \Leftrightarrow \exists x,y \in W (x \leq z ~\&~ y \leq z ~\&~ x \cancel{R}_N y). \]
However, in Theorem \ref{trans2}, we have defined $Y_0 :=\{x\in W ~|~ \forall z (x \cancel{R}_N z) \}$, in order to give a simpler expression for $Y_0$. In fact, it can be easily seen that the two expressions are equivalent.
%, and 
%The set $Y_0$ obtained from either  expression is, in fact, the same. Indeed, 
%one can easily show that
% in any strictly condensed $J$-frame $(W,\leq,R_N)$, 
%for any $x \in W$,
%\[ \forall z\in W (x \cancel{R}_N z) \mbox{ if and only if } \exists y,z \in W (y \leq x \mbox{ and } z \leq x \mbox{ and } y \cancel{R}_N z). \]
%\sout{($\Leftarrow$). Let $z' \in W$ be such that $xR_{N}z'$. We have to show that, given $y,z \in W$, $y \leq x$ and $z \leq x$ imply $y R_{N}z$. Since $(W,\leq,R_N)$ is a $J$-frame, $xR_N z'$ implies that there exists $z'' \in W$ such that $x\leq z''$, $z' \leq z''$ and $xR_N z''$. In the strictly condensed frames, we have $(\leq R_N \leq^{-1}) \subseteq R_N$. Therefore, $y \leq x$, $xR_N z''$ and $z\leq x \leq z''$ imply that $y R_N z$.\\
%($\Rightarrow$). Suppose for all $y,z \in W$, if $y \leq x$ and $z \leq x$, then $ y R_N z$. We have to find a $z' \in W$ such that $x R_N z'$. Take $y=z=x$.}
}
\end{obs}

In Theorem \ref{trans1}, we have obtained a mapping $\Phi: \mathfrak{F}_2 \rightarrow \mathfrak{F}_1$ from sub-normal frames to $\hat{N}$-frames. In fact, the restriction $\Phi|_{\mathfrak{F}_2'}$ of $\Phi$ is also a map from the subclass $\mathfrak{F}_2'$ of sub-normal identity frames to the subclass $\mathfrak{F}_1'$ of $\hat{N}'$-frames. In the same way, in Theorem \ref{trans2}, $\Psi: \mathfrak{F}_1 \rightarrow \mathfrak{F}_2$ is a map from $\hat{N}$-frames to sub-normal frames. The restriction $\Psi|_{\mathfrak{F}_1'}$ of $\Psi$ is then a map from the subclass $\mathfrak{F}_1'$ of $\hat{N}'$-frames to the subclass $\mathfrak{F}_2'$ of sub-normal identity frames. We also get the following easily.

\begin{theorem}
\label{trans3}
Consider the maps $\Phi$ and $\Psi$ obtained in Theorems \ref{trans1} and \ref{trans2}.
\begin{enumerate}[label={$\arabic*$.}, leftmargin=1.5 \parindent, topsep=0pt, noitemsep]
\item 
\begin{enumerate}[label={$(\alph*)$}, leftmargin=1.5 \parindent, topsep=0pt, noitemsep]
\item $\Phi \Psi$ and $\Psi \Phi$ are identity maps on the classes $\mathfrak{F}_1$ and $\mathfrak{F}_2$ respectively.
%of $\hat{N}$-frames and sub-normal frames respectively.
\item $\Phi|_{\mathfrak{F}_2'} \Psi|_{\mathfrak{F}_1'}$ and $\Psi|_{\mathfrak{F}_1'} \Phi|_{\mathfrak{F}_2'}$ are identity maps on the classes $\mathfrak{F}_1'$ and $\mathfrak{F}_2'$ respectively.
%of $\hat{N}$-frames and sub-normal frames respectively.
\end{enumerate}
\item For any formula $\alpha\in F$, (a) $\mathfrak{F}_1 \vDash \alpha \Leftrightarrow \mathfrak{F}_2 \vDash \alpha $ and (b) $\mathfrak{F}_1' \vDash \alpha \Leftrightarrow \mathfrak{F}_2' \vDash \alpha $.
%\begin{enumerate}
%\item $\mathfrak{F}_1 \vDash \alpha \Leftrightarrow \mathfrak{F}_2 \vDash \alpha $, for any formula $\alpha\in F$.
%\item $\mathfrak{F}_1' \vDash \alpha \Leftrightarrow \mathfrak{F}_2' \vDash \alpha $, for any formula $\alpha\in F$.
%\end{enumerate}
\end{enumerate}
\end{theorem}

We have already proved that the class $\mathfrak{F}_1$ ($\mathfrak{F}_1'$) of all $\hat{N}$-frames ($\hat{N}'$-frames) determines ILM (ILM-$\vee$) in Theorem \ref{comp-ILM-dosen}. Using Theorem \ref{trans3}(2), we obtain the following corollary.

\begin{corollary}\label{char-subnILM}
The classes $\mathfrak{F}_2$ of sub-normal frames and $\mathfrak{F}_2'$ of sub-normal identity frames determine ${\rm ILM}$ and ${\rm ILM}$-${\vee}$ respectively.
\end{corollary}

As pointed out in Theorem \ref{char-JP'}, the class of sub-normal frames determines the logic ${\rm JP'}$. From Corollary \ref{char-subnILM}, we have obtained that the same class of frames determines ${\rm ILM}$. %Using this we can obtain the finite model property. 
The logic ${\rm JP'}$ has the finite model property (FMP) with respect to the class of sub-normal frames, and being finitely axiomatizable, is decidable as well \cite{Goldblatt1974}.
% i.e. for any ${\rm JP'}$-formula $\alpha$, $\alpha$ is a theorem in ${\rm JP'}$ if and only if $\alpha$ is valid in the class of finite sub-normal frames.
As a result, one can obtain FMP and decidability for ${\rm ILM}$.
%,  using those properties of ${\rm JP'}$. %Thus, we have the same subclass of $j$-frames - $\mathfrak{F}_2$ - characterizing two 

We have shown earlier in Theorem \ref{prop-JP'ILM}, ${\rm JP'}$ and ${\rm ILM}$ cannot be equivalent. The two logics have different alphabets, the latter having the extra propositional constant $\bot$, not definable in ${\rm JP'}$ using other connectives. 
This indicates limitations of these relational frames  -- non-equivalent logics of the above kind cannot be  differentiated through them.

\section{The logics \texorpdfstring{$K_{im}$}~ and \texorpdfstring{$K_{im-\vee}$}~ without implication}
\label{sec-bdll}

As mentioned in Section \ref{intro}, we now turn to an investigation of the two negation operators introduced through  $ccpBa$s, adopting the approaches of Dunn and Vakarelov.  Consider an alphabet that has propositional variables $p,q,r, \ldots$, binary connectives $\vee$ and $\wedge$, unary connectives $\neg$ and ${\sim}$, and constants $\bot$ and $\top$. Define a language $\tilde{\mathcal{L}}$ with this alphabet and class $\tilde{F}$ of well-formed formulas  given by the scheme:
\[  p \mid \top \mid \bot \mid \alpha \wedge \beta \mid \alpha \vee \beta \mid \neg \alpha \mid {\sim} \alpha \]

\noindent The logical consequence is defined as a pair of formulas $(\phi,\psi)$, written as $\phi \vdash \psi$ and called a {\it sequent}. The rule expressing `if $\alpha \vdash \beta$ then $\gamma \vdash \delta$' is written as $\alpha \vdash \beta$ / $\gamma \vdash \delta$.

\begin{definition}[The logics $\bm{K_{im}}$ and $\bm{K_{im-\vee}}$ \cite{anuj2019}]
\label{def-Kim}
The language of $K_{im}$ is $\tilde{\mathcal{L}}$. The axioms and rules of $K_{im}$ are as follows:
\begin{multicols}{2}
\begin{enumerate}[label={A$\arabic*$.}, leftmargin=2 \parindent, noitemsep]
\item $\alpha \vdash \alpha$
\item $\alpha \vdash \beta$, $\beta \vdash \gamma$ / $\alpha \vdash \gamma$
\item $\alpha \wedge \beta \vdash \alpha$; $\alpha \wedge \beta \vdash \beta$
\item $\alpha \vdash \beta$, $\alpha \vdash \gamma$ / $\alpha \vdash \beta \wedge \gamma$
\item $\alpha \vdash \gamma$, $\beta \vdash \gamma$ / $\alpha \vee \beta \vdash \gamma$
\item $\alpha \vdash \alpha \vee \beta$; $\beta \vdash \alpha \vee \beta$
\item $\alpha \wedge (\beta \vee \gamma) \vdash (\alpha \wedge \beta) \vee (\alpha \wedge \beta)$
\item $\alpha \vdash \top$ (Top)
\item $\bot \vdash \alpha$ (Bottom)
\item $\alpha \vdash \beta$ / $\neg \beta \vdash \neg \alpha$
\item $\neg \alpha \wedge \neg \beta \vdash \neg (\alpha \vee \beta)$
\item $\top \vdash \neg \bot$
\item $\alpha \vdash \neg \neg \alpha$
\item $\alpha \wedge \beta \vdash \gamma$ / $\alpha \wedge \neg \gamma \vdash \neg \beta$
\item $\alpha \wedge \neg \alpha \vdash \beta$
\item ${\sim} \alpha \vdash \neg (\alpha \wedge \neg {\sim} \top)$
\item $\neg (\alpha \wedge \neg {\sim} \top) \vdash {\sim} \alpha$
\end{enumerate}
\end{multicols}
\noindent The logic $K_{im-\vee}$ is $K_{im}$ enhanced with the following axiom.
\begin{enumerate}[label={A$\arabic*$.}, leftmargin=2.3 \parindent, noitemsep]
\setcounter{enumi}{17}
\item $\top \vdash \alpha \vee {\sim} \alpha$
\end{enumerate}
\end{definition}

\noindent Derivability is defined in the standard manner. Following the nomenclature in \cite{Dunn2005}, axioms A1-A7 give the 
%positive fragment of the logic, called 
Distributive Lattice Logic, while A1-A9 give the Bounded Distributive Lattice Logic (BDLL). A10-A12 are the axioms and rules that make the negation $\neg$ {\it preminimal}, and further, A13 and A14 make it {\it minimal}; adding A15 makes it {\it intuitionistic}. Note that minimal or intuitionistic negation defined here is different from that mentioned in Remark \ref{negations}.
BDLL, along with preminimal negation $\neg$, is denoted by $K_i$.
%, and the consequence relation `$\vdash$' here differs from the `$\vdash$' defined for ${\rm ML}$.

\begin{proposition}
\label{prop-Kim}
The following can be proved in the system $K_{im}$:
\begin{enumerate}[label={{\rm P}$\arabic*$.}, leftmargin=3 \parindent, noitemsep]
\item $\alpha \vdash \beta$, $\delta \vdash \gamma$ / $\alpha \wedge \delta \vdash \beta \wedge \gamma$
\item ${\sim}$-Contraposition: $\alpha \vdash \beta$ / ${\sim} \beta \vdash {\sim} \alpha$
\item ${\sim}$-$\vee$-Linearity: ${\sim} \alpha \wedge {\sim} \beta \vdash {\sim} (\alpha \vee \beta)$
\item ${\sim}$-Nor: $\top \vdash {\sim} \bot$
\item $\alpha \vdash {\sim} {\sim} \alpha$
\item $\alpha \wedge \beta \vdash \gamma$ / $\alpha \wedge {\sim} \gamma \vdash {\sim} \beta$
\item {\rm (DNE(${\sim}\top$))} ${\neg} \neg {\sim} \top \vdash {\sim} \top$
\end{enumerate}
\end{proposition}
\noindent `DNE' in P7 stands for `double negation elimination'.
\begin{proof} We only prove P$7$.
$\top \wedge \neg {\sim} \top \vdash \neg {\sim} \top$ (using A$3$). A$10$ implies $\neg \neg {\sim} \top \vdash \neg (\top \wedge \neg {\sim} \top)$. Finally using A$17$ and A$2$, we get $\neg \neg {\sim} \top \vdash {\sim} \top$.
\end{proof}

\noindent Here, P$2$, P$3$ and P$4$ make the negation ${\sim}$ a preminimal negation. Further, P$5$ and P$6$ 
%are the versions of $A13$ and $A14$, where $\neg$ is replaced by ${\sim}$. Thus, the negation 
make ${\sim}$ minimal. In fact, we have the following `equivalent' version $K_{im}'$ of $K_{im}$.

\begin{definition}[The logic $\bm{K_{im}'}$] \label{def-Kim2}  The language of $K_{im}'$ is $\tilde{\mathcal{L}}$. The axioms and rules are A$1$-A$15$, P$2$-P$6$ and P$7$ {\rm (DNE(${\sim}\top$))}.
\end{definition}

\begin{theorem}
\label{prop-Kim2}
For any $\alpha,\beta \in \tilde{F}$, $\alpha \vdash_{K_{im}} \beta$ if and only if $\alpha \vdash_{K_{im}'} \beta$. 
\end{theorem}
\begin{proof}
We shall only see the proofs of A$16$ and A$17$ in the logic $K_{im}'$.\\
A16: Using A3, $\alpha \wedge \top \vdash \alpha$. Then P6 implies $
\alpha \wedge {\sim} \alpha \vdash  {\sim} \top$. A14 then implies $\alpha \wedge \neg {\sim}\top \vdash \neg {\sim}\alpha$. Finally, A10 gives 
$\neg \neg {\sim} \alpha \vdash  \neg (\alpha \wedge \neg {\sim} \top)$\\
A17: Using A1 and A14, $\alpha \wedge \neg(\alpha \wedge \neg {\sim}\top) \vdash \neg \neg {\sim} \top$. Using P7 and A2, $\alpha \wedge \neg(\alpha \wedge \neg {\sim}\top) \vdash {\sim} \top$. P6 then implies ${\sim}{\sim} \top \wedge \neg(\alpha \wedge \neg {\sim}\top) \vdash {\sim} \alpha$.
%\begin{flalign*}
%\textrm{A}16:&&& \alpha \wedge \top \vdash \alpha && \tag{Using A$3$}\\
%&&\Rightarrow ~ & \alpha \wedge {\sim} \alpha \vdash  {\sim} \top && \tag{Using P$6$}\\
%&&\Rightarrow ~& \alpha \wedge \neg {\sim}\top \vdash \neg {\sim}\alpha && \tag{Using A$14$}\\
%&&\Rightarrow ~& \neg \neg {\sim} \alpha \vdash  \neg (\alpha \wedge \neg {\sim} \top) &&\tag{Using A$10$}\\[5pt]
%\textrm{A}17:&&&\alpha \wedge \neg {\sim}\top \vdash \alpha \wedge \neg {\sim}\top && \tag{Using A$1$}\\
%&&\Rightarrow ~& \alpha \wedge \neg(\alpha \wedge \neg {\sim}\top) \vdash \neg \neg {\sim} \top && \tag{Using A$14$}\\
%&&\Rightarrow ~& \alpha \wedge \neg(\alpha \wedge \neg {\sim}\top) \vdash {\sim} \top && \tag{Using P$7$ and A$2$}\\
%&&\Rightarrow ~& {\sim}{\sim} \top \wedge \neg(\alpha \wedge \neg {\sim}\top) \vdash {\sim} \alpha && \tag{Using P$6$}
%\end{flalign*}
We have $\top \vdash {\sim}{\sim} \top$ using P$5$. Now, using P$1$ and A$2$, $\top \wedge \neg(\alpha \wedge \neg {\sim}\top) \vdash {\sim} \alpha$. A$8$ and P$1$ imply $\neg(\alpha \wedge \neg {\sim}\top) \vdash  {\sim} \alpha$.
\end{proof}

Henceforth, we shall consider $K_{im}'$ instead of $K_{im}$.  

\subsection{Algebraic semantics for \texorpdfstring{$K_{im}$}~ and \texorpdfstring{$K_{im-\vee}$}~}

%Let us look at the algebraic semantics for $K_{im}$ and $K_{im-{\vee}}$. As the name suggests, the algebraic class corresponding to $DLL$ is the class of distributive lattices.
%A {\rm distributive lattice with preminimal negation} $\mathcal{A}$ is of the form $(A,1,0,\vee,\wedge,{\sim})$, where the reduct $(A,1,0,\vee,\wedge)$ is a bounded distributive lattice, satisfying the following properties: For all $a,b \in A$,\\
%(1). $a \leq b \Rightarrow {\sim} b \leq {\sim} a$,\\
%(2). ${\sim} a \wedge {\sim} b \leq {\sim} (a \vee b)$, and\\
%(3). $1 = {\sim} 0$.\\
%If a distributive lattice with preminimal negation $\mathcal{A}:=(A,1,0,\vee,\wedge,{\sim})$ satisfies $a \leq {\sim} {\sim} a$ and $a \wedge b \leq c \Rightarrow a \wedge {\sim} c \leq {\sim} b$, for all $a,b,c \in A$, then it is called a {\rm distributive lattice with minimal negation}.\\
%Moreover, if a distributive lattice $\mathcal{A}$ with minimal negation satisfies $a \wedge {\sim} a \leq b$ for all $a,b \in A$, then it is called a {\rm distributive lattice with intuitionistic negation}. We have the following relationships:\\
%(1). $K_i$ corresponds to the class of distributive lattices with preminimal negation.\\
%(2). The logic $K_i^m$ with minimal negation corresponds to the class of distributive lattices with minimal negation.\\
%(3). The logic $K_i^i$ with intuitionistic negation corresponds to the class of distributive lattices with intuitionistic negation.

Let us define the  `$K_{im}$-algebras'. For that, let us note the terminologies related to algebras defined in \cite{dunn1993}. A {\it distributive lattice with preminimal negation}  is of the form $(A,1,0,\vee,\wedge,{\sim})$, where the reduct $(A,1,0,\vee,\wedge)$ is a bounded distributive lattice, and $\sim$ satisfies the following properties, for all $a,b \in A$:\\
(1) $a \leq b \Rightarrow {\sim} b \leq {\sim} a$,\\
(2) ${\sim} a \wedge {\sim} b \leq {\sim} (a \vee b)$, and\\
(3) $1 = {\sim} 0$.\\
If the preminimal negation satisfies $a \leq {\sim} {\sim} a$ and $a \wedge b \leq c \Rightarrow a \wedge {\sim} c \leq {\sim} b$, for all $a,b,c \in A$, then it is called a {\it  minimal negation}. If it further satisfies 
%If a distributive lattice with preminimal negation $\mathcal{A}:=(A,1,0,\vee,\wedge,{\sim})$ satisfies $a \leq {\sim} {\sim} a$ and $a \wedge b \leq c \Rightarrow a \wedge {\sim} c \leq {\sim} b$, for all $a,b,c \in A$, then it is called a {\rm distributive lattice with minimal negation}.\\
%Moreover, if a distributive lattice $\mathcal{A}$ with minimal negation satisfies 
$a \wedge {\sim} a \leq b$ for all $a,b \in A$,  it is called an {\it  intuitionistic negation}. Moreover, intuitionistic negation satisfying ${\sim}{\sim} a \leq a$ is called an {\it ortho negation}.
%It is observed in \cite{Dunn2005} that an ortho negation is, in fact, a classical negation, i.e. satisfies $a \wedge {\sim} a = 0$.
%To ensure soundness, the algebra must be of the form $(A,1,0,\vee,\wedge,\neg,{\sim})$ satisfying the following properties:\\
%(a). The reduct $(A,1,0,\vee,\wedge,\neg)$ must be a distributive lattice with intuitionistic negation.\\
%(b). The reduct $(A,1,0,\vee,\wedge,{\sim})$ must be a distributive lattice with minimal negation.

\begin{definition}[$\bm{K_{im}}$-algebras]
\label{Kim}\index{K$_{im}$-algebra}
$ $\\
{\rm
\noindent A {\it $K_{im}$-algebra} $\mathcal{A}$ is a tuple of the form $(A,1,0,\vee,\wedge,\neg,{\sim})$ satisfying the following conditions:\\
(1) the reduct $(A,1,0,\vee,\wedge)$ is a bounded distributive lattice,\\
(2) the negation $\neg$ is an intuitionistic negation,\\
(3) the negation ${\sim}$ is a minimal negation, and\\
(4) $\neg \neg {\sim} 1 = {\sim} 1$.\\
A {\it $K_{im-{\vee}}$-algebra} is a $K_{im}$-algebra $\mathcal{A}$ satisfying
\[ a \vee {\sim} a =1 \mbox{~for all $a \in A$} \tag{EM}. \]
}
\end{definition}
\noindent EM stands for `excluded middle'.\\
\noindent Note that in (4) of Definition \ref{Kim}, ${\sim} 1 \leq \neg \neg {\sim} 1$ holds anyway, as $\neg$ is minimal (being intuitionistic). We denote the other direction by DNE($\sim 1$), i.e.
\[ \neg \neg {\sim} 1 \leq {\sim} 1 \tag{DNE($\sim 1$)}. \] 
% Condition $a \vee {\sim} a =1 $ is denoted by EM (Excluded middle).

Let us now give the algebraic semantics for $K_{im}$ ($K_{im-{\vee}}$). Consider any $K_{im}$-algebra ($K_{im-{\vee}}$-algebra) $\mathcal{A}:=(A,1,0, \vee, \wedge,\neg,{\sim})$. 
A {\it valuation} is a map $v$ from  PV to $A$, and can be extended to all formulas in the language $\tilde{\mathcal{L}}$ in the standard way \cite{HR}. For formulas $\alpha,\beta$ in $\tilde{\mathcal{L}}$, if for all valuations $v$ on $\mathcal{A}$, $v(\alpha)=1$ implies  $v(\beta)=1$, then we say $\alpha \vdash \beta$ is valid in $\mathcal{A}$. If this is true for all $K_{im}$-algebras ($K_{im-{\vee}}$-algebras), then we denote it by $\alpha \vDash_{K_{im}} \beta$ ($\alpha \vDash_{K_{im-{\vee}}} \beta$). It is straightforward to see

\begin{theorem}
\label{kim-sc}
For any $\alpha,\beta$ in $\tilde{\mathcal{L}}$, $\alpha \vdash_{K_{im}} \beta$ ($\alpha \vdash_{K_{im-{\vee}}} \beta$) if and only if $\alpha \vDash_{K_{im}} \beta$ ($\alpha \vDash_{K_{im-{\vee}}} \beta$), i.e. $K_{im}$ ($K_{im-{\vee}}$) is sound and complete with respect to the class of $K_{im}$-algebras ($K_{im-{\vee}}$-algebras).
\end{theorem}

Since the definition of $K_{im}$ is motivated through the logic ${\rm ILM}$, one expects a relationship between $K_{im}$-algebras and $ccpBa$. The following is clear.
%For any $ccpBa$ $\mathcal{A}:=(A,1,0, \vee, \wedge, \rightarrow,\neg,\sim )$, $(A,1,0, \vee, \wedge)$ is a bounded distributive lattice, $\neg$ is an intuitionistic negation, and ${\sim}$ is minimal negation.

\begin{proposition}
\label{kim-ccpba}$ $\\
For any $ccpBa$ $(A,1,0, \vee, \wedge, \rightarrow,\neg,\sim )$, the reduct $(A,1,0, \vee, \wedge,\neg,\sim )$ is a $K_{im}$-algebra. 
\end{proposition}

Now the question is whether every $K_{im}$-algebra $(A,1,0, \vee, \wedge,\neg,{\sim})$ can be extended to a {\it ccpBa} $(A,1,0, \vee, \wedge,\rightarrow,\neg,\sim )$? Consider the lattice $L := \mathbb{Z} \times \mathbb{Z}$, 
%= \{(m,n) ~\mid ~ m,n \in \mathbb{Z}\}$, 
the set of pairs of integers, with the usual ordering $\leq:$ $(m,n) \leq (r,s)$ if and only if $m \leq r$ and $n \leq s$. $L$ is a distributive lattice. Define $L' := L \cup \{\hat{0},\hat{1}\}$ ($\hat{0} \neq \hat{1}$), such that $\leq$ is extended to $L'$ in the following way: $\hat{0} \leq (m,n) \leq \hat{1}$ for all $(m,n) \in L$. Addition of $\hat{0}$ and $\hat{1}$ makes the lattice bounded, i.e. $L'$ is a bounded distributive lattice. Define two negations $\neg$ and ${\sim}$ on $L'$ as follows:\\
(1) $\neg (m,n) := \hat{0}$ for all $(m,n) \in L$, $\neg \hat{1} := \hat{0}$, and $\neg \hat{0} := \hat{1}$.\\
(2) ${\sim} a := \hat{1}$, for all $a \in L'$.\\
One can easily check that $\neg$ and ${\sim}$ are intuitionistic and minimal negations respectively.
Therefore, $\mathscr{L}:=(L',\hat{1},\hat{0},\vee,\wedge,\neg,{\sim})$ is a $K_{im}$-algebra. For $\mathscr{L}$ to be extended to a $ccpBa$, we must be able to define an operator `$\rightarrow$', such that $(L',\hat{1},\hat{0},\vee,\wedge,\rightarrow)$ is an {\it rpc} lattice, i.e the following holds: for all $a,b,x \in L'$, $a \wedge x \leq b  \Leftrightarrow x \leq a \rightarrow b$.\\
Let $a:=(1,0)$ and $b:=(0,1)$. The possible choices of $x$ for which $a \wedge x \leq b$ are from the set $\{\hat{0}\} \cup \{(m,n) \in \mathbb{Z} \times \mathbb{Z} ~|~ m \leq 0\}$. Since $x \leq a \rightarrow b$ for all such $x$, the only possible choice for $a \rightarrow b$ is $\hat{1}$. However, this value would make the converse false, because $(1,1) \leq \hat{1}$, but $(1,0) \wedge (1,1) = (1,0) \nleq (0,1)$. Thus, we have the following.

\begin{proposition}
\label{ccpba-kim}
$ $\\
There exists a $K_{im}$-algebra $\mathcal{A}:=(A,1,0, \vee, \wedge,\neg,\sim )$ such that there is no binary operator $\rightarrow$ on $A$ that makes the algebra $(A,1,0, \vee, \wedge, \rightarrow,\neg,\sim )$ a $ccpBa$.
\end{proposition}

\noindent  Propositions \ref{kim-ccpba} and \ref{ccpba-kim} demonstrate that even though there is a $K_{im}$-algebra that cannot be extended to a $ccpBa$, the  properties of the two negation operators  in a $K_{im}$-algebra are enough to capture the `non-implicative' version of $ccpBa$.

%The aim of showing the above two propositions is that, even though there is a $K_{im}$-algebra that cannot be extended to a $ccpBa$, the negation properties in a $K_{im}$-algebra are enough to capture the `non-implicative' version of $ccpBa$, as the reduct of any $ccpBa$ is a $K_{im}$-algebra (by removing the operator  $\rightarrow$ from the algebra).
%Further, note that (i) not every minimal negation, however, satisfies the condition $\neg \neg {\sim} 1 = {\sim} 1$, and (ii) the ${\sim}$ in $K_{im}$-algebra need not be intuitionistic. 
%For (i), the $K_{im}$-algebra considered to establish Proposition \ref{ccpba-kim},  $\mathcal{L}':=(L',\hat{1},\hat{0},\vee,\wedge,\neg,{\sim})$, suffices. For $a:=(m,n) \in L$, $a \wedge {\sim} a = a \wedge \hat{1} = a \neq \hat{0}$.\\
%For (ii), we have already shown examples of 6-element {\it ccpBa} (whose reduct forms a $K_{im}$) in Section \ref{sec-eg-ccpba}, where ${\sim}$ is not intuitionistic. what have you written? Correct and insert in proof below.

\begin{proposition}
\label{kim-kite}
$ $
\begin{enumerate}[label={{\rm \arabic*.}}, leftmargin=2 \parindent, noitemsep, topsep=0pt]
\item There exists a bounded distributive lattice $\mathcal{A}:=(A,1,0,\vee,\wedge,\neg,{\sim})$ with minimal negation ${\sim}$ and  intuitionistic negation ${\neg}$ such that $\neg \neg {\sim} 1 \neq {\sim} 1$.
\item There exists a $K_{im}$-algebra $\mathcal{A}:=(A,1,0,\vee,\wedge,\neg,{\sim})$  such that ${\sim}$ is not an intuitionistic negation.
\end{enumerate}
\end{proposition}
\begin{proof}$ $\\
$1.$ Consider the {\it pBa} $\mathcal{H}_6:=(H_6,1,0,\vee,\wedge,\rightarrow,\neg)$ with intuitionistic negation ${\neg}$ (Figure \ref{fig-6rpcA}). $(H_6,1,0,\vee,\wedge)$ is thus a bounded distributive lattice.  Define ${\sim}$ as follows: ${\sim} 1 :=w$ and for all $a \in H_6$, ${\sim} a := a \rightarrow {\sim} 1$. Then it can be shown that  ${\sim}$  is a minimal negation. But, ${\sim} 1 = w \neq \neg \neg {\sim} 1 ~(=1)$.
\vspace{1mm}

\noindent $2$. The $K_{im}$-algebra, $\mathscr{L}:=(L',\hat{1},\hat{0},\vee,\wedge,\neg,{\sim})$, considered to establish Proposition \ref{ccpba-kim}, suffices. For $a:=(m,n) \in L$, $a \wedge {\sim} a = a \wedge \hat{1} = a \neq \hat{0}$.
\end{proof}

A comprehensive analysis of different negations in any bounded distributive lattice is given in \cite{Dunn2005}, where a Kite-like figure is obtained with negations as nodes. Each node in the figure corresponds to a unique property of negation. A path connecting two nodes (that is two properties of negation) $A$ and $B$ where $A$ is above $B$, represents the fact that the negation at  $B$ holds at $A$. Moreover, two paths starting at nodes $A$ and $B$ and meeting at a higher node $C$, implies that the  property at $C$ can be derived from the  properties at $A$ and $B$.

We enhance this Kite diagram to one where each node corresponds to a {\it pair of negations} $(\sim,\neg)$ with $\neg$ as intuitionistic, and only the property of $\sim$ is mentioned against a node. The enhanced kite is called a {\it Kite with negation pair $({\sim},\neg)$}, where $\neg$ is intuitionistic (Figure \ref{fig-dunnkite2}). In such a diagram, due to Proposition \ref{kim-kite},   we can  place the negation ${\sim}$ of the $K_{im}$-algebra  strictly in between the nodes of minimal and intuitionistic negations.

Let us now see the placement of the negation ${\sim}$ of $K_{im-\vee}$-algebras in the Kite with negation pair $({\sim},\neg)$. The following examples of bounded distributive lattices of the type $(A,1,0,\vee,\wedge,\neg,{\sim})$ with two negations, where $\neg$ is intuitionistic, turn out to be useful for the purpose.
\vspace{1mm}

\noindent (1) Consider the $3$-element bounded distributive lattice $(A,1,0,\vee,\wedge)$ (Section \ref{sec-eg-ccpba}), where $A := \{0,a,1\}$ with ordering $0 \leq a \leq 1$. Define $\neg := \neg_1$ and ${\sim} := \neg_1$, where $\neg_1$ is as given in Table \ref{3pba-neg}. Then ${\sim}$ is intuitionistic, but $(A,1,0,\vee,\wedge,\neg_1,\neg_1)$ is not a $K_{im-\vee}$-algebra (as $a \vee \neg_1 a = a \neq 1$). Thus, ${\sim}$ of $K_{im-\vee}$-algebras cannot lie between `Intuitionistic' and DNE(${\sim}1$).
\vspace{1mm}

\noindent (2) For the same $3$-element bounded distributive lattice, define $\neg :=\neg_1$ and ${\sim}:={\sim}_1$, where $\neg_1$ and ${\sim}_1$ are as given in Tables \ref{3pba-neg} and \ref{3ccpba-neg} respectively. Then $(A,1,0,\vee,\wedge,\neg_1,{\sim}_1)$ is a $K_{im-\vee}$-algebra, but ${\sim}_1$ is not intuitionistic ($a \wedge {\sim}_1 a = a \neq 0$). Thus, ${\sim}$ of $K_{im-\vee}$-algebras cannot be in the path above the node `Intuitionistic'.
\vspace{1mm}

\noindent (3) For the same $3$-element bounded distributive lattice, define $\neg :=\neg_1$ and ${\sim}$ as identity map. Then ${\sim}$ is a De Morgan negation, but $(A,1,0,\vee,\wedge,\neg_1,{\sim})$ is not a $K_{im-\vee}$-algebra. This implies that ${\sim}$ of $K_{im-\vee}$-algebras cannot be in the path below the node `De Morgan'.
\vspace{1mm}

\noindent (4) Consider the $6$-element {\it pBa} $(H_6,1,0,\vee,\wedge,\rightarrow,\neg)$ (Figure \ref{fig-6rpcA}).
Recall the {\it c${\vee}$cpBa} $(H_6,1,0,\vee,\wedge,\rightarrow,\neg,{\sim})$ mentioned in Example (C)(3) of Section \ref{sec-eg-ccpba}.\\
The reduct $(H_6,1,0,\vee,\wedge,\neg,{\sim})$ is a $K_{im-\vee}$-algebra. However, ${\sim}$ is not De Morgan (${\sim}{\sim}w = 1 \nleq w$). So ${\sim}$ of $K_{im-\vee}$-algebras cannot be on the path connecting the nodes `De Morgan' and `Ortho'.
\vspace{1mm}

\begin{figure}[ht]
\begin{tikzpicture}[scale=1]
        \draw (0,0) -- (0,0);
        \draw  (3,0) -- (5,2);
		\draw [fill] (3,0) circle [radius = 0.075];
		\node [left] at (2.75,0) {$De$ $Morgan$};
		\node [below] at (1.65,-0.15) {${\sim}{\sim} a \leq a $};
        \draw  (5,2) -- (7,1);
		\draw [fill] (7,1) circle [radius = 0.075];
		\node [right] at (7.2,0.7) {$a \wedge {\sim} a = 0$};
		\node [above] at (8.6,0.8) {$Intuitionistic$};
        \draw  (7,1) -- (7,-1);
		\draw [fill] (7,-1) circle [radius = 0.075];
		\node [right] at (7.2,-1.3) {$a \wedge b \leq c \Rightarrow a \wedge {\sim} c \leq {\sim} b$};
		\node [above] at (8.2,-1.15) {$Minimal$};
        \draw  (7,-1) -- (5,-2);
		\draw [fill] (5,-2) circle [radius = 0.075];
		\node [right] at (5.25,-2.6) {$a \leq {\sim}{\sim} a$};
		\node [right] at (5.2,-2.1) {$Quasi$-$minimal$};
		\draw (3,0) -- (5,-2);
		%\draw (8,-2) node [label=left:{\it $ a \leq \sim \sim a$}]{};
		%\draw (8,-4) node [label=right:{\it $ a \leq b \Rightarrow \sim b \leq \sim a$}]{};
		\draw  (5,-2) -- (5,-4);
		\draw [fill] (5,-4) circle [radius = 0.075];
		%\node [left] at (4.75,-4) {${\sim} a \wedge {\sim} b \leq {\sim} (a \vee b) $};
		\node [right] at (5.25,-4) {$Preminimal$};
		\draw [fill] (7,0) circle [radius = 0.075];
		\node [right] at (7.2,-0.02) {DNE(${\sim}1$)};
		\node [right] at (9.2,-0.02) {$\neg \neg {\sim} 1 \leq {\sim} 1$};
		\draw [fill] (5,2) circle [radius = 0.075];
        \node [above] at (5,2.25) {$Ortho$};
        
        \node [below] at (4.7,0.85) {EM};
        \node [below] at (4.7,0.45) {$a \vee {\sim} a = 1$};
        \draw [fill] (5,1) circle [radius = 0.075];
        \draw  (7,0) -- (5,1);
        \draw  (5,1) -- (5,2);
\end{tikzpicture}
\caption{Kite with negation pair $({\sim},\neg)$, where $\neg$ is intuitionistic}
\label{fig-dunnkite2}
\end{figure}
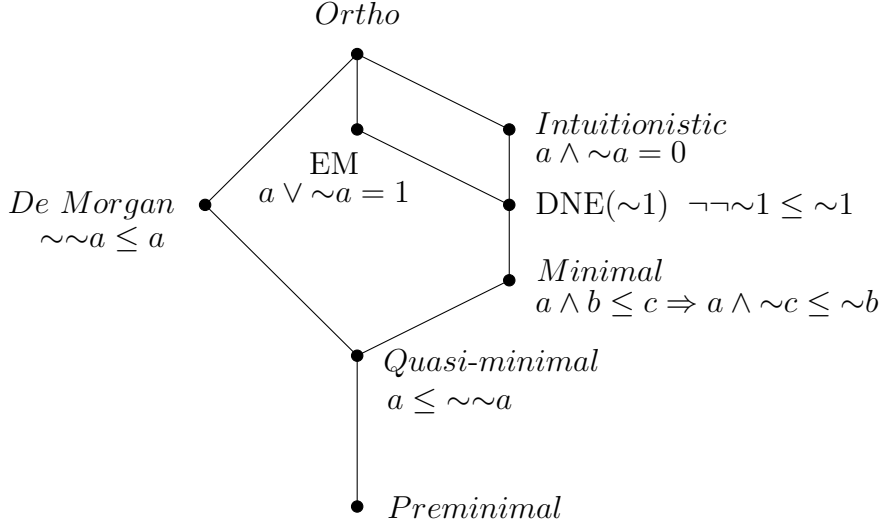

\noindent Thus, ${\sim}$ of $K_{im-\vee}$-algebra has to be placed on a separate path connecting DNE(${\sim}1$) and Ortho in the Kite with negation pair $({\sim},\neg)$ (Figure \ref{fig-dunnkite2}), marked by a node with property EM. Note also that $a \vee {\sim} a = 1$ and $a \wedge {\sim} a =0$ for all $a \in A$ imply that ${\sim}$ is ortho \cite{dunn1993}.
%So, the nodes `EM' and `Intuitionistic' meet at `Ortho'.

\subsection{Relational semantics for \texorpdfstring{$K_{im}$}~ and \texorpdfstring{$K_{im-\vee}$}~}
\label{relation-Kim}

%We next turn to relational semantics for the logics. 
One of the motivations of Dunn was to study negation as a modal impossibility operator in `compatibility' frames, in line with the work on negation 
%and consistency 
by Vakarelov \cite{Vakarelov1989}. Let us first see what `compatibility' frames are, and define the semantics over such frames for the logic $K_i$.

\begin{definition}[Compatibility frames \cite{Dunn2005}]
\label{compframe}
$ $\\
A tuple $\mathcal{F}:=(W,C ,\leq)$ is called a {\it compatibility frame} if $(W,\leq)$ is a poset, 
%$\leq$ is a partial order on the set $W$, 
and $C$ is a binary relation on $W$ satisfying the following condition for all $x,x',y,y' \in W$:
\[\mbox{If } x' \leq x, ~ y' \leq y \mbox{ and } xCy \mbox{ then } x'Cy'. \tag{C} \]
%$\mathcal{F}$ is called a {\it compatibility frame}.
\end{definition}

Let us compare  the strictly condensed $N$-frames mentioned in Remark \ref{Nframe} of Section \ref{sec-KripkeILM} with  compatibility frames.

\begin{obs}$ $\\
{\rm $(W,\leq,R_N)$ is a strictly condensed $N$-frame if and only if  $({\leq R_N \leq^{-1}}) \subseteq R_N$. So,
\begin{align*}
& \forall x,y,x',y' \in W\bigl(({x' \leq x} ~\&~ {x R_N y}~\&~ {y' \leq y})\Rightarrow {x' R_N y'}\bigr).  
\end{align*}
According to Definition \ref{compframe} therefore, $(W,R_N,\leq)$ is a compatibility frame. Conversely, it is clear that any compatibility frame gives a strictly condensed $N$-frame.
In other words, the class of strictly condensed $N$-frames  is the same as that of compatibility frames. 
 }
\end{obs}

The definitions of valuations and truth of a formula $\alpha$ at $x \in W$ under a valuation  on a compatibility frame $\mathcal{F}:=(W,C ,\leq)$ (denoted  $x \vDash \alpha$)  are then given in the standard way \cite{Dunn2005}.
%Definition \ref{ILM-truth}.
For the compatibility frame $\mathcal{F}:=(W,C,\leq)$, the pair $\mathcal{M}:=(F, \vDash)$ is called a {\it model} of $K_i$. 
%A consequence pair $(\alpha, \beta)$ is {\rm valid} in the model $\mathcal{M}$, denoted by $\alpha \vDash_\mathcal{M} \beta$, if for all $x \in W$, $x \vDash \alpha \Rightarrow x \vDash \beta$.
%The {\rm validity} of a consequence pair $(\alpha, \beta)$ in a frame $\mathcal{F}$, denoted as $\alpha \vDash_\mathcal{F} \beta$, is defined as $\alpha \vDash_\mathcal{M} \beta$ for every model $\mathcal{M}$ on the frame $\mathcal{F}$.
A consequence pair $(\alpha, \beta)$ is {\it valid} in a compatibility frame $\mathcal{F}$, denoted as $\alpha \vDash_\mathcal{F} \beta$, when for  every model $\mathcal{M}$ on the frame $\mathcal{F}$ and for all $x \in W$, $x \vDash \alpha \Rightarrow x \vDash \beta$.
%For a class $\mathfrak{F}$ of  compatibility frames, if $\alpha \vDash_\mathcal{F} \beta$ for every $\mathcal{F} \in \mathfrak{F}$, we write $\alpha \vDash_\mathfrak{F} \beta$. 

It can be checked using induction on the number of connectives in a formula, that the following {\it hereditary} condition
%, called as `{\it hereditary condition}', 
holds for any formula $\alpha$ in $K_i$:\\
\centerline{for all $x,y \in W$, $x \leq y$ and $x \vDash \alpha$ imply $y \vDash \alpha.$}

%We have the following soundness and completeness results for the logic $K_i$.

\begin{theorem}
\label{Kicomplete}
{\rm \cite{Dunn2005}} For any two formulas $\alpha$ and $\beta$, 
$\alpha \vdash_{K_i}\beta$ if and only if $\alpha \vDash_{\mathcal{F}} \beta$ for any compatibility frame $\mathcal{F}:=(W,C,\leq)$.
%, the following are equivalent:
%\begin{enumerate}
%\item $\alpha \vdash_{K_i}\beta$.
%\item $\alpha \vDash_{\mathcal{F}} \beta$ for any compatibility frame $\mathcal{F}:=(W,C,\leq)$.
%\end{enumerate}
\end{theorem}

As $K_{im}$ is an `extension' of $K_i$, it is then expected that the former will be sound and complete with respect to some class of compatibility frames. Let us now define {\it sub-compatibility frames}.

\begin{definition}[Sub-compatibility frame]
\label{subcompframe}\index{sub-compatibility frame}
Let $\mathcal{F}:=(W,C,\leq)$ be a compatibility frame such that 
\begin{enumerate}[label={($\arabic*$)}, noitemsep, topsep=0pt]
\item $C$ is symmetric,
\item $\forall x, y \in W (xCy \Rightarrow \exists z \in W (x \leq z ~\&~ y \leq z ~\&~ xCz)) \mbox{, and}$
\item $\forall x \in W \bigl(\forall y \in W \bigl(x \leq y \Rightarrow \exists z \in W (y \leq z ~\&~ \forall z' \in W~ (z \cancel{C} z'))\bigr) \nonumber \\
~\hfill \Rightarrow \forall z'' \in W ~ (x \cancel{C} z'')\bigr)$.
\end{enumerate}
$\mathcal{F}$ is called a {\it sub-compatibility frame}. $\mathfrak{F}_3$ denotes the class of all such frames.  \\
Sub-compatibility frames satisfying the additional condition $C \subseteq (\leq^{-1})$ are called {\it sub-compatibility identity frames}. The class of sub-compatibility identity frames is denoted by $\mathfrak{F}_3'$.
\end{definition}

The definitions of valuation, truth  and validity remain the same.
%can be easily extended to the case of $K_{im}$ with respect to sub-compatibility frames and sub-compatibility models. 
Let us note the  definitions for the truth of a formula $\alpha$ involving $\neg$ and ${\sim}$:
\begin{enumerate}[label={$\arabic*$.}, leftmargin=2 \parindent, topsep=0pt, noitemsep]
\item $x \vDash \neg \alpha$ if and only if $\forall y \in W$, $(x \leq y \Rightarrow y \nvDash \alpha)$.
\item $x \vDash  {\sim} \alpha$ if and only if $\forall y \in W$, $(x C y \Rightarrow y \nvDash \alpha).$
\end{enumerate}

As done in previous sections, let us investigate the special conditions defining the frame here, namely Conditions (2) and (3) in Definition \ref{subcompframe}. In \cite{Dunn2005}, it has been shown that Condition (2) is {\it canonical} to P$6:~\alpha \wedge \beta \vdash \gamma$ / $\alpha \wedge {\sim} \gamma \vdash {\sim} \beta$ from Proposition \ref{prop-Kim}, i.e. P$6$ is valid in a compatibility frame $\mathcal{F}:=(W,C,\leq)$ if and only if $C$ satisfies Condition (2).
\vskip 3pt 
\noindent Now recall the sequent P$7$ ${\rm (DNE({\sim}\top))}: {\neg} \neg {\sim} \top \vdash {\sim} \top$ from Proposition \ref{prop-Kim}. Validity of ${\rm DNE({\sim}\top)}$ in a sub-compatibility model $\mathcal{M}:=((W,C,\leq), \vDash)$ means the following.
\begin{align*}
\forall x \in W & (x \vDash \neg \neg {\sim} \top \Rightarrow x \vDash {\sim} \top))\\
\Leftrightarrow ~\forall x \in W & (\forall y \in W (x \leq y \Rightarrow y \nvDash \neg {\sim} \top ) \Rightarrow x \vDash {\sim} \top) \\
\Leftrightarrow ~\forall x \in W & \bigl(\forall y \in W (x \leq y \Rightarrow \exists z\in W (y \leq z ~\&~ z \vDash {\sim} \top)) \Rightarrow  x \vDash {\sim} \top\bigr)\\
\Leftrightarrow ~\forall x \in W & \bigl(\forall y \in W \bigl(x \leq y \Rightarrow \exists z\in W (y \leq z ~\&~ \forall z'\in W~ (z C z'))\bigr)\\
&  \hspace{67mm} \Rightarrow  \forall z''\in W~ (x C z'')\bigr).
\end{align*}
The last condition is just Condition (3). Thus, we have the following.

\begin{proposition}
Condition {\rm (3)} in Definition \ref{subcompframe} is canonical to {\rm DNE(${\sim}\top$)}.
\end{proposition}

We also note here that the  reverse implication in Condition (3) is always true. Consider the contraposition of the implication, i.e. for any $x \in W$:\\
\centerline{$\exists  y \in W \bigl(x \leq y ~\&~ \forall z \in W (y \leq z \Rightarrow \exists z' \in W~ z C z'))\bigr) \Rightarrow \exists z'' \in W~ (x C z'')$.}
%This is always true:  
Now let $y \in W$ such that $x \leq y$ and $\forall z \in W (y \leq z \Rightarrow \exists z' \in W~ z Cz')$. $y \leq y$ implies there exists $z' \in W$ such that $y Cz'$. Using condition (C), $yCz'$, $x \leq y$ and $z' \leq z'$ imply $xCz'$.

 %What exactly are the Conditions \ref{eq-int} and \ref{eq-dne}?  

 In Definition \ref{def-Kim2}, we have given an equivalent version $K_{im}'$ of $K_{im}$. The axioms/rules involving the negation ${\sim}$ are P$2$-P$7$. Based on the above points, the presence of Conditions (2) and (3) in sub-compatibility frames ensures the validity of P$6$ and P$7$ in the class. Moreover, any sub-compatibility frame is a compatibility frame, and P$2$-P$5$ are valid in any compatibility frame in the language of $K_i$ (Theorem \ref{Kicomplete}). So we have the soundness result for $K_{im}$ with respect to  sub-compatibility frames.
One can obtain the completeness result by following similar steps as given in \cite{Dunn2005}.

\begin{theorem}\label{comp-Kim}
For any formulas $\alpha,\beta \in \tilde{F}$,
\begin{enumerate}[label={$\arabic*$.}, leftmargin=2 \parindent, topsep=0pt, noitemsep]
\item $\alpha \vDash_{\mathfrak{F}_3} \beta \Leftrightarrow \alpha \vdash_{K_{im}} \beta$, and
\item $\alpha \vDash_{\mathfrak{F}_3'} \beta \Leftrightarrow \alpha \vdash_{K_{im-{\vee}}} \beta$.
\end{enumerate}
\end{theorem}

\section{Connections between relational and algebraic semantics}
\label{sec-compare}

We have observed the duality between topological spaces and the algebras in Theorem \ref{dual-ccpba}. Let us now connect the  frames studied in the previous sections and the algebras.
 Kripke \cite{Kripke1965a} linked normal frames and {\it pBa}s (cf. \cite{chagrov,bezhanishvili}): every {\it pBa} $\mathcal{A}$ is embedded in the `complex algebra' of the `canonical frame' of $\mathcal{A}$; every normal frame $\mathcal{F}$  can be embedded into the `canonical' frame of the `complex algebra' of $\mathcal{F}$. Note that an {\it embedding} between two posets $(A,\leq)$ and $(A',\leq')$ is  a map $\phi: A \rightarrow A'$ such that for all $a,b \in A$, $a \leq b \Rightarrow \phi(a) \leq' \phi(b)$~ (Order-preserving) and $\phi(a) \leq' \phi(b) \Rightarrow a \leq b$ (Order-reflecting). Then, an {\it embedding} between two algebras or frames is just defined as an embedding between the underlying posets (preserving the operators/relations). We link sub-normal frames and {\it ccpBa}s similarly.  Embeddings between sub-normal frames and canonical algebras of sub-normal frames are first defined.

\begin{definition}[Embeddings between sub-normal frames]
\label{embeddings}\index{sub-normal frame!embedding}
$ $\\
Given two sub-normal frames $(W,\leq,Y_0)$ and $(W',\leq',Y_0')$, a mapping $f: W \rightarrow W'$ is an {\it embedding} if it satisfies the following conditions for all $a,b \in W$:
\begin{enumerate}[label={$\arabic*$.}, leftmargin=2 \parindent, topsep=0pt, noitemsep]
\item $a \leq b$ if and only if $f(a) \leq' f(b)$ (poset embedding), and 
\item $a \in Y_0$ if and only if $f(a) \in Y_0'$.
\end{enumerate}
\end{definition}

\begin{definition}[Complex algebra of a sub-normal frame]
\label{def-complex}\index{sub-normal frame!complex algebra}
Consider a sub-normal frame $\mathcal{F}:=(W,\leq,Y_0)$. 
%Let $Up(W)$ denote the set of upsets of $W$. 
Define the following operators on $Up(W)$, the set of all upsets of $W$. For any $U,V \in Up(W)$,
\begin{enumerate}[label={$\arabic*$.}, leftmargin=2 \parindent, topsep=0pt, noitemsep]
\item $U \rightarrow V := \{ w \in W ~ \mid ~ \forall v \in W(w \leq v \Rightarrow (v \in U \Rightarrow v \in V )) \}$,
\item $\neg U:= U \rightarrow \emptyset$, i.e. $\neg U:= \{ w \in W ~ \mid ~ \forall v \in W(w \leq v \Rightarrow v \notin U) \}$, and
\item ${\sim} U := U \rightarrow Y_0$, , i.e. ${\sim} U:= \{ w \in W ~ \mid ~ \forall v \in W((w \leq v ~\&~ v \in U) \Rightarrow v \in Y_0) \}$.
\end{enumerate}
The structure $Up(\mathcal{F}):=(Up(W), W, \emptyset, \cap, \cup, \rightarrow, \neg, {\sim})$  is called the {\it complex algebra} of the sub-normal frame $\mathcal{F}$.
\end{definition}

\begin{definition}[Canonical frame of a {\it ccpBa}]
\label{def-canonical}\index{sub-normal frame!canonical}
$ $\\
Consider a {\it ccpBa} $\mathcal{A}:=(A,1,0,\vee,\wedge,\rightarrow,\neg,{\sim})$. Let $X_A$ denote the set of all prime filters in $\mathcal{A}$. Define a set $Y_0$ as follows: $Y_0 := \{P \in X_A ~ \mid ~ {\sim} 1  \in P \}.$\\
Then the triple $\mathcal{F}_{\mathcal{A}} := (X_A,\subseteq,Y_0)$ is called the {\it canonical} frame of $\mathcal{A}$.
\end{definition}
 
It is easy to check that for a sub-normal frame $\mathcal{F}:=(W,\leq,Y_0) \in \mathfrak{F}_2$, the complex algebra  $Up(\mathcal{F})$ of $\mathcal{F}$ forms a {\it ccpBa}. Further, if $\mathcal{F}$ ($\in \mathfrak{F}_2'$) is a sub-normal identity frame then $Up(\mathcal{F})$ is a {\it c${\vee}$cpBa}. On the other hand, given a {\it ccpBa} $\mathcal{A}:=(A,1,0,\vee,\wedge,\rightarrow,\neg,{\sim})$, the canonical frame $\mathcal{F}_\mathcal{A} := (X_A,\subseteq,Y_0)$ is a sub-normal frame. If $\mathcal{A}$ is a {\it c${\vee}$cpBa} then $\mathcal{F}_{\mathcal{A}}$ is a sub-normal identity frame. 
%\in \mathfrak{F}_2'$. 
This to and fro connection between sub-normal frames $\mathcal{F} \in \mathfrak{F}_2$ ($\mathfrak{F}_2'$) and {\it ccpBa}s ({\it c${\vee}$cpBa}s) gives us the following result. The proof is similar to that in the case of  {\it pBa}s.

\begin{theorem}
\label{frame-algebra}
$ $
\begin{enumerate}[label={$\arabic*$.}, leftmargin=2 \parindent, topsep=0pt, noitemsep]
\item Every ccpBa (or c${\vee}$cpBa) $\mathcal{A}$ is embeddable into the complex algebra $Up(\mathcal{F}_{\mathcal{A}})$ of the canonical frame $\mathcal{F}_{\mathcal{A}}$ of $\mathcal{A}$.
\item Any sub-normal (sub-normal identity) frame $\mathcal{F}\in \mathfrak{F}_2$ (or $\mathfrak{F}_2'$) can be embedded into the canonical frame  $\mathcal{F}_{Up(\mathcal{F})}$ of the complex algebra $Up(\mathcal{F})$ of $\mathcal{F}$.
\end{enumerate}
\end{theorem}
\begin{proof}$ $ We just mention the maps involved.\\
1. Consider the complex algebra $Up(\mathcal{F}_{\mathcal{A}}):=(Up(X_A), X_A, \emptyset, \cap, \cup, \rightarrow, \neg, {\sim})$ of the canonical frame $\mathcal{F}_{\mathcal{A}}:=(X_A,\subseteq, Y_0)$ of a {\it ccpBa} $\mathcal{A}:=(A,1,0,\vee, \wedge,\rightarrow,\neg,{\sim})$. The map $h: A \rightarrow Up(X_A)$ such that for all $a \in A$,$h(a):= \{P \in X_A ~|~  a \in P \}$, is a monomorphism from the {\it ccpBa} (or {\it c${\vee}$cpBa}) $(A,1,0, \vee,\wedge, \rightarrow,\neg,{\sim})$ to the {\it ccpBa} (or {\it c${\vee}$cpBa}) $(Up(X_A), \emptyset, X_A, \cap, \cup, \rightarrow, \neg,{\sim})$.\\
2. Consider the canonical frame $(X_{Up(W)}, \subseteq, Y_{0_{Up(W)}})$ of the complex algebra $Up(\mathcal{F}):=(Up(W),W,\emptyset,\cap,\cup,\rightarrow,\neg, {\sim})$ of the sub-normal frame (or sub-normal identity frame) $\mathcal{F}:=(W, \leq, Y_0)$. The map $g: W \rightarrow X_{A}$ such that\\
\centerline{ for all $w \in W$, $g(w):= \{U \in Up(W) ~|~  w \in U \}$,}
is an embedding from $(W,\leq,Y_0)$ to $(X_{Up(W)}, \subseteq,Y_{0_{Up(W)}})$.
\end{proof}

\noindent A similar result can also be obtained for $K_{im}$ and $K_{im}$-${\vee}$. Let us first give the definition of embeddings between sub-compatibility frames and complex algebras of  sub-compatibility frames.

\begin{definition}[Embeddings between sub-compatibility frames]
\label{embeddings2}\index{sub-compatibility frame!embedding}
$ $\\
Given two sub-compatibility frames $(W,C,\leq)$ and $(W',C',\leq')$, a mapping $f: W \rightarrow W'$ is an {\it embedding} if it satisfies the following conditions for all $x,y \in W$:
\begin{enumerate}[label={$\arabic*$.}, leftmargin=2 \parindent, topsep=0pt, noitemsep]
\item $x \leq y$ if and only if $f(x) \leq' f(y)$, and 
\item $xCy $ if and only if $f(x)C'f(y)$.
\end{enumerate}
\end{definition}

\begin{definition}[Complex algebra of a sub-compatibility frame]
\label{def-complex2}\index{sub-compatibility frame!complex algebra}
$ $\\
Given a sub-compatibility frame $\mathcal{F}:=(W,C,Y_0)$, the {\it complex algebra} of the sub-compatibility frame $\mathcal{F}$ is the structure $Up(\mathcal{F}):=(Up(W), W, \cap, \cup, \neg, {\sim})$, where
%Let $Up(W)$ denote the set of upsets of $W$. 
the operators  $\neg,\sim$ on $Up(W)$ are defined as:
\begin{enumerate}[label={$\arabic*$.}, leftmargin=2 \parindent, topsep=0pt, noitemsep]
\item $\neg U:= \{w \in W~|~ \forall v \in W(w \leq v \Rightarrow v \notin U)\}$, and
\item ${\sim} U := \{w \in W~|~ \forall v \in W(w C v \Rightarrow v \notin U)\}$.
\end{enumerate}

\end{definition}

\begin{definition}[Canonical frame of a $K_{im}$-algebra]
\label{def-canonical2}\index{sub-compatibility frame!canonical}
$ $\\
Given a $K_{im}$-algebra $\mathcal{A}:=(A,1,0,\vee,\wedge,\neg,{\sim})$, the triple $\mathcal{F}_{\mathcal{A}} := (X_A,C,\subseteq)$ is called the {\it canonical} frame, where 
%Let $X_A$ denote the set of all prime filters in $\mathcal{A}$. Define 
the relation $C$ on $X_A$ is defined as follows. For any $P,Q \in X_A$,\\
\centerline{$PCQ$ if and only if (for all $a \in A$, ${\sim} a \in P \Rightarrow a \notin Q$).}
%The triple $\mathcal{F}_{\mathcal{A}} := (X_A,C,\subseteq)$ is called the {\rm canonical} frame.
\end{definition}

\noindent We can check that the complex algebra of a sub-compatibility (identity) frame forms a $K_{im}$-algebra ($K_{{im}\mbox{-}{\vee}}$-algebra), and the canonical frame of a $K_{im}$-algebra ($K_{{im}\mbox{-}{\vee}}$-algebra) forms a sub-compatibility (identity) frame. Similar to Theorem \ref{frame-algebra}, this to and fro connection between $K_{im}$-algebras and sub-compatibility frames gives the following result.

\begin{theorem}
\label{frame-algebra2}
$ $
\begin{enumerate}[label={$\arabic*$.}, leftmargin=2 \parindent, topsep=0pt, noitemsep]
\item Every $K_{im}$-algebra (or $K_{{im}\mbox{-}{\vee}}$-algebra) $\mathcal{A}$ can be embedded into the complex algebra $Up(\mathcal{F}_{\mathcal{A}})$ of the canonical frame $\mathcal{F}_{\mathcal{A}}$ of $\mathcal{A}$.
\item Any sub-compatibility (sub-compatibility identity) frame $\mathcal{F}\in \mathfrak{F}_3$ (or $\mathfrak{F}_3'$) can be embedded into the canonical frame  $\mathcal{F}_{Up(\mathcal{F})}$ of the complex algebra $Up(\mathcal{F})$ of $\mathcal{F}$.
\end{enumerate}
\end{theorem}

\section{Conclusions}
\label{conc}

The algebraic classes {\it ccpBa}s and {\it c${\vee}$cpBa}s are studied through examples, properties, representation theorems, and comparison with existing algebras. The corresponding logics ILM and ILM-${\vee}$ are defined, and  ILM is compared with existing logics ML and IL. Further, ILM is shown to be equivalent to a logic which is an extension of ${\rm JP}'$, a special case of Peirce's logic, and hence observed to be decidable. 
%In fact, we can obtain more from this correspondence. The logic ${\rm JP'}$ has the finite model property \cite{Goldblatt1974}, i.e. for any ${\rm JP'}$-formula $\alpha$, $\alpha$ is a theorem in ${\rm JP'}$ if and only if $\alpha$ is valid in the class of finite sub-normal frames. This helps in proving the decidability of the logic ${\rm JP'}$ \cite{Goldblatt1974}. One can then obtain the decidability of ${\rm ILM}$ by using the decidability of ${\rm JP'}$.
A study of the features of the two negations is carried out next through the logics   $K_{im}$ and $K_{{im}\mbox{-}{\vee}}$ that do not have an implication operator. Algebraic and relational semantics of these logics are given. A study of properties of $K_{im}$ and $K_{{im}\mbox{-}{\vee}}$-algebras is conducted. Dunn's Kite of negations is enhanced to define a Kite with negation pair $({\sim},{\neg})$, where $\neg$ is intuitionistic. It is then shown that the negations in the algebras occupy distinct positions in this kite. The property ${\rm DNE}({\sim} 1)$ ($\neg \neg {\sim} 1 \leq {\sim} 1$) leads to a position strictly between the Minimal and Intuitionistic nodes, while the property EM ($a \vee {\sim} a = 1$) yields a new path from ${\rm DNE}({\sim} 1)$ to the Ortho node. Finally, relations between frames and algebras defined in the work are established through duality results.

The study of relational semantics of ILM and ILM-${\vee}$ helps in  understanding the behaviour of the negation operators in the logics. We have given two different semantics  for ILM through the classes $\mathfrak{F}_1$ and $\mathfrak{F}_2$ of frames - $\mathfrak{F}_1$ using Do{\v{s}}en's $N$-frames and $\mathfrak{F}_2$ using Segerberg's $j$-frames. The difference in both lies in the treatment of negation. For $\mathfrak{F}_1$, both negations are considered as unary modal connectives and the semantics is defined using the modal accessibility relations $R_{N_1}$ and $R_{N_2}$. For $\mathfrak{F}_2$, both negations are treated as unary connectives, their semantics being defined using the relation $\leq$ and $Y_0$. This naturally gives the idea of constructing new frame classes and corresponding semantics, by considering one of the two negations of {\rm ILM} as a unary connective, and the other negation as a unary `impossibility' modal connective - of the types $(X,\leq,R_{N_2})$ and $(X,\leq,R_{N_1},Y_0)$. In fact, we have seen the frames of the type $(X,\leq,R_{N_2})$ as sub-compatibility frames (where the relation $R_{N_2}$ was represented by $C$). It is then naturally expected that, by suitably defining conditions on relations $R_{N_2}$, $R_{N_1}$ and $Y_0$, {\rm ILM} ({\rm ILM}-${\vee}$) can also be determined by the class of frames of the types $(X,\leq,R_{N_2})$ and $(X,\leq,R_{N_1},Y_0)$. An inter-translation between these classes can also be obtained using the following relations.
\begin{enumerate}[label={{\rm (\arabic*)}}, leftmargin=1.5 \parindent, noitemsep, topsep=0pt]
\item $x R_{N_1} y\Leftrightarrow \exists z (x \leq z \mbox{ and } y \leq z)$,
\item $x R_{N_2} y\Leftrightarrow \exists z (x \leq z \mbox{ and } y \leq z \mbox{ and } z \notin Y_0)$, and
\item $Y_0 := \{x ~ \mid ~ \forall z (x \cancel{R}_{N_2} z) \}$.
\end{enumerate}
\begin{center}
\begin{tikzpicture}
    % set up the nodes
    \node (A1) at (0,0) {$(X,\leq,R_{N_1},Y_0)$};
		\node (B1) at (0,2) {$(X,\leq,Y_0)$};
		\node (B2) at (5,2) {$(X,\leq,R_{N_1},R_{N_2})$};
		\node (C1) at (5,0) {$(X,\leq,R_{N_2})$};
    % draw arrows and text between them
\draw[transform canvas={xshift=1ex},->] (A1) -- (B1) node[above, midway, xshift=-2.5cm] {};
\draw[transform canvas={xshift=-1ex},->] (B1) -- (A1) node[above] {};

\draw[transform canvas={xshift=1ex},->] (C1) -- (B2) node[above, midway, xshift=2.5cm, yshift=-0.3cm] {};
\draw[transform canvas={xshift=-1ex},->] (B2) -- (C1) node[above] {};

\draw[transform canvas={yshift=-0.7ex}, shorten <=2ex, shorten >=2ex, ->] (A1) -- (C1) node[above, midway,yshift=-0.2cm] {};
\draw[transform canvas={yshift=0.7ex}, shorten >=2ex, shorten <=2ex, ->] (C1) -- (A1) node[above, midway,yshift =0.2cm] {};

\draw[transform canvas={yshift=-0.7ex}, shorten <=2ex, shorten >=2ex, ->] (B1) -- (B2) node[above, midway,yshift=-0.2cm] {};
\draw[transform canvas={yshift=0.7ex}, shorten >=2ex, shorten <=2ex, ->] (B2) -- (B1) node[above, midway,yshift =0.2cm] {};
\end{tikzpicture}
\end{center}

%The discussion on relational semantics in this work for ILM are based on either the works of Do{\v{s}}en or Segerberg. For constructive logic with strong negation, the logic corresponding to the class of Nelson algebras, Kripke-style models are defined by \cite{YG} as follows. Since, a Nelson algebra involves two negations just as {\it ccpBa}s or $K_{im}$-algebras do, one may ask whether such models can be applied here. For example, in \cite{YG}, a Kripke model over Nelson logic (and its extensions) is defined as a quadruple $(M,\leq,\delta,\tau)$, where $(M,\leq)$ is a poset (set of stages), $\delta$ is a non-decreasing function associating a set of propositional constants with each $X \in M$ (set of objects involved at stage $X$), and $\tau_X$  is a valuation function satisfying certain conditions.
%The discussion on logics corresponding to {\it ccpBa}s and {c${\vee}$cpBa}s are restricted to propositional logics. For any propositional logic studied, extending the logic in predicate calculus is an important direction. Relational and algebraic semantics for predicate calculus can be found in \cite{vakarelov1977,vorobev1952}. Thus, the next step in the study of {\rm ILM} here can be to define predicate logic corresponding to the algebras.

Investigations of properties of negations have resulted in various schemes of logical systems. Some of the work in this direction may be found in \cite{Dosen1999,Johansson,sikorski1953,Segerberg1968} and more recently, in \cite{colacito2017,Dunn2005,odintsov2007,odintsov2008}. A common approach adopted is that a `base logic with negation' with minimum properties on negation is first defined, and new logics are obtained by adding axioms over the existing ones. This is followed by defining relational semantics for the base logic, and obtaining canonical properties for various negation properties. Odintsov \cite{odintsov2008} studied the class of extensions of minimal logic, and presented them in a diagram, just like Dunn's Kite. In the direction of logics with two negations, the class of extensions of Nelson logic (denoted $\mathbf{N4}^{\bot}$) and its semantics is also discussed by Odintsov \cite{odintsov2005,odintsov2008}. However, we have shown that negations in ILM differ from those in Nelson logic. Another relevant and independent work on logics with two negations is done in \cite{Dunn2005}. Taking a cue from the dual properties of negations \cite{shramko2005}, Dunn defined logics with two negations obtained by merging two minimal systems (dual to each other) \cite{Dunn2005}. Diagrammatically, it is represented by uniting the lopsided kite of negations with the dual lopsided kite of negations. The base negation is kept as preminimal (along with its dual) in the `United Kite'. The Kite of negation pair $(\sim,\neg)$ presented  in this work opens up a different direction of study. We have made the properties of one negation ($\sim$) of the pair vary, while the other ($\neg$)  is fixed   to be intuitionistic. $\neg\neg{\sim} \top \vdash {\sim}\top$ (DNE(${\sim}\top$)) connects both the negations. 
%The  properties of the second negation could be made to vary as well.  
This suggests a line of work where one may start with a base logic with a pair $({\sim},{\neg})$ of preminimal negations connected by an appropriate condition. A Kite of the negation pair $({\sim},{\neg})$ may  be developed, where one may  wish to ensure that the base logic is extendable to existing logics with two negations such as Nelson logic and its extensions and Dunn's logic with dual negations.

%In other words, we have a pair of preminimal negations, connected through DNE(${\sim} 1$). One may also wish to  ensure that the base logic in a  kite with a general negation pair is sufficiently minimal so that existing logics with two negations  such as Nelson logic and its extensions or Dunn's logic with dual negations, may be obtained by addition of axioms and rules on the base logic.
%Analogous to Odintsov's work, with suitable comparisons, class of logics with two negations can then be defined and their properties studied.

%\paragraph{Paragraph headings} Use paragraph headings as needed.
%\begin{equation}
%\end{equation}

%\begin{acknowledgements}
%\end{acknowledgements}


\bibliographystyle{abbrv}
\bibliography{myrefs}

%\begin{thebibliography}{}
%
%%
%% and use \bibitem to create references. Consult the Instructions
%% for authors for reference list style.
%\end{thebibliography}

\end{document}